\documentclass[12pt,a4paper]{article}
\usepackage[a4paper]{geometry}
\usepackage{amsthm,amssymb}
\usepackage{tikz}
\usepackage{lmodern}
\usepackage{hyperref}
\usepackage[T1]{fontenc}
\usepackage[utf8]{inputenc}
\hypersetup{
   colorlinks,
   citecolor=red,
   filecolor=black,
  linkcolor=blue,
   urlcolor=green
}

\usepackage[titles]{tocloft}

\newcommand{\x}{\mathbf{x}}
\newcommand{\ga}{\mathbf{a}}
\newcommand{\gb}{\mathbf{b}}
\newcommand{\gc}{\mathbf{c}}
\newcommand{\y}{\mathbf{y}}
\newcommand{\z}{\mathbf{z}}
\newcommand{\g}{\mathbf{g}}
\newcommand{\h}{\mathbf{h}}
\newcommand{\e}{\mathbf{e}}

\newcommand{\tr}{\mathbf{tr}_G}

\newcommand{\B}{\mathscr{B}}

\newcommand{\G}{\mathscr{G}}

\newcommand{\N}{\mathscr{N}}
\newcommand{\Ss}{\mathcal{S}}

\newcommand{\I}{\mathscr{I}}
\newcommand{\Q}{\mathcal{Q}}
\newcommand{\V}{\mathscr{V}}

\newcommand{\Y}{\mathscr{Y}}
\newcommand{\X}{\mathscr{X}}
\newcommand{\pr}[2]{G_{[{#1},{#2})}}

\newcommand{\prh}[2]{H_{[{#1},{#2})}}

\usepackage{amsmath}
\usepackage{hyperref}
\usepackage{mathrsfs}
\usepackage[T1]{fontenc}
\usepackage[utf8]{inputenc}
\usepackage{lmodern}

\newtheorem{thm}{Theorem}[section]
\newtheorem{cor}[thm]{Corollary}
\newtheorem{lem}[thm]{Lemma}
\newtheorem{question}[thm]{Question}
\newtheorem{rem}[thm]{Remark}
\newtheorem{prop}[thm]{Proposition}

\newtheorem{example}[thm]{Example}

\begin{document}

\title{\bf On the generic family of Cayley graphs\\ of a finite group\thanks{The research  was supported by the grant WZ/WI/1/2019 from Bia{\l}ystok University of Technology and funded from the resources for research by Ministry of Science and Higher Education of Poland.}}
\author{Czes{\l}aw Bagi\'{n}ski \qquad Piotr Grzeszczuk\\
\small Faculty of Computer Science\\[-0.8ex]
\small Bialystok University of Technology\\[-0.8ex]
\small Wiejska 45A Bialystok, 15-351 Poland\\
\small\tt \{c.baginski,p.grzeszczuk\}@pb.edu.pl}
\date{}

\maketitle
\begin{abstract}
 Let $G$ be a finite group. For each $m>1$ we define the symmetric canonical subset $S=S(m)$ of the Cartesian power $G^m$ and we consider the family of Cayley graphs $\mathscr{G}_m(G)=Cay(G^m,S)$. We describe properties of these graphs and show that for a fixed $m>1$ and groups $G$ and $H$ the graphs $\mathscr{G}_m(G)$ and $\mathscr{G}_m(H)$ are isomorphic if and only if the groups $G$ and $H$ are isomorphic. We describe also the groups of automorphisms $\mathbf{Aut}(\mathscr{G}_m(G))$. It is shown that if $G$ is a non-abelian group, then $\mathbf{Aut}(\mathscr{G}_m(G))\simeq \big(G^m \rtimes \mathbf{Aut}(G)\big)\rtimes D_{m+1}$, where $D_{m+1}$ is the dihedral group of order $2m+2$. If $G$ is an abelian group (with some exceptions for $m=3$), then $\mathbf{Aut}(\mathscr{G}_m(G))\simeq G^m\rtimes \big(\mathbf{Aut}(G)\times S_{m+1}\big)$, where $S_{m+1}$ is the symmetric group of degree $m+1$. As an example of application we discuss relations between  Cayley graphs $\mathscr{G}_m(G)$ and Bergman-Isaacs Theorem on rings with fixed-point-free group actions.
 \bigskip
 
\end{abstract}


\bigskip

\newcommand{\cftchapterdotsep}{\cftnodots}
\renewcommand{\cftchapterdotsep}{\cftdotsep}
{\footnotesize \tableofcontents}


\section{Introduction}

Let $G$ be a group with the identity element $e$. A subset $S$ of $G\setminus\{e\}$ is said to be symmetric if $S=S^{-1}$.
Recall that  the Cayley graph with respect to $S$, denoted by $Cay(G,S)$, is the graph whose vertex set is $G$ and
two vertices $g,h$ are adjacent (we denote this by $g\sim h$) if $h=s\cdot g$ for some $s\in S$. With any group $G$ we associate the family of graphs $\G_m(G)=Cay(G^m,\Ss)$  ($m=2,3,\dots $), where $\Ss=\Ss(m)$ is the symmetric subset canonically determined by $G$ and $m$.  There are many results concerning isomorphisms of Cayley graphs, in the literature. Mostly they concern the question when two Cayley graphs on a given group (depending on the set $S$) are isomorphic. In this paper we consider a different question related to the isomorphism problem. 
Namely, we show that any member of our family of graphs determines the group $G$, in the sense that for the given integer $m>1$ and groups $G$, $H$ if the graphs $\G_m(G)$ and $\G_m(H)$ are isomorphic, then $G$ and $H$ are isomorphic as well. It is interesting that our construction appeared as a result of investigation of some ring theory problems. Some structural invariants of the graphs $\G_m(G)$
appears as important invariants of rings -- for details see next section. We are convinced that investigation of the graphs $\G_m(G)$ is also of independent interest, as their combinatorial properties determine structure of the group $G$.   

The definition of the the graph $\G_m(G)$ is based on a description of the set $\Ss$.
For $x\in G^{\times}=G\setminus\{e\}$ and $1\leqslant k< l\leqslant m+1$, we denote by $\x_{[k,l)}$  the element
$$
(\underbrace{e,e,\dots,e}_{k-1 \  \  {\rm times }},\underbrace{x,x,\dots,x}_{l-k \  \ {\rm  times }},e,e,\dots,e)\  
$$
of $G^m$. By $\pr{k}{l}$ we denote the set of all elements $\x_{[k,l)}$, where $x\in G^{\times}$
and call it {\em  an interval}. The symmetric set $\Ss$ is the union of all intervals:
$$
\Ss=\bigcup_{1\leqslant k<l\leqslant m+1}\pr{k}{l}.
$$
Thus if $\g=(g_1,\dots,g_m)$, $\h=(h_1,\dots,h_m)$ are two vertices of $\G_m(G)$, then
$$
\g \sim \h   \  \text{ iff }\    \h=\x_{[k,l)}\cdot\g \  \text{ for some }\  x\in G^\times \      
\text{ and }\ 1\leqslant k<l\leqslant m+1.
$$
It is easy to see that the graph $\G_m(G)$ has $|G|^m$ vertices and obviously is $d$-regular,  where $d=\binom{m+1}{2}(|G|-1)$.
\medskip

We begin by providing an overview of the main results in this paper. In Section 2 we present ring theoretical motivations for considering graphs $\G_m(G)$. In Sections 3 and 4 we collect elementary properties of the family $\{\G_m(G)\mid m\geqslant 2\}$. In particular, we prove  that for any group $G$ the graph $\G_2(G)$ is strongly regular and for $m>2$ the graphs $\G_m(G)$ are edge regular.
Recall that a maximum clique of a graph is a clique, such that there is no clique with more vertices. The clique number of the graph $\G_m(G)$, that is the number of vertices in a maximum clique (Proposition 4.7), is equal 
	$$\left\{\begin{array}{cl}
	{\rm max}\{m+1, |G|\} & {\rm if}\ \ (m,|G|)\neq (2,2)\\
	4 & {\rm if}\ \ (m,|G|)= (2,2).\\
	\end{array}\right.
	$$

In  Section 4 we distinguish two types of cliques of graphs $\G_m(G)$: {\em interval cliques}, whose vertices belong to the one interval and {\em dispersed cliques}, whose vertices belong to different intervals. We prove (Corollary 4.9) that any automorphism of $\G_m(G)$ preserves the type of a maximum clique, with one exception - the graph $\G_2(C_3)$. This fact will play a key role in determining the group of automorphisms $\mathbf{Aut}(\G_m(G))$.
\medskip

In the next preparatory Section 5, we introduce the concept of a homogeneous homomorphism between graphs
$\G_m(G)$ and $\G_m(H)$, as a graph homomorphism preserving intervals. The main result states

\medskip

\noindent{\bf Theorem 5.2.} {\em Let $G$, $H$ be groups and $m>1$. Then every homogeneous graph homomorphism (isomorphism) $F\colon \G_m(G)\to \G_m(H)$ is induced by a group monomorphism (isomorphism), that is
 $$
 F(g_1,g_2,\dots,g_m)=(f(g_1),f(g_2),\dots,f(g_m))
 $$
 for some monomorphism (isomorphism) of groups $f\colon G\to H$.}
 
 \medskip

In the last main Section 6, we describe the groups of automorphisms of the graphs $\G_m(G)$. The final description depends only on whether the group $G$ is abelian or non-abelian. We prove 

\medskip

\noindent{\bf Theorems 6.3 and 6.10.} {\em Let $G$ be a non-trivial group with the automorphism group $\mathbf{Aut}(G)$ and let $m>1.$
\begin{enumerate}
\item If $G$ is abelian and either
\begin{enumerate}
\item[{\rm (a)}] $m>3$, or
\item[{\rm (b)}] $m=3$ and $G$ is of exponent bigger than $2$, or
\item[{\rm (c)}] $m=2$ and $|G|\neq 3$,
\end{enumerate} 
then
$$
\mathbf{Aut}(\G_m(G))\simeq G^m\rtimes \big(\mathbf{Aut}(G)\times S_{m+1}\big),
$$
where $S_{m+1}$ is the symmetric group of degree $m+1$.
\item If $G$ is non-abelian, then
$$
\mathbf{Aut}(\G_m(G))\simeq \big(G^m \rtimes \mathbf{Aut}(G)\big)\rtimes D_{m+1},
$$
where $D_{m+1}$ is the dihedral group of order $2m+2$.
\end{enumerate}}

We then combine all of the partial results from Sections 4, 5 and 6 to
obtain the main result of our paper.

\medskip

\noindent{\bf Theorem 6.12.} {\em Let $G$ and $H$ be groups and $m>1$. Then the graphs $\G_m(G)$ and $\G_m(H)$ are isomorphic if and only if the groups $G$ and $H$ are isomorphic.}

\medskip

Throughout the paper we consider only finite groups. Our notation is mainly standard. We use exponential notation for automorphisms, that is $x^f$ means the image of $x$ under the action of an automorphism $f$. The exception is Section 4, where we use classical notation $f(x)$. We use, beside some special cases, the script font for denoting a graph and the standard font for denoting its vertex set, for instance, if $\mathscr{X}$ is a graph, then $X$ means its vertex set. 


\section{Motivation}

Let a finite group $G$ acts on a non-commutative  ring $R$ so that we have a group homomorphism $G\to \mathbf{Aut}(R)$, $r\mapsto r^g$. Then we can form the fixed subring
$$
R^G=\{r\in R \mid r^g=r \   \   {\rm  for\ all }\  \  g\in G\}.
$$
A natural way to construct fixed points of the action is to use the trace map $\tr\colon R\to R$ defined by $\tr(r)=\sum\limits_{g\in G}r^g$. The image $T=\tr(R)$ is an ideal of $R^G$.
One of the most fundamental results in the theory of fixed rings is the following  theorem of G.M. Bergman and I.M. Isaacs \cite{BI}.
\bigskip

\noindent {\bf Theorem.} {\em
Let $G$ be a finite group of automorphisms of the ring $R$ with no additive $|G|$-torsion. If $\tr(R)$ is nilpotent
of index $d$, then $R$ is nilpotent of index at most $f(|G|)^d$, where 
$f(m)=\prod\limits_{k=1}^{m}\left(\binom{m}{k}+1\right)$. 
In particular if $\tr(R)=0$, then $R^{f(|G|)}=0$.}

\bigskip

Bergman-Isaacs theorem is extremely useful and has been the basic tool in theory of finite group actions on non-commutative rings for a long time.   If the acting group $G$ is solvable it is known that the best possible nilpotence bound is $|G|$ (see \cite{BI},\cite{P1}). There were other proofs of Bergman-Isaacs theorem (cf. \cite{P}, \cite{Q}), but none of them  yield better information on the bound. 

It appears that working with the generic model defined by D.S. Passman in \cite{P} (cf. \cite{P1}, chapter 6) one can easily reduce the problem of nilpotence bound to description of properties of  the graph $\G_m(G)$ for sufficiently large $m$.

\medskip 

For a given finite group $G$ and some index set  $I$  let 
$$
Z_G=\mathbb{Z}\langle\zeta_{i,g}\mid g\in G, i\in I\rangle \  \  \ {\rm and }\  \  \    Q_G=\mathbb{Q}\langle\zeta_{i,g}\mid g\in G, i\in I\rangle
$$ 
be the free algebras without $1$ over the ring of integers $\mathbb{Z}$ and the field $\mathbb{Q}$, respectively. The group $G$ acts naturally on the left side on $Z_G$  and $Q_G$  permuting the variables according to the formula  $(\zeta_{i,g})^x=\zeta_{i,x^{-1}g}$.  The algebra $Z_G$ has a nice universal property
saying that  for any ring $A$ acted upon by $G$  and for given elements $a_i\in A$ ($i\in I$) the map
$$
\theta\colon \zeta_{i,g}\mapsto a_i^g
$$
extends to a $G$-homomorphism of rings $\theta\colon Z_G\to A$. Furthermore, for a sufficiently large set $I$ this map can be made a surjection and in this case $\theta(\tr(Z_G))=\tr(A)$.

For any positive integer $m$ let $Q_G(m)$ be the linear span over $\mathbb{Q}$ of all monomials in $\zeta_{i,g}$ of degree $m$. For monomials $a=\zeta_{i_1,g_1}\dots\zeta_{i_{k-1},g_{k-1}}$, $b=\zeta_{i_k,g_k}\dots\zeta_{i_{l-1},g_{l-1}}$,
$c=\zeta_{i_l,g_l}\dots\zeta_{i_{m},g_{m}}$ we have

\begin{equation}\label{T_G}
a\tr(b)c=\sum_{h\in G}\zeta_{i_1,g_1}\dots\zeta_{i_{k-1},g_{k-1}}\zeta_{i_k,hg_k}\dots\zeta_{i_{l-1},hg_{l-1}}\zeta_{i_l,g_l}\dots\zeta_{i_{m},g_{m}}.
\end{equation}

Let $\mathscr{T}$ be the set of all elements $a\tr(b)c$, where $a,b,c$ are monomials in $\zeta_{i,g}$ such that $\deg(b)\geqslant 1$ and $\deg(a)+\deg(b)+\deg(c)=m$, and let
$T_G(m)$ be the linear span of $\mathscr{T}$ over $\mathbb{Q}$. We call $T_G(m)$ a {\em trace subspace}. It is clear that $T_G(m)$ is a subspace of $Q_G(m)$. In the context of Bergman-Isaacs theorem we are asking the following

\bigskip
\begin{question}\label{Qn}
Is there a positive integer $m=m(G)$ such that $T_G(m)=Q_G(m)$?
\end{question}
\bigskip

To explain the connection of the question with the theorem consider a ring $R$ satisfying its assumptions. Let $\theta\colon Z_G\to R$ be an epimorphism such that $\theta(\tr(Z_G))=\tr(R)$.
Now, if we have $m$ giving positive answer to the question, we obtain that all
$\theta(\zeta_{i_1,g_1}\dots\zeta_{i_{m},g_{m}})$ belong to the ideal of $R$ generated by $\tr(R)$. Then a product of any $m$ elements of $R$ belongs to this ideal, so $m$ is a bound for the nilpotency index of $R$ modulo the ideal generated by $\tr(R)$, and in particular a bound for the nilpotency index of $R$ in the case $\tr(R)=0$. 

If $G=\{e,g\}$ is cyclic of order $2$, the question is easy and the answer is $m=2$. The identities like
$$
\begin{array}{rl}
2\zeta_{1, e}\zeta_{2,e }=
\zeta_{1,e}\tr(\zeta_{2,e})+\tr(\zeta_{1,e})\zeta_{2,e}-\tr(\zeta_{1,e}\zeta_{2,g}) 
\end{array}
$$
show that $T_G(2)=Q_G(2)$.

We will now reduce Question \ref{Qn} to some questions concernig properties of the graph $\G_m(G)$. 
For a fixed sequence $\mathbf{i}=(i_1,i_2,\dots, i_m)$  of elements of $I$ (not necessarily different) let 
$$
\Omega_\mathbf{i}=\{\zeta_{i_1,g_1}\zeta_{i_2,g_2}\dots \zeta_{i_m,g_m}\mid g_1,g_2,\dots,g_m\in G\}.
$$
It is clear that $|\Omega_\mathbf{i}|=|G|^m$. Notice that each monomial $\omega \in \Omega_\mathbf{i}$ determines in a natural way
$\binom{m+1}{2}$ elements of $T_G(m)$. Indeed, each partition 
$$
\{1,2,\dots,m\}=\{1,2,\dots,k-1\}\cup\{k,k+1,\dots,l-1\}\cup
\{l,l+1\dots,m\},
$$ 
where $1\leqslant k<l\leqslant m+1$ determines the element $a\tr(b)c$, where $a=\zeta_{i_1,g_1}\dots\zeta_{i_{k-1},g_{k-1}}$, $b=\zeta_{i_k,g_k}\dots\zeta_{i_{l-1},g_{l-1}}$, $c=\zeta_{i_l,g_l}\dots\zeta_{i_{m},g_{i_{m}}}$ and  $\omega=abc$. 

We will interpret the identity 
\begin{equation}\label{Eq}
\sum_{h\in G}\zeta_{i_1,g_1}\dots\zeta_{i_{k-1},g_{k-1}}\zeta_{i_k,hg_k}\dots\zeta_{i_{l-1},hg_{l-1}}\zeta_{i_l,g_l}\dots\zeta_{i_{m},g_{m}}=a\tr(b)c
\end{equation}
as a linear equation in the set of variables $\Omega_\mathbf{i}$. Our aim is to express any $\omega\in \Omega_\mathbf{i}$ using the elements of $T_G(m)$.

Since  $a\tr(b^x)c= a\tr(b)c$ for $x\in G$, the set $\Omega_\mathbf{i}$ determines the system of $\binom{m+1}{2}|G|^{m-1}$ linear equations in $|G|^m$ variables $\omega\in\Omega_\mathbf{i}$. Let $\mathbf{B}$ be the  matrix of this system (with respect to a fixed order of elements of $\Omega_\mathbf{i}$ and equations (\ref{Eq})). Clearly each entry of 
$\mathbf{B}$ is either $0$ or $1$. Furthermore, each its row has exactly $|G|$ entries equal to $1$, and in each of its column the number $1$ appears exactly $\binom{m+1}{2}$-times.
 Thus
\begin{equation}\label{Eq2}
\mathbf{B}^T\mathbf{B}=\binom{m+1}{2}\mathbf{I}+\mathbf{A},
\end{equation}
where $\mathbf{A}$ is a symmetric $|G|^m\times|G|^m$ matrix. 

We will see that $\mathbf{A}$ is the adjacency matrix of the Cayley graph $\G_m(G)$. 
It is clear that there is one  to one correspondence
$$
\zeta_{i_1,g_1}\zeta_{i_2,g_2}\dots \zeta_{i_m,g_m}\mapsto (g_1,g_2,\dots,g_m)
$$
between elements of $\Omega_\mathbf{i}$ and $G^m$. 
Let $\omega=\zeta_{i_1,g_1}\zeta_{i_2,g_2}\dots \zeta_{i_m,g_m}$ and $\omega^\prime=\zeta_{i_1,h_1}\zeta_{i_2,h_2}\dots \zeta_{i_m,h_m}$. 
Clearly in the matrix $\mathbf{B}^T\mathbf{B}$ the $(\omega,\omega^\prime)$-position   is equal to the product of  the $\omega$-th column by the $\omega^\prime$-th column of $\textbf{B}$. Notice that the $\omega$-th column of $\mathbf{B}$ has entries equal to $1$ precisely in the rows corresponding to equations (\ref{Eq}). On the other hand two different monomials $\omega$ and $\omega^\prime$
can appear simultaneously in at most one equation of the form (\ref{Eq}). Therefore the product of the $\omega$-th column and the $\omega^\prime$-th column is $1$ only in the case when $\omega$ and $\omega^\prime$ appear in the same equation, so when $\g\sim \h$ in the graph $\G_m(G)$. It means that $\mathbf{A}$ is the adjacency matrix of the graph $\G_m(G)$.

By the Cauchy-Binet formula $\det(\mathbf{B}^T\mathbf{B})=\sum\limits_i(\det \mathbf{B}_i)^2$, where the sum runs over all 
$|G|^m\times|G|^m$ submatrices of $\mathbf{B}$. Thus if $\det(\mathbf{B}^T\mathbf{B})\neq 0$,  $\det\mathbf{B}_i
\neq 0$ for some submatrix $\mathbf{B}_i$. In this case the system of equations (\ref{Eq}) has a unique solution, and therefore $T_G(m)=R_G(m)$. Let $\lambda_{\rm min}$ be the smallest eigenvalue of $\mathbf{A}$. Since the matrix $\mathbf{B}^T\mathbf{B}$ is positive semi-definite, by formula (\ref{Eq2}), it follows that $$\lambda_{\rm min}\geqslant -\binom{m+1}{2}.$$ Therefore  $\det(\mathbf{B}^T\mathbf{B})\neq 0$ if and only if $\lambda_{\rm min}>-\binom{m+1}{2}$.

The above discussion can be summarized as follows. 

\begin{cor}\label{ineq}
If the smallest eigenvalue  $\lambda_{\rm min}$ of the adjacency matrix of the graph $\G_m(G)$ satisfies the inequality 
$$\lambda_{\rm min}>-\binom{m+1}{2},$$
then $T_G(m)=Q_G(m)$.

\end{cor}

In general, there is no simple explicit formula for computing eigenvalues of a Cayley graph. When the symmetric set $S\subset G^\times$ is normal (that is $s^g\in S$ for all $s\in S$, $g\in G$) the spectrum of  $Cay(G,S)$ can be computed explicitly in terms of the
complex character values (see \cite{Z}). Namely, if $\mathbf{Irr}(G)=\{\chi_1,\dots,\chi_t\}$ is the set of all irreducible characters of $G$, then for $j=1,\dots,t$
\begin{equation}\label{eigenvalue}
\lambda_j=\frac{1}{\chi_j(e)}\sum_{s\in S}\chi_j(s)
\end{equation}
are all eigenvalues of $Cay(G,S)$. Moreover the multiplicity of $\lambda_j$ is equal to $\sum\limits_{k, \lambda_k=\lambda_j} \chi_k(e)^2$. 

But the symmetric set $\mathcal{S}$ of  $\G_m(G)$ is normal  only in the case when the group $G$ is abelian. Thus the above formulas for $\lambda _j$ can be applied to abelian groups only. It is well known (see \cite{I}, Theorem 4.21) that if $\mathbf{Irr}(G)=\{\chi_1,\dots,\chi_s\}$, then the set of all irreducible characters of $G^m$ consists of the characters $\chi=\chi_{i_1}\times\dots\times\chi_{i_m}$ defined as:
$$
\chi(g_1,\dots,g_m)=\chi_{i_1}(g_1)\cdot ...\cdot\chi_{i_m}(g_m),
$$
where $\chi_{i_1},\dots,\chi_{i_m}\in \mathbf{Irr}(G).$  Furthermore, if $G$ is abelian, then irreducible characters are linear and  form the group $\widehat{G}=\mathbf{Hom}(G,\mathbb{C}^*)\simeq G$. Let $\chi_1=\mathbf{1}$ be the identity of $\widehat{G}$. The orthogonality relations for characters yield that if $\chi_{j_1},\dots,\chi_{j_k}\in \widehat{G}$, then
\begin{equation}\label{sum_a}
\sum_{g\in G^\times}\chi_{j_1}(g)\dots\chi_{j_k}(g)=\left\{\begin{array}{cl}
|G|-1\  \  &\  \    {\rm if}\  \   \chi_{j_1}\cdot ...\cdot\chi_{j_k}=\mathbf{1} \\
-1\  \  &\  \   {\rm if}\  \   \chi_{j_1}\cdot ...\cdot\chi_{j_k}\neq\mathbf{1}
\end{array}\right.
\end{equation}

\begin{prop}
Let $G$ be a finite abelian group.  If $m\geqslant |G|$, then each eigenvalue $\lambda$ of $\G_m(G)$ satisfies
the inequality
$$
\lambda >-\binom{m+1}{2}
$$
Therefore $T_G(m)=Q_G(m)$ for $m\geqslant |G|$.
\end{prop}

\begin{proof}
Take $\chi=\chi_{i_1}\times\dots\times\chi_{i_m} \in \mathbf{Irr}(G^m)$. Since $G$ is abelian,  
$\chi(\e)=\chi_{i_1}(e)\dots\chi_{i_m}(e)=1$. 
According to (\ref{eigenvalue})  and (\ref{sum_a}) the eigenvalue 
$\lambda$  corresponding  to $\chi$  is equal:
$$
\begin{array}{ll}
\lambda &= \sum\limits_{1\leqslant k<l\leqslant m+1}\sum\limits_{g\in G^\times}\chi_{i_1}(e)\cdot ...\cdot\chi_{i_{k-1}}(e)
\chi_{i_k}(g)\cdot ...\cdot\chi_{i_{l-1}}(g)\chi_{i_l}(e)\cdot ...\cdot\chi_{i_m}(e)\\  \\
& = \sum\limits_{1\leqslant k<l\leqslant m+1}\sum\limits_{g\in G^\times}
\chi_{i_k}(g)\cdot ...\cdot\chi_{i_{l-1}}(g) =-\binom{m+1}{2}+n_\chi|G|,
\end{array}
$$
where $n_\chi$ is the number of pairs $(k,l)$ such that $1\leqslant k<l\leqslant m+1$ and
$\chi_{i_k}\cdot ...\cdot\chi_{i_{l-1}}=\mathbf{1}$. If $m\geqslant |G|$, then the pigeonhole principle implies that some elements of the sequence 
$$
\chi_{i_1},\   \chi_{i_1}\chi_{i_2},\  \dots, \chi_{i_1}\chi_{i_2}\cdot ...\cdot\chi_{i_m}
$$
must be equal, that is $\chi_{i_k}\cdot \  ...\  \cdot\chi_{i_{l-1}}=\mathbf{1}$ for some $k<l$. Therefore $n_\chi>0$, and hence $\lambda >-\binom{m+1}{2}.$ 
\end{proof}

It is worth emphasizing that the expression of monomials from $Q_G(m)$ as a linear combination of {\em trace monomials} from $T_G(m)$ is not immediate even for small groups. This is already seen for the cyclic group of order 3.

\begin{example}{\rm  Let  $G=\{e,g, g^2\}$ be the group of order $3$.  Inverting a suitable $27\times 27$ submatrix $\mathbf{B}_i$ of the matrix of the system of equations (\ref{Eq}) we obtain the identity}
{\footnotesize
$$
\begin{array}{lllllll}
\medskip
9\cdot\zeta_{1,e}\zeta_{2,e}\zeta_{3,e} 
&=& 3\cdot\zeta_{1,e}\zeta_{2,e}\tr[\zeta_{3,e}]
&+& 5\cdot\zeta_{1,e}\tr[\zeta_{2,e}]\zeta_{3,e} 
&-& 3\cdot\zeta_{1,e}\tr[\zeta_{2,e}\zeta_{3,e}]\\ \medskip

&+& 4\cdot\tr[\zeta_{1,e}\zeta_{2,e}\zeta_{3,e}]
&-& 5\cdot\zeta_{1,e}\tr[\zeta_{2,e }\zeta_{3,g }]
&+& 4\cdot\tr[\zeta_{1,e}]\zeta_{2,e }\zeta_{3,g}\\ \medskip

&-& 2\cdot\tr[\zeta_{1,e}\zeta_{2,e}\zeta_{3,g}]
&+& \  \tr[\zeta_{1,e}]\zeta_{2,e}\zeta_{3,g^2}
&-& 4\cdot\tr[\zeta_{1,e}\zeta_{2, e}\zeta_{3,g^2 }] \\ \medskip

&-& 5\cdot\tr[\zeta_{1,e}\zeta_{2, g}]\zeta_{3,e}
&+& 3\cdot\tr[\zeta_{1,e}]\zeta_{2,g }\zeta_{3,g }
&+& 2\cdot\tr[\zeta_{1,e}]\zeta_{2,g}\zeta_{3,g^2}\\ \medskip

&+& 3\cdot\tr[\zeta_{1,e}\zeta_{2, g}]\zeta_{3,g^2}
&+& 5\cdot\tr[\zeta_{1,e}]\zeta_{2,g^2}\zeta_{3,g} 
&-& 5\cdot \tr[\zeta_{1,e}\zeta_{2,g^2}\zeta_{3,g}]\\ \medskip

&+& 5\cdot\tr[\zeta_{1,e}\zeta_{2,g^2}]\zeta_{3,g^2}
&-& 2\cdot\tr[\zeta_{1,e}\zeta_{2,g^2}\zeta_{3,g^2}]
&-& 3\cdot\zeta_{1,g}\tr[\zeta_{2, e}\zeta_{3,e}] \\ \medskip

&+& 5\cdot\zeta_{1,g}\tr[\zeta_{2,e }]\zeta_{3,e}  
&-& 4\cdot\zeta_{1,g}\tr[\zeta_{2,e}]\zeta_{3,g} 
&-& \  \zeta_{1,g}\tr[\zeta_{2,e}\zeta_{3,g^2}]\\ \medskip

&+& 5\cdot\zeta_{1,g^2}\tr[\zeta_{2,e}]\zeta_{3,e}
&-& \  \zeta_{1,g^2}\tr[\zeta_{2,e}]\zeta_{3,g}
&-& 3\cdot\zeta_{1,g^2}\tr[\zeta_{2,e}\zeta_{3,g}] \\  \medskip

&-& 4\cdot \zeta_{1,g^2}\tr[\zeta_{2,e}]\zeta_{3,g^2}
\end{array}
$$}
\end{example}

Unfortunately, the general case goes beyond this scheme. So we state the following

\medskip

\begin{question} Whether the smallest eigenvalue  $\lambda_{\rm min}$ of the adjacency matrix of the graph $\G_m(G)$ satisfies the inequality 
$$
\lambda_{\rm min}>-\binom{m+1}{2}
$$
for any group $G$ and $m\geqslant |G|$?
\end{question}


\section{Elementary properties}

In this section we collect some basic properties of the graphs $\G_m(G)$. We begin with the following easy observation.

\begin{lem} Let $G$ be a group and $m>1$. For any $\g\in G^m\setminus\{\e\}$ there exists a sequence $1\leqslant i_1<i_2<\dots<i_{k} < i_{k+1}\leqslant m+1$ and elements $x_1,\dots,x_k$ in $G$ such that
\begin{enumerate}
\item $x_1\neq e \neq x_k$;
\item $x_i\neq x_{i+1}$ for $i=1,2,\dots,k-1$;
\item $\g={\x_1}_{[i_1,i_2)} {\x_2}_{[i_2,i_3)}\dots {\x_k}_{[i_k,i_{k+1})}$.
\end{enumerate}
\end{lem}

It is clear that the decomposition in point (3) is unique for $\g$. The number $\vartheta(\g)=k$  we call the {\em weight} of $\g$ and the decomposition we call the {\em weight decomposition} of $\g$. It is clear that   $1\leqslant \vartheta(\g)\leqslant m$  for any $\g\neq\e$.

\begin{example}{\rm Let $a,b,c$ be different elements of $G^\times$ and let 
$$
\g=(e,a,a,b,b,b,c,e,e),\   \h=(e,a,e,e,b,b,c,c,e).
$$
Then we have the following weight decompositions:
$$ 
\g =\ga_{[2,4)}\gb_{[4,7)}\gc_{[7,8)}   \   {\rm  and }\  \vartheta(\g)=3
$$
$$
\h=\ga_{[2,3)}\e_{[3,5)}\gb_{[5,7)}\gc_{[7,9)}  \   {\rm  and }\  \vartheta(\h)=4.
$$}
\end{example}

The following lemma follows immediately from the 
definition of the set $\Ss$.

\begin{lem}\label{SS_in_S}
Let $\g_{[k,l)}$ be a fixed element of $\Ss$, $g\in G^{\times}$, $1\leqslant k<l\leqslant m+1$. Then for $h\in G^{\times}$ and $1\leqslant i<j\leqslant m+1$, 
$$\h_{[i,j)}\g_{[k,l)}\in \Ss$$
if and only if one of the following conditions is satisfied:\smallskip\par
\noindent
$
\begin{array}{llll} 
{\rm 1.} & h\neq g^{-1} & {\rm and} & (i=k,\  j=l); \\
{\rm 2.} & h=g & {\rm and} & (j=k \ \ {\rm or}\ \ i=l); \\
{\rm 3.} & h=g^{-1} &{\rm and} & (i\neq k,\ j=l\ \  {\rm or}\ \ i=k,\ j\neq l). 
\end{array}
$
\end{lem}

\begin{lem}\label{weight}
If $\g,\h\in G^m\setminus\{\e\}$ are adjacent in $\G_m(G)$, then
$|\vartheta(\g)-\vartheta(\h)|\leqslant 2.$
\end{lem}
\begin{proof}
Suppose that $\g={\x_1}_{[i_1,i_2)} {\x_2}_{[i_2,i_3)}\dots {\x_k}_{[i_k,i_{k+1})}$. Since $\g$ and $\h$ are adjacent, there exists $y\in G^\times$ and $1\leqslant l<s\leqslant m+1$ such that 
$$
\h=\y_{[l,s)}\g=\y_{[l,s)}{\x_1}_{[i_1,i_2)} {\x_2}_{[i_2,i_3)}\dots {\x_k}_{[i_k,i_{k+1})}.
$$
Take $r\leqslant t$ such that $ i_r\leqslant l< i_{r+1}$ and $i_t\leqslant s< i_{t+1}$. Then we can write a formal  weight decomposition of $\h$ 
$$
\h={\x_1}_{[i_1,i_2)} \dots {\x_{r-1}}_{[i_{r-1},i_{r})}{\x_{r}}_{[i_{r},l)}{(\y\x_r)}_{[l,i_{r+1})}\dots{(\y\x_t)}_{[i_t,s)}{\x_t}_{[s,i_{t+1})}\dots {\x_k}_{[i_k,i_{k+1})}.
$$
The factors 
$$
{\x_1}_{[i_1,i_2)},\ \ldots,\ {\x_{r-2}}_{[i_{r-2},i_{r-1})},\ {(\y\x_{r+1})}_{[i_{r+1},i_{r+2})},\ \ldots ,\ {(\y\x_{t-1})}_{[i_{t-1},i_t)},\ {\x_{t+2}}_{[i_{t+2},i_{t+3})},\ \ldots,\ {\x_k}_{[i_k,i_{k+1})}$$ 
are the weight components of $\h$. If $i_r < l$, then 
$$
{\x_{r-1}}_{[i_{r-1},i_{r})},\ {\x_{r}}_{[i_{r},l)}\ \ {\rm and}\ \ {(\y\x_r)}_{[l,i_{r+1})}
$$ 
are also such components. In a sense the weight component ${\x_{r}}_{[i_{r},i_{r+1})}$ of $\g$ creates two weight components of $\h$. If $i_r=l$, then the factor ${\x_{r}}_{[i_{r},l)}$ disappear. Moreover if $x_{r-1}\neq yx_r$, then ${\x_{r-1}}_{[i_{r-1},i_{r})}$ and ${(\y\x_r)}_{[l,i_{r+1})}$ are 
weight components of $\h$, if $x_{r-1}=yx_r$, then both elements create one weight component equal to 
${\x_{r-1}}_{[i_{r-1},i_{r+1})}$. Similar situation is created by relations between $s$ and $i_t$.
Therefore the numbers of factors in weight decompositions of $\g$ and $\h$ may differ at most by $2$.
\end{proof}

For $\g\in G^m$ let $V(\g)$ be the set of all neighbours of $\g$ in the graph $\G_m(G)$.

\begin{prop} \label{neighbours} Let $G$ be a group and $m>1$. Then for any element $\g\in G^m\setminus\{\e\}$:
\begin{enumerate}
\item  $\g\in V(\e)$ if and only if $\vartheta(\g)=1$. In this case $|V(\e)\cap V(\g)|=|G|+2m-4$;
\item  if $\vartheta(\g)=2$, then $|V(\e)\cap V(\g)|=6$;
\item  if $\vartheta(\g)=3$, then $|V(\e)\cap V(\g)|$ belongs to the set $\{0,1,2,4,6\}$. More precisely \medskip\par
\noindent
{\footnotesize 
$$
\begin{array}{|c|c|c|}
\hline
&& \\[-6pt]
 |V(\e)\cap V(\g)|  & \g &  {\rm Conditions}\\[-4pt] 
 && \\[-6pt]
 \hline 
&& \\[-10pt]
6 & \x_{[i_1,i_2)}\e_{[i_2,i_3)}\x_{[i_3,i_4)} &  o(x)=2 \\[4pt]
\hline 
&& \\[-10pt]
4 & \x_{[i_1,i_2)}\e_{[i_2,i_3)}\x_{[i_3,i_4)} & o(x)>2 \\[4pt]
& \x_{[i_1,i_2)}\x_{[i_2,i_3)}^2\x_{[i_3,i_4)} & \\[4pt]
& \x_{[i_1,i_2)}\e_{[i_2,i_3)}\x_{[i_3,i_4)}^{-1} & \\[4pt]
\hline 
&& \\[-10pt]
2 & \x_{[i_1,i_2)}\y_{[i_2,i_3)}\z_{[i_3,i_4)} & y\neq e,\ xy=yx,\ z=x^{-1}y\ \\[4pt]
  & \x_{[i_1,i_2)}\e_{[i_2,i_3)}\z_{[i_3,i_4)} & x\neq z\neq x^{-1} \\[4pt]
  & \x_{[i_1,i_2)}\y_{[i_2,i_3)}\x_{[i_3,i_4)} & x^2\neq y\neq e  \\[4pt]
\hline 
&& \\[-10pt]
1 & \x_{[i_1,i_2)}\y_{[i_2,i_3)}\z_{[i_3,i_4)} &  xy\neq yx\ {\rm and\ either}  \\[4pt]
  &                                            &   z=x^{-1}y\ {\rm or}\ z=yx^{-1}\\[4pt]
\hline 
&& \\[-10pt]
0 & \x_{[i_1,i_2)}\y_{[i_2,i_3)}\z_{[i_3,i_4)} & x\neq z,\ y\neq e,\ z\not\in\{x^{-1}y,yx^{-1}\} \\[4pt]
\hline
\end{array}
$$}\vspace{3pt}

\item if $\vartheta(\g)\geqslant 4$, then $V(\e)\cap V(\g)=\emptyset$.
\end{enumerate}
\end{prop}
\begin{proof}
(1) It follows from the definition of the graph $\G_m(G)$ that $\g\in V(\e)$ if and only if $\g\in \Ss$ if and only if $\vartheta(\g)=1$. For the second part take $\g=\g_{[k,l)}$. Then $\x\in V(\e)\cap V(\g)$ if and only if $\x=\h_{[i,j)}\g_{[k,l)}\in \Ss$. All possible values for $\h_{[i,j)}$ according to conditions listed in Lemma \ref{SS_in_S} are equal \smallskip\par
\noindent$
\begin{array}{ll} 
\bullet & |G|-2; \\
\bullet & (k-1) + (m+1-l); \\
\bullet & (l-2) + (m-k) 
\end{array}
$\smallskip\par
\noindent respectively. Since all these values of $\h_{[i,j)}$ define different $\x$, we have 
$$|G|-2+(k-1)+(m+1-l)+(l-2)+(m-k)= |G|+2m-4$$
common neighbours of $\e$ and $\g$. 

\medskip
Before proving next parts of the proposition let us make a useful observation: 

\medskip

\noindent{\it Suppose that $1\leqslant i_1 < \cdots <i_{k+1}\leqslant m+1$, $1\leqslant i < j \leqslant m+1$,  $\g={\x_1}_{[i_1,i_2)}, {\x_2}_{[i_2,i_3)}\dots {\x_k}_{[i_k,i_{k+1})}$,   and $\h = \h_{[i,j)}\in \Ss$.  Then}
\begin{equation}\label{useful}
\vartheta(\h\g)<\vartheta(\g)\Rightarrow \{i,j\}\subseteq \{i_1,\ldots,i_{k+1}\}	
\end{equation}
\medskip

\noindent (2) If $\vartheta(\g)=2$, then $\g=\x_{[k,l)}\y_{[l,s)}$, where $e\neq x\neq y\neq e$, and $1\leqslant k<l<s\leqslant m+1$. Hence if $\h_{[i,j)}\in \Ss$ and $\h\g\sim \e$, then by (\ref{useful}) 
\begin{itemize}
	\item $(i,j) = (k,l)$ and $h=x^{-1}$ or $hx=y$ that is $h=yx^{-1}$ or 
	\item $(i,j) = (l,s)$ and $h=y^{-1}$ or $hy=x$ that is $h=xy^{-1}$ or 
	\item $(i,j) = (k,s)$ and $h=x^{-1}$ or $h=y^{-1}$. 
\end{itemize}
It gives six common neighbours of $\e$ and $\g$: 
 $$
 \y_{[l,s)},\ \y_{[k,s)},\ \x_{[k,l)},\  \x_{[k,s)},\   (\x^{-1}\y)_{[l,s)},\  (\y^{-1}\x)_{[k,l)}.
 $$
 Therefore $|V(\e)\cap V(\g)|=6$.
 
\medskip

\noindent (3) Suppose that $\vartheta(\g)=3$. Then $\g=\x_{[i_1,i_2)}\y_{[i_2,i_3)}\z_{[i_3,i_4)}$, where $x,z\in G^\times$, $x\neq y$, $y\neq z$ and $1\leqslant i_1<i_2<i_3<i_4\leqslant m+1.$ 
Any neighbour of $\g$,  has the form $\h_{[k,l)} \x_{[i_1,i_2)}\y_{[i_2,i_3)}\z_{[i_3,i_4)}$, where $h\in G^\times$ and $\{k,l\}\subseteq \{i_1,i_2,i_3,i_4\}$. 

Assume first that $x=z$. If $y\neq e$ and $y\neq x^2$ then for $\g=\x_{[i_1,i_2)}\y_{[i_2,i_3)}\x_{[i_3,i_4)}$ the element $\h\g$ has weight $1$ only when $\h={\x^{-1}}_{[i_1,i_4)}$ or $\h=(\x\y^{-1})_{[i_2,i_3)}$. So we get two neighbours of $\g$ of weight $1$: 
 \begin{equation}\label{x=z}
 (\x^{-1}\y)_{[i_2,i_3)}\ \  {\rm and}\ \ \x_{[i_1,i_4)}.
 \end{equation}
 If $y=x^2$, then beside these two vertices we have two other neighbours of $\g$ of weight $1$. Taking $\h={\x^{-1}}_{[i_1,i_3)}$ or $\h={\x^{-1}}_{[i_2,i_4)}$ we obtain
 $$\x_{[i_2,i_4)}\ \  {\rm and}\ \ \x_{[i_1,i_3)}$$
respectively. So in this case we have four common neighbours of $\e$ and $\g$. 

If $y=e$, then beside the vertices given in (\ref{x=z}) we have two other neighbours of $\g$ of weight $1$: 
 $$
 \x_{[i_1,i_2)}=\x^{-1}_{[i_3,i_4)}\g\ \  {\rm and}\ \ \x_{[i_3,i_4)}=\x^{-1}_{[i_1,i_2)}\g.
 $$ 
 These are the only neigbours of $\g$ of weight $1$ if we additionaly assume that $x\neq x^{-1}$. So under this additional assumption we have again four common neighbours of $\e$ and $\g$.  If $x$ is an element of order $2$, then also 
 $$\x_{[i_1,i_3)}=\x_{[i_2,i_4)}\g\ \  {\rm and}\ \ \x_{[i_2,i_4)}=\x^{-1}_{[i_1,i_3)}\g
 $$
are neighbours of $\g$ of weight $1$ and in this last case there are $6$ common neighbours of $\e$ and $\g$.

Now let $z = x^{-1}$, $o(x)>2$.  If $y=e$, that is $\g=\x_{[i_1,i_2)}\e_{[i_2,i_3)}\x^{-1}_{[i_3,i_4)}$, the only neighbours of $\g$ of weight $1$ are equal $\h\g$, where $\h\in\{\x_{[i_3,i_4)},\ \x_{[i_2,i_4)},\ \x^{-1}_{[i_1,i_2)}, \x^{-1}_{[i_1,i_3)}\}$. This set of neighbours is equal $$\{\x_{[i_1,i_2)},\ \x_{[i_1,i_3)},\ \x^{-1}_{[i_3,i_4)},\ \x^{-1}_{[i_2,i_4)} \}.$$ 
If $x\neq y \neq e$, then multiplying $\g = \x_{[i_1,i_2)}\y_{[i_2,i_3)}\x^{-1}_{[i_3,i_4)}$ by any element from any interval $\pr{k}{l}$ with $k\in \{i_1,i_2,i_3\}$ and $l\in\{i_2,i_3,i_4\}$, $k<l$, we cannot obtain an element of weight $1$. Therefore such $\g$ and $\e$ do not have common neighbours. 

In view of the considered cases we may assume that $x\neq z$ and $x\neq z^{-1}$. If $y=e$, then one can easily notice that the only common neighbours of $\g = \x_{[i_1,i_2)}\e_{[i_2,i_3)}\z_{[i_3,i_4)}$ and $\e$ are $\x_{[i_1,i_2)}$ and $\z_{[i_3,i_4)}$. So assume that $y\neq e$. 
It is easily seen that multyplying $\g$ by any element from the intervals $\pr{i_1}{i_2}$, $\pr{i_2}{i_3}$, $\pr{i_3}{i_4}$ or $\pr{i_1}{i_4}$ we do not get an element of weight $1$. So let us consider the product 
$$\h_{[i_1,i_3)}\g=(\h\x)_{[i_1,i_2)}(\h\y)_{[i_2,i_3)}\z_{[i_3,i_4)}.$$
Since $z\neq e$ and $hx\neq hy$, this is an element of weight $1$ when $hx=1$ and $hy=z$. This gives 
$$ h=x^{-1},\ \ z=x^{-1}y,\ \ \h_{[i_1,i_3)}\g=(\x^{-1}\y)_{[i_2,i_4)}.$$ 
Analogously, the product
$$\h_{[i_2,i_4)}\g=\x_{[i_1,i_2)}(\h\y_{[i_2,i_3)})(\h\z)_{[i_3,i_4)}$$
has weight $1$ when $hy=x$ and $hz=1$, that is when $h=z^{-1}$ and $z^{-1}y=x$. This means that 
$$z=yx^{-1},\ \  \h_{[i_2,i_4)}\g=\x_{[i_1,i_3)}).$$ 
Thus these two products are different neighbours of $\e$ only when 
\begin{equation}\label{xy=yx}
x^{-1}y = z = yx^{-1}, 
\end{equation}
that is, when $x$ and $y$ commute. If only one equation of (\ref{xy=yx}) is satisfied, then $\g$ has exactly one common neighbour with $\e$. If none of the equations is satisfied, $\g$ and $\e$ do not have common neighbours.

\medskip

\noindent (4) If $\vartheta(\g)\geqslant 4$, then by Lemma \ref{weight} the weight of any neighbour of $\g$ is not less than $2$. Hence $V(\e)\cap V(\g)=\emptyset$.
\end{proof} 

\begin{cor}
The sets of common neighbours of $\e$ and $\g$ depending on $\vartheta(\g)$ are presented in the following table:\vspace{4pt}	
\begin{center}
{\tiny\begin{tabular}{||l||c|c|c|c|c|c|c||}
\hline
\multicolumn{8}{||c||}{}\\[-4pt]
\multicolumn{8}{||c||}{$\vartheta(\g)=1$}  \\[4pt]
\hline
 &&&&&&&\\[-4pt]
       &                &              & $i<k$             &  $k<i<l$          & $k<j<l$ & $l<j$ &\\[4pt]
$\h$   & $\x_{[i,k)}$   & $\x_{[l,j)}$ & $\x^{-1}_{[i,l)}$ & $\x^{-1}_{[i,l)}$ & $\x^{-1}_{[k,j)}$ & $\x^{-1}_{[k,j)}$ & $\y_{[k,l)}$\\[4pt]
\hline
$\g$   & \multicolumn{7}{|l||} {$\x_{[k,l)}$}  \\[4pt]
\hline
$\h\g$ & $\x_{[i,l)}$   & $\x_{[k,j)}$  & $\x^{-1}_{[i,k)}$ & $\x_{[k,i)}$ & $\x_{[j,l)}$& $\x^{-1}_{[l,j)}$ & $(\y\x)_{[k,l)}$ \\[4pt]
\hline
\hline
\multicolumn{8}{|c|}{}\\[-4pt]
\multicolumn{8}{|c|}{$\vartheta(\g)=2$}  \\[4pt]
\hline
$\h$   & $\x^{-1}_{[k,l)}$   & $(\y\x^{-1})_{[k,l)}$ & $\y^{-1}_{[l,s)}$ & $(\x\y^{-1})_{[l,s)}$ & $\x^{-1}_{[k,s)}$ & $\y^{-1}_{[k,s)}$ & \\[4pt]
\hline
$\g$   & \multicolumn{7}{|l||} {$\x_{[k,l)}\y_{[l,s)}$}  \\[4pt]
\hline
$\h\g$ & $\y_{[l,s)}$   & $\y_{[k,s)}$  & $\x_{[k,l)}$ & $\x_{[k,s)}$ & $(\x^{-1}\y)_{[l,s)}$ & $(\y^{-1}\x)_{[k,l)}$ & \\[4pt]
\hline
\hline
\multicolumn{8}{|c|}{}  \\[-4pt]
\multicolumn{8}{|c|}{$\vartheta(\g)=3$}  \\[4pt]
\hline
$\h$   & $\x_{[i_1,i_4)}$   & $\x_{[i_2,i_3)}$ & $\x_{[i_3,i_4)}$ & $\x_{[i_1,i_2)}$ & $\x_{[i_2,i_4)}$ & $\x_{[i_1,i_3)}$ & \\[4pt]
\hline
$\g$   & \multicolumn{7}{|l||} {$\x_{[i_1,i_2)}\e_{[i_2,i_3)}\x_{[i_3,i_4)}$,\ \ $o(x)=2$}  \\[4pt]
\hline
$\h\g$ & $\x_{[i_2,i_3)}$   & $\x_{[i_1,i_4)}$  & $\x_{[i_1,i_2)}$ & $\x_{[i_3,i_4)}$ & $\x_{[i_1,i_3)}$ & $\x_{[i_2,i_4)}$ & \\[4pt]
\hline
\hline
$\h$   & $\x^{-1}_{[i_1,i_4)}$   & $\x_{[i_2,i_3)}$ & $\x^{-1}_{[i_3,i_4)}$ & $\x^{-1}_{[i_1,i_2)}$ & \multicolumn{1}{||c|}{$\x^{-1}_{[i_1,i_4)}$} & $(\x\y^{-1})_{[i_2,i_3)}$ & \\[4pt]
\hline
$\g$   & \multicolumn{4}{|l|} {$\x_{[i_1,i_2)}\e_{[i_2,i_3)}\x_{[i_3,i_4)}$,\ \ $o(x)>2$} & \multicolumn{2}{||l|}{$\x_{[i_1,i_2)}\y_{[i_2,i_3)}\x_{[i_3,i_4)}, \ x^2\neq y\neq e$} & \\[4pt]
\hline
$\h\g$ & $\x^{-1}_{[i_2,i_3)}$   & $\x_{[i_1,i_4)}$  & $\x_{[i_1,i_2)}$ & $\x_{[i_3,i_4)}$ & \multicolumn{1}{||c|}{$(\x^{-1}\y)_{[i_2,i_3)}$} & $\x_{[i_1,i_4)}$ & \\[4pt]
\hline
\hline
$\h$   & $\x_{[i_3,i_4)}$   & $\x_{[i_2,i_4)}$ & $\x^{-1}_{[i_1,i_2)}$ & $\x^{-1}_{[i_1,i_3)}$ & \multicolumn{1}{||c|}{$\z^{-1}_{[i_3,i_4)}$} & $\x^{-1}_{[i_1,i_2)}$ & \\[4pt]
\hline
$\g$   & \multicolumn{4}{|l|} {$\x_{[i_1,i_2)}\e_{[i_2,i_3)}\x^{-1}_{[i_3,i_4)}$,\ \ $o(x)>2$} & \multicolumn{2}{||l|}{$\x_{[i_1,i_2)}\e_{[i_2,i_3)}\z_{[i_3,i_4)}, \ z\neq x\neq z^{-1}$} & \\[4pt]
\hline
$\h\g$ & $\x_{[i_1,i_2)}$   & $\x_{[i_1,i_3)}$  & $\x^{-1}_{[i_3,i_4)}$ & $\x^{-1}_{[i_2,i_4)}$ & \multicolumn{1}{||c|}{$\x_{[i_1,i_2)}$} & $\z_{[i_3,i_4)}$ & \\[4pt]
\hline
\hline
$\h$   & $\x^{-1}_{[i_1,i_4)}$   & $\x^{-1}_{[i_2,i_3)}$ & $\x^{-1}_{[i_2,i_4)}$ & $\x^{-1}_{[i_1,i_3)}$ & \multicolumn{3}{||c||}{} \\[4pt]
\hline
$\g$   & \multicolumn{4}{|l|} {$\x_{[i_1,i_2)}\x_{[i_2,i_3)}^2\x_{[i_3,i_4)}$,\ \ $o(x)>2$} & \multicolumn{3}{||l||}{} \\[4pt]
\hline
$\h\g$ & $\x_{[i_2,i_3)}$   & $\x_{[i_1,i_4)}$  & $\x_{[i_1,i_3)}$ & $\x_{[i_2,i_4)}$ & \multicolumn{3}{||c||}{} \\[4pt]
\hline
\hline

$\h$   & \multicolumn{4}{|l|}{$\x^{-1}_{[i_1,i_3)}$}  & \multicolumn{3}{||l||} {$\x\y^{-1}_{[i_2,i_4)}$}  \\[4pt]
\hline
$\g$   & \multicolumn{4}{|l|} {$\x_{[i_1,i_2)}\y_{[i_2,i_3)}(\x^{-1}\y)_{[i_3,i_4)}$,\ \ $xy\neq yx$} & \multicolumn{3}{||l||}{$\x_{[i_1,i_2)}\y_{[i_2,i_3)}(\y\x^{-1})_{[i_3,i_4)}, \ xy\neq yx$}  \\[4pt]
\hline
$\h\g$ & \multicolumn{4}{|l|} {$(\x^{-1}\y)_{[i_2,i_4)}$}   & \multicolumn{3}{||l||} {$\x_{[i_1,i_3)}$}    \\[4pt]
\hline

\end{tabular}}
\end{center}
\end{cor}

Recall that a non-empty $k$-regular graph $X$ on $n$ vertices is called edge regular if there exists a constant $a$ such that every pair of adjacent vertices has precisely $a$ common neighbours. Then we say that $X$ is edge regular with parameters $(n,k,a)$. Furthermore, the graph  $X$ is said to be strongly regular with parameters $(n, k, a, c)$ if it is edge regular with parameters $(n,k,a)$ and every pair of distinct nonadjacent vertices has $c$ common neighbours. From Proposition \ref{neighbours}, it follows immediately that $\G_2(G)$ is strongly regular and for  $m\geqslant 3$ the graph $\G_m(G)$ is edge regular.
Indeed,  since the right transfer $T_\g\colon G^m\to G^m$, $\x\mapsto \x\g$ is an automorphism of  $\G_m(G)$, the number of common neighbours of adjacent  vertices $\g$ and $\h$  can be computed as follows 
$$
|V(\g)\cap V(\h)|=|V(\e)\cap V(\h\g^{-1})|=|G|+2m-4.
$$ 
Consequently $\G_m(G)$ is edge regular. Since the weight of any element of $\G_2(G)$ does not exceed $2$,  only points 1 and 2 of Proposition \ref{neighbours} apply to this graph. 
Furthermore, a simple calculation using formulas from Section 10.2 in \cite{GR} shows that

\begin{cor}\label{G2} For a given non-trivial group $G$ and integer $m>1$
\begin{enumerate}
\item the graph $\G_m(G)$ is edge regular with parameters 
$$
(|G|^m,\binom{m+1}{2}(|G| -1),|G|+2m-4);
$$
\item the graph  $\G_2(G)$ is strongly regular with parameters $$(|G|^2,3(|G|-1) ,|G|,6).$$
The eigenvalues of $\G_2(G)$ are $3(|G|-1)$, $|G|-3$, and $-3$ with multiplicities  equal to $1$, $3(|G|-1)$ and
$|G|^2-3|G|+2$ respectively.
\end{enumerate}

\end{cor}

In light of Corollary \ref{G2} the spectrum of $\G_2(G)$ depends only on the size of the group $G$. It could suggest that $\G_m(G)$ does not carry to much information about $G$. However, it is not true. In next sections we show that for any $m\geqslant 2$ the graph $\G_m(G)$ fully determines $G$ in the sense that  for  given non isomorphic groups $G$ and $H$ the graphs $\G_m(G)$ and $\G_m(H)$ are also non isomorphic.

\bigskip


\section{Maximum cliques}

For $\g\in G^m$ by $\V_m(\g)=\V_m(\g,G)$ we denote a subgraph of $\G_m(G)$ whose vertex set is equal $V(\g)$. For an element $x\in G^{\times}$ by $\I_m(x,G)$ ($\I_m(x)$ for short) we denote a subgraph of $\V_m(\e,G)$ whose vertex set is equal 
$$I_m(x)=\{\x_{[k,l)}:\ 1\leqslant k < l \leqslant m+1\}\cup \{\x^{-1}_{[k,l)}:\ 1\leqslant k < l \leqslant m+1\}.$$ 


\begin{lem}\label{I(x)}
For a given non-trivial group $G$, $x\in G^{\times}$ and integer $m>1$:
\begin{enumerate} 
	\item If $x$ has order $2$, then $|I_m(x)|=m(m+1)/2$; if $x$ has order bigger than $2$, then $|I_m(x)|=m(m+1)$;
	\item $\I_m(x)$ is a $d$-regular graph with $d=2m-2$ when $x$ has order $2$ and $d=2m-1$ when $x$ has order bigger than $2$. 
\end{enumerate}	
\end{lem}


\begin{proof}
	The first assertion follows strightforward from the definifion of $I_m(x)$. The proof of the second one follows from the analysis done in the beginning of the proof of Proposition \ref{neighbours}. In fact, in the graph $\I_m(x,G)$ for a fixed interval $\pr{k}{l}$ a vertex $\x_{[k,l)}$ is adjacent to: 
	\begin{enumerate}
		\item[$\bullet$] $m-k+l-2$ vertices of the form $\x_{[i,j)}$ (when $k=i$, $k<j$ and $l\neq j$ or $l=j$, $i<l$ and $k\neq i$),
		\item[$\bullet$] $m+k-l$ vertices of the form $\x^{-1}_{[i,j)}$ (when $k=j+1$ and $i<k$ or $l=i$ and $l<j$),
		\item[$\bullet$] $\x^{-1}_{[k,l)}$ (this connection is excluded when $x$ is an element of order $2$).\vspace{-20pt}
	\end{enumerate}
\end{proof}


As it is seen the structure of $\I_m(x, G)$ depends only on $m$ and the fact whether $x$ has order $2$ or not. It does not depend on the structure of $G$. See {\bf Fig.\,1} and {\bf Fig.\,2} presenting $\mathscr{I}_m(x)$ for $m=3$ and $m=4$. 


\begin{center}
\begin{tikzpicture}[scale=0.35] 
 \clip (-9,-10.5) rectangle (9,9);
\def\R{2.2} 
\def\r{2}

\def\ca{0.866}
\def\cb{0}
\def\cc{-0.866}
\def\cd{-0.866}
\def\ce{0}
\def\cf{0.866}

\def\sa{0.5}
\def\sb{1}
\def\sc{0.5}
\def\sd{-0.5}
\def\se{-1}
\def\sf{-0.5}

\def\a{3}

\def\Aa{(\a*\R*\ca,\a*\R*\sa)}
\def\Ab{(\a*\R*\cb,\a*\R*\sb)}
\def\Ac{(\a*\R*\cc,\a*\R*\sc)}
\def\Ad{(\a*\R*\cd,\a*\R*\sd)}
\def\Ae{(\a*\R*\ce,\a*\R*\se)}
\def\Af{(\a*\R*\cf,\a*\R*\sf)}

\def\b{1.2}


\def\Ba{(\b*\R*\ca,\b*\R*\sa)}
\def\Bb{(\b*\R*\cb,\b*\R*\sb)}
\def\Bc{(\b*\R*\cc,\b*\R*\sc)}
\def\Bd{(\b*\R*\cd,\b*\R*\sd)}
\def\Be{(\b*\R*\ce,\b*\R*\se)}
\def\Bf{(\b*\R*\cf,\b*\R*\sf)}

\def\v{4.3pt}


\draw[very thick] [color=red]  
\Aa -- \Ba
\Ab -- \Bb
\Ac -- \Bc
\Ad -- \Bd
\Ae -- \Be
\Af -- \Bf;


\filldraw[black, thick] 
\Aa circle(\v)
\Ab circle(\v)
\Ac circle(\v)
\Ad circle(\v)
\Ae circle(\v)
\Af circle(\v);


\node[right] at  (\a*\R*\ca + 0.15,\a*\R*\sa)   {\footnotesize ${\bf{x}_{[2,4)}}$};
\node[above] at  (\a*\R*\cb,\a*\R*\sb + 0.1)    {\footnotesize ${\bf{x}^{-1}_{[1,2)}}$};
\node[left]  at  (\a*\R*\cc - 0.15,\a*\R*\sc)   {\footnotesize ${\bf{x}_{[2,3)}}$};
\node[left]  at  (\a*\R*\cd - 0.15,\a*\R*\sd)   {\footnotesize ${\bf{x}_{[1,3)}}$};
\node[below] at  (\a*\R*\ce,\a*\R*\se - 0.15)   {\footnotesize ${\bf{x}_{[3,4)}}$};
\node[right] at  (\a*\R*\cf + 0.15,\a*\R*\sf)   {\footnotesize ${\bf{x}_{[1,4)}}$};


\filldraw[black, thick] 
\Ba circle(\v)
\Bb circle(\v)
\Bc circle(\v)
\Bd circle(\v)
\Be circle(\v)
\Bf circle(\v);

\draw[thick] \Aa -- \Ab -- \Ac -- \Ad -- \Be -- \Bf --\Ba -- \Bb -- \Bc -- \Bd -- \Ae -- 
\Af -- \Aa;
\draw[thin]   \Aa -- \Ac -- \Be -- \Ba -- \Bc -- \Ae -- \Aa 
                    \Ab -- \Bd -- \Bf -- \Ab 
                    \Ad -- \Af -- \Bb -- \Ad; 

\node at (0,-9.5) {\footnotesize {\rm  The graph}\  $\mathscr{I}_3(x),\  \ o(x)>2$};

\end{tikzpicture} 
\hskip1cm
\begin{tikzpicture}[scale=0.35] 
 \clip (-9,-10.5) rectangle (9,9);
\def\R{2.2} 
\def\r{2}

\def\ca{0.866}
\def\cb{0}
\def\cc{-0.866}
\def\cd{-0.866}
\def\ce{0}
\def\cf{0.866}

\def\sa{0.5}
\def\sb{1}
\def\sc{0.5}
\def\sd{-0.5}
\def\se{-1}
\def\sf{-0.5}

\def\a{3}

\def\Aa{(\a*\R*\ca,\a*\R*\sa)}
\def\Ab{(\a*\R*\cb,\a*\R*\sb)}
\def\Ac{(\a*\R*\cc,\a*\R*\sc)}
\def\Ad{(\a*\R*\cd,\a*\R*\sd)}
\def\Ae{(\a*\R*\ce,\a*\R*\se)}
\def\Af{(\a*\R*\cf,\a*\R*\sf)}

\def\b{3}


\def\Ba{(\b*\R*\ca,\b*\R*\sa)}
\def\Bb{(\b*\R*\cb,\b*\R*\sb)}
\def\Bc{(\b*\R*\cc,\b*\R*\sc)}
\def\Bd{(\b*\R*\cd,\b*\R*\sd)}
\def\Be{(\b*\R*\ce,\b*\R*\se)}
\def\Bf{(\b*\R*\cf,\b*\R*\sf)}

\def\v{4.3pt}


\draw[thick] [color=red]  
\Aa -- \Ba
\Ab -- \Bb
\Ac -- \Bc
\Ad -- \Bd
\Ae -- \Be
\Af -- \Bf;


\filldraw[black, very thick] 
\Aa circle(\v)
\Ab circle(\v)
\Ac circle(\v)
\Ad circle(\v)
\Ae circle(\v)
\Af circle(\v);


\node[right] at  (\a*\R*\ca + 0.15,\a*\R*\sa)   {\footnotesize ${\bf{x}_{[2,4)}}$};
\node[above] at  (\a*\R*\cb,\a*\R*\sb + 0.1)    {\footnotesize ${\bf{x}_{[1,2)}}$};
\node[left]  at  (\a*\R*\cc - 0.15,\a*\R*\sc)   {\footnotesize ${\bf{x}_{[2,3)}}$};
\node[left]  at  (\a*\R*\cd - 0.15,\a*\R*\sd)   {\footnotesize ${\bf{x}_{[1,3)}}$};
\node[below] at  (\a*\R*\ce,\a*\R*\se - 0.15)   {\footnotesize ${\bf{x}_{[3,4)}}$};
\node[right] at  (\a*\R*\cf + 0.15,\a*\R*\sf)   {\footnotesize ${\bf{x}_{[1,4)}}$};


\filldraw[black, very thick] 
\Ba circle(\v)
\Bb circle(\v)
\Bc circle(\v)
\Bd circle(\v)
\Be circle(\v)
\Bf circle(\v);

\draw[thick] \Aa -- \Ab -- \Ac -- \Ad -- \Be -- \Bf --\Ba -- \Bb -- \Bc -- \Bd -- \Ae -- 
\Af -- \Aa;
\draw[thin]   \Aa -- \Ac -- \Be -- \Ba -- \Bc -- \Ae -- \Aa 
                    \Ab -- \Bd -- \Bf -- \Ab 
                    \Ad -- \Af -- \Bb -- \Ad; 

\node at (0,-9.5) {\footnotesize {\rm  The graph}\  $\mathscr{I}_3(x),\  \ o(x)=2$};
\end{tikzpicture}

\par
\smallskip

{\bf Fig.\,1.}
\end{center}
\smallskip

\begin{center}
\begin{tikzpicture}[scale=.4] 
 \clip (-9,-10.5) rectangle (9,8);
\def\R{2.2} 
\def\r{2}

\def\ca{0.951}
\def\cb{0.588}
\def\cc{0}
\def\cd{-0.588}
\def\ce{-0.951}
\def\cf{-0.951}
\def\cg{-0.588}
\def\ch{0}
\def\ci{0.588}
\def\cj{0.951}

\def\sa{0.309}
\def\sb{0.809}
\def\sc{1}
\def\sd{0.809}
\def\se{0.309}
\def\sf{-0.309}
\def\sg{-0.809}
\def\sh{-1}
\def\si{-0.809}
\def\sj{-0.309}

\def\a{3}

\def\Aa{(\a*\R*\ca,\a*\R*\sa)}
\def\Ab{(\a*\R*\cb,\a*\R*\sb)}
\def\Ac{(\a*\R*\cc,\a*\R*\sc)}
\def\Ad{(\a*\R*\cd,\a*\R*\sd)}
\def\Ae{(\a*\R*\ce,\a*\R*\se)}
\def\Af{(\a*\R*\cf,\a*\R*\sf)}
\def\Ag{(\a*\R*\cg,\a*\R*\sg)}
\def\Ah{(\a*\R*\ch,\a*\R*\sh)}
\def\Ai{(\a*\R*\ci,\a*\R*\si)}
\def\Aj{(\a*\R*\cj,\a*\R*\sj)}

\def\b{.55}
\def\c{2.25}

\def\Ba{(\b*\R*\ca,\b*\R*\sa)}
\def\Bb{(\c*\R*\cb,\c*\R*\sb)}
\def\Bc{(\b*\R*\cc,\b*\R*\sc)}
\def\Bd{(\c*\R*\cd,\c*\R*\sd)}
\def\Be{(\b*\R*\ce,\b*\R*\se)}
\def\Bf{(\c*\R*\cf,\c*\R*\sf)}
\def\Bg{(\b*\R*\cg,\b*\R*\sg)}
\def\Bh{(\c*\R*\ch,\c*\R*\sh)}
\def\Bi{(\b*\R*\ci,\b*\R*\si)}
\def\Bj{(\c*\R*\cj,\c*\R*\sj)}

\def\v{4pt}


\draw[thick] [color=red]  
\Aa -- \Ba
\Ab -- \Bb
\Ac -- \Bc
\Ad -- \Bd
\Ae -- \Be
\Af -- \Bf
\Ag -- \Bg
\Ah -- \Bh
\Ai -- \Bi
\Aj -- \Bj;


\filldraw[black, thick] 
\Aa circle(\v)
\Ab circle(\v)
\Ac circle(\v)
\Ad circle(\v)
\Ae circle(\v)
\Af circle(\v)
\Ag circle(\v)
\Ah circle(\v)
\Ai circle(\v)
\Aj circle(\v);


\node[right] at       (\a*\R*\ca + 0.15,\a*\R*\sa)        {\footnotesize ${\bf{x}^{-1}_{[2,5)}}$};
\node[above right] at (\a*\R*\cb - 0.10,\a*\R*\sb + 0.1)  {\footnotesize ${\bf{x}_{[1,2)}}$};
\node[above]  at      (\a*\R*\cc,\a*\R*\sc + 0.15)        {\footnotesize ${\bf{x}_{[1,3)}}$};
\node[above left] at  (\a*\R*\cd + 0.1,\a*\R*\sd + 0.1)   {\footnotesize ${\bf{x}_{[2,3)}}$};
\node[left] at        (\a*\R*\ce - 0.1,\a*\R*\se)         {\footnotesize ${\bf{x}_{[2,4)}}$};
\node[left] at        (\a*\R*\cf - 0.1,\a*\R*\sf)         {\footnotesize ${\bf{x}_{[3,4)}}$};
\node[below left] at  (\a*\R*\cg + 0.1,\a*\R*\sg - 0.1)   {\footnotesize ${\bf{x}_{[3,5)}}$};
\node[below] at       (\a*\R*\ch,\a*\R*\sh - 0.15)        {\footnotesize ${\bf{x}_{[4,5)}}$};
\node[below right] at (\a*\R*\ci - 0.2,\a*\R*\si - 0.1)   {\footnotesize ${\bf{x}^{-1}_{[1,4)}}$};
\node[right] at       (\a*\R*\cj + 0.15,\a*\R*\sj)        {\footnotesize ${\bf{x}^{-1}_{[1,5)}}$};


\filldraw[black, thick] 
\Ba circle(\v)
\Bb circle(\v)
\Bc circle(\v)
\Bd circle(\v)
\Be circle(\v)
\Bf circle(\v)
\Bg circle(\v)
\Bh circle(\v)
\Bi circle(\v)
\Bj circle(\v);

\draw[thick] \Aa -- \Ab -- \Ac -- \Ad -- \Ae -- \Af --\Ag -- \Ah -- \Ai -- \Aj -- \Aa;
\draw[thick] \Ba -- \Bb -- \Bc -- \Bd -- \Be -- \Bf --\Bg -- \Bh -- \Bi -- \Bj -- \Ba;
\draw[thin]   \Aa -- \Bd -- \Ag -- \Bj -- \Ac -- \Bf -- \Ai -- \Bb -- \Ae -- \Bh -- \Aa 
                    \Aa -- \Be -- \Ai -- \Bc -- \Ag -- \Ba -- \Ae -- \Bi -- \Ac -- \Bg -- \Aa
                    \Ab -- \Bd -- \Af -- \Bh -- \Aj -- \Bb -- \Ad -- \Bf -- \Ah -- \Bj -- \Ab
                    \Ab -- \Be -- \Ah -- \Ba -- \Ad -- \Bg -- \Aj -- \Bc -- \Af -- \Bi -- \Ab; 

\node at (0,-9.5) {\footnotesize {\rm  The graph}\  $\mathscr{I}_4(x),\  \ o(x)>2$};

\end{tikzpicture} \hskip1cm
\begin{tikzpicture}[scale=.4] 
 \clip (-9,-10.5) rectangle (9,8);
\def\R{2.2} 
\def\r{2}

\def\ca{0.951}
\def\cb{0.588}
\def\cc{0}
\def\cd{-0.588}
\def\ce{-0.951}
\def\cf{-0.951}
\def\cg{-0.588}
\def\ch{0}
\def\ci{0.588}
\def\cj{0.951}

\def\sa{0.309}
\def\sb{0.809}
\def\sc{1}
\def\sd{0.809}
\def\se{0.309}
\def\sf{-0.309}
\def\sg{-0.809}
\def\sh{-1}
\def\si{-0.809}
\def\sj{-0.309}

\def\a{3}

\def\Aa{(\a*\R*\ca,\a*\R*\sa)}
\def\Ab{(\a*\R*\cb,\a*\R*\sb)}
\def\Ac{(\a*\R*\cc,\a*\R*\sc)}
\def\Ad{(\a*\R*\cd,\a*\R*\sd)}
\def\Ae{(\a*\R*\ce,\a*\R*\se)}
\def\Af{(\a*\R*\cf,\a*\R*\sf)}
\def\Ag{(\a*\R*\cg,\a*\R*\sg)}
\def\Ah{(\a*\R*\ch,\a*\R*\sh)}
\def\Ai{(\a*\R*\ci,\a*\R*\si)}
\def\Aj{(\a*\R*\cj,\a*\R*\sj)}

\def\b{3}
\def\c{3}

\def\Ba{(\b*\R*\ca,\b*\R*\sa)}
\def\Bb{(\c*\R*\cb,\c*\R*\sb)}
\def\Bc{(\b*\R*\cc,\b*\R*\sc)}
\def\Bd{(\c*\R*\cd,\c*\R*\sd)}
\def\Be{(\b*\R*\ce,\b*\R*\se)}
\def\Bf{(\c*\R*\cf,\c*\R*\sf)}
\def\Bg{(\b*\R*\cg,\b*\R*\sg)}
\def\Bh{(\c*\R*\ch,\c*\R*\sh)}
\def\Bi{(\b*\R*\ci,\b*\R*\si)}
\def\Bj{(\c*\R*\cj,\c*\R*\sj)}

\def\v{4pt}


\filldraw[black, very thick] 
\Aa circle(\v)
\Ab circle(\v)
\Ac circle(\v)
\Ad circle(\v)
\Ae circle(\v)
\Af circle(\v)
\Ag circle(\v)
\Ah circle(\v)
\Ai circle(\v)
\Aj circle(\v);


\node[right] at       (\a*\R*\ca + 0.15,\a*\R*\sa)        {\footnotesize ${\bf{x}_{[2,5)}}$};
\node[above right] at (\a*\R*\cb - 0.10,\a*\R*\sb + 0.1)  {\footnotesize ${\bf{x}_{[1,2)}}$};
\node[above]  at      (\a*\R*\cc,\a*\R*\sc + 0.15)        {\footnotesize ${\bf{x}_{[1,3)}}$};
\node[above left] at  (\a*\R*\cd + 0.1,\a*\R*\sd + 0.1)   {\footnotesize ${\bf{x}_{[2,3)}}$};
\node[left] at        (\a*\R*\ce - 0.1,\a*\R*\se)         {\footnotesize ${\bf{x}_{[2,4)}}$};
\node[left] at        (\a*\R*\cf - 0.1,\a*\R*\sf)         {\footnotesize ${\bf{x}_{[3,4)}}$};
\node[below left] at  (\a*\R*\cg + 0.1,\a*\R*\sg - 0.1)   {\footnotesize ${\bf{x}_{[3,5)}}$};
\node[below] at       (\a*\R*\ch,\a*\R*\sh - 0.15)        {\footnotesize ${\bf{x}_{[4,5)}}$};
\node[below right] at (\a*\R*\ci - 0.2,\a*\R*\si - 0.1)   {\footnotesize ${\bf{x}_{[1,4)}}$};
\node[right] at       (\a*\R*\cj + 0.15,\a*\R*\sj)        {\footnotesize ${\bf{x}_{[1,5)}}$};


\filldraw[black, very thick] 
\Ba circle(\v)
\Bb circle(\v)
\Bc circle(\v)
\Bd circle(\v)
\Be circle(\v)
\Bf circle(\v)
\Bg circle(\v)
\Bh circle(\v)
\Bi circle(\v)
\Bj circle(\v);

\draw[thick] \Aa -- \Ab -- \Ac -- \Ad -- \Ae -- \Af --\Ag -- \Ah -- \Ai -- \Aj -- \Aa;
\draw[thick] \Ba -- \Bb -- \Bc -- \Bd -- \Be -- \Bf --\Bg -- \Bh -- \Bi -- \Bj -- \Ba;
\draw[thin]   \Aa -- \Bd -- \Ag -- \Bj -- \Ac -- \Bf -- \Ai -- \Bb -- \Ae -- \Bh -- \Aa 
                    \Aa -- \Be -- \Ai -- \Bc -- \Ag -- \Ba -- \Ae -- \Bi -- \Ac -- \Bg -- \Aa
                    \Ab -- \Bd -- \Af -- \Bh -- \Aj -- \Bb -- \Ad -- \Bf -- \Ah -- \Bj -- \Ab
                    \Ab -- \Be -- \Ah -- \Ba -- \Ad -- \Bg -- \Aj -- \Bc -- \Af -- \Bi -- \Ab; 

\draw[thick] [color=red]  
\Aa -- \Ba
\Ab -- \Bb
\Ac -- \Bc
\Ad -- \Bd
\Ae -- \Be
\Af -- \Bf
\Ag -- \Bg
\Ah -- \Bh
\Ai -- \Bi
\Aj -- \Bj;

\node at (0,-9.5) {\footnotesize {\rm  The graph}\  $\mathscr{I}_4(x),\  \ o(x)=2$};

\end{tikzpicture}

\smallskip

{\bf Fig.\,2.}\footnote{The red edges connect elements with their inverses}
\end{center}


\begin{lem} 
	Let $m>1$ be a fixed integer. For non-trivial groups $G$ and $H$ and elements $x\in G^{\times}$ and $y\in H^{\times}$
	$$\I_m(x,G)\simeq \I_m(y,H)\ \ {\rm if\ and\ only\ if}\ \ o(x)=2=o(y)\ \ {\rm or}\ \ o(x)\neq 2\neq o(y).$$
\end{lem}


This means that locally, around each vertex, all graphs $\G_m(G)$ are isomorphic. 


\begin{prop}\label{V(g)=V(h)} 
	Let $G$ and $H$ be groups such that $|G|=|H|$. If $G$ and $H$ have the same numbers of elements of order $2$, then for any $\g\in G^m$ and $\h\in H^m$
	$$\V_m(\g,G)\simeq \V_m(\h,H).$$	
\end{prop} 


\begin{proof}
Observe that for each interval $\pr{k}{l}$ the graph $\V_m(\e,G)$ has a clique containing all $|G|-1$\ vertices $\x_{[k,l)}$, $x\in G^{\times}$. So this graph is the union of all disjoint graphs $\I_m(x,G)$, $x\in G$ completed by all edges belonging by these cliques beside edges $\x_{[k,l)}\sim \x^{-1}_{[k,l)}$ which are in $\I_m(x,G)$.  So the structure of $\V_m(\e,G)$ depends only on $m$ and the number of elements of order $2$.
\end{proof}

\bigskip

 Let $\B_m$ be a graph whose vertices are all the intervals $\pr{k}{l}$, $1\leqslant k<l\leqslant m+1$. We say that two intervals $\pr{k}{l}$ and $\pr{i}{j}$ are adjacent if there exist $\x_{[k,l)}\in \pr{k}{l}$ and $\y_{[i,j)}\in \pr{i}{j}$ which are adjacent in $\G_m(G)$. Note that $\B_m$ is isomorphic to the graph $\I_m(x,G)$, where $x$ is an element of order $2$. It can be obtained also from $\I_m(x,G)$, where $x$ has order bigger than $2$, by removing the  edges $\x_{[k,l)}\sim\x^{-1}_{[k,l)}$, with simultaneous merging of ends and identifying suitable edges to avoid multiple edges.


\begin{prop}\label{kneser}
	The complement $\overline{\B}_m$ of the graph $\B_m$ is isomorphic to the Kneser graph $KG_{m+1,2}$.
\end{prop}


\begin{proof}
It is seen from observation done in the proof of Proposition \ref{neighbours} that in the graph $\B_m$ two intervals $\pr{k}{l}$ and $\pr{i}{j}$ are adjacent if and only if $|\{k,l\}\cap\{i,j\}|=1$. The vertex set of the Kneser graph $GK_{m+1,2}$ consists of all $2$-element subsets of $\{1,2,\ldots,m+1 \}$. The subsets $\{k,l\}$ and $\{i,j\}$ are adjacent in it if and only if they are disjoint. This is the opposite condition to that satisfied by adjacent intervals $\pr{k}{l}$ and $\pr{i}{j}$.	
\end{proof}


Now we are ready to describe maximum cliques of $\G_m(G)$. Since the right transfer maps create a transitive subgroup of the automorphism group of $\G_m(G)$ we need to describe cliques containing $\e$ only, and then contained in $V(\e)\cup\{\e\} = \Ss\cup\{\e\}$.  A clique which beside $\e$ contains vertices of one fixed interval we call {\it an interval clique}. If a clique contains vertices from different intervals, then we call it {\it a dispersed clique}. If $\Q$ is such a clique, then by $\Q^{*}$ we denote the subgraph of $\Q$ on vertices different from $\e$. Notice that for any
$x\in G^\times$ and $1\leqslant j\leqslant m+1$  vertices from the set
$$C(x,j)=\{\x_{[i,j)}^{-1}\, |\, 1\leqslant i < j\}\cup\{\x_{[j,k)}\, |\, j<k\leqslant m+1\}$$
form an $m$-element dispersed clique contained in $I_m(x)$. We use the name {\em maximum dispersed (interval) clique} for a dispersed (resp. interval) clique with a maximal number of vertices. It is clear that notions of maximum and maximal interval clique mean the same, while one can find $3$-vertex dispersed cliques which are maximal but not maximum cliques.


\begin{lem}\label{interval_clique}
	If $\Q$ is a maximum interval clique in $\G_m(G)$, then $|Q|=|G|$. Moreover, for any $\x,\y \in Q^{*}$, $\x \neq \y$, and $\g \in V(\e)$, if $\{\x,\y,\g\}$ is a clique, then $\g \in Q$.
\end{lem}


\begin{proof}
The first part of the lemma is obvious. The second one follows from the fact that for a given $\g\in V(\e)$ and an interval $\pr{k}{l}$ not containing $\g$ there exists at most one vertex in $\pr{k}{l}$ adjacent to $\g$.
\end{proof}


\begin{lem}\label{dispersed_clique} 
Let  $\Q$ be a maximum dispersed clique in $\G_m(G)$. Then $|Q| =  m+1$ and there exists $x\in G^{\times}$ such that $Q^*\subseteq I_m(x)$. If  $m\geqslant 3$ and $(m,o(x))\neq (3,2)$, then
there exists $j$, $1\leqslant j\leqslant m+1$ such that $Q^*=C(x,j)$.  Furthermore
	\begin{enumerate}
	\item if $o(x)>2$, then for any $\ga, \gb \in Q^{*}$, $\ga \neq \gb$, and $\g \in V(\e)$, if $\{\ga,\gb,\g\}$ is a clique, then $\g \in Q$.
	\item if $o(x)=2$, then for any $\ga,\gb \in Q^{*}$, $\ga \neq \gb$, there exists $\g \in V(\e)\setminus Q$ such that $\{\ga,\gb,\g\}$ is a clique. 
	\item if $m=3$ and $o(x)=2$, then either $Q^{*}=C(x,j)$  or $Q^*=\{\x_{[i,j)}, \x_{[i,k)}, \x_{[j,k)}\}$ for some $1\leqslant i<j<k\leqslant 4$.
	\end{enumerate}
\end{lem}


\begin{proof} 
Take a clique $\Q$ with a maximal number of vertices and at least two vertices from different intervals. If $\x_{[k,l)}\in Q^{*}$ and $\y\in Q^{*}$ is a vertex from another interval $\pr{i}{j}$, then it is the unique vertex from this interval and either $\y=\x_{[i,j)}$ or $\y=\x^{-1}_{[i,j)}$. Moreover, as it was observed in the proof of Proposition \ref{kneser}, $\{k,l\}$ and $\{i,j\}$ have exactly one common element. This last condition means that there are at most as many elements in $Q^{*}$ as in the maximal family of $2$-element subsets of the set $\{1,2,\ldots,m+1\}$ every two of which has one common element. By the well known Erd{\"o}s-Ko-Rado theorem on intersecting families, the maximal family contains at most $m$ subsets  and then $|Q|\leqslant m+1$. On the other hand for a fixed element $x\in G$ the set $\{ \e, \x_{[1,2)}, \ldots \x_{[1,m+1)}\}$ forms a clique in $\G_m(G)$ with $m+1$ vertices.

For the proof of the second part of Lemma we may assume that $m>3$ as the case $m=3$ is easily seen from {\bf Fig.\,1}. It follows from the mentioned Erd{\"o}s-Ko-Rado theorem that the set of all pairs ${k,l}$ such that $\x_{[k,l)}$ or $\x_{[k,l)}^{-1}$ has exactly one common element $j$.
	Now it follows from Lemma \ref{SS_in_S} which vertices belong to $Q$.
\end{proof}


As a consequence of Lemmas \ref{interval_clique} and \ref{dispersed_clique} we obtain

\begin{prop}\label{maximum_clique}
	The maximal number of vertices in a clique of a graph $\G_{m}(G)$ is equal 
	$$\left\{\begin{array}{cl}
	{\rm max}\{m+1, |G|\} & {\rm if}\ \ (m,|G|)\neq (2,2)\\
	4 & {\rm if}\ \ (m,|G|)= (2,2).\\
	\end{array}\right.
	$$
\end{prop}


\begin{proof} Using a suitable right transfer map we may assume that the considered cliques contain $\e$, so they are either interval cliques or dispersed cliques. If $m>2$ the proposition follows immediately from both Lemmas. For $m=2$ and $x$ of order $2$ we have a $4$-vertex clique $\{\e,(x,e),(x,x),(e,x)\}$. Note also that if $G$ does not have  elements of order $2$, then there is no a $4$-vertex clique in $\G_2(G)$.  
\end{proof}


If $v$ is a vertex of a graph $\X$ and $\Y$ is a subgraph of $\mathscr{X}$ with the vertex set $Y$, then we say that $v$ is {\em a neighbour of $\Y$} if $v\notin Y$ and for some $y\in Y$ we have $v\sim y$. By $\N(\Y)$ we denote the subgraph of $\X$ whose set of vertices is equal to the set of all neighbours of $\Y$. According to our convention, $N(\Y)$ is the vertex set of $\N(\Y)$.

Notice that if $|G|=m+1$, then maximum interval cliques and maximum dispersed cliques in $\G_m(G)$ have the same size. Below we will show
that the {\em neighbours graph} $\N(\Q^*)$ is an invariant differentiating the type of a maximum clique. 

\begin{prop}\label{inv_type}
Let $G$ be a group of order bigger than $3$ and $|G|=m+1$. Then
\begin{enumerate}
\item if $\Q$ is a maximum dispersed clique in $\G_m(G)$, then the graph $\N(\Q^*)$ is not regular;

\item if $\Q$ is  a maximum interval clique in $\G_m(G)$, then the graph $\N(\Q^*)$ is regular of degree $d=2m-2$.
\end{enumerate}
\end{prop}
\begin{proof}
Figure 3 presents graphs $\V_3(\e)$ for  the groups $C_4=\{e,x,x^2,x^3\}$ and $C_2\times C_2=\{e,x,y,z\}$.  The neighbours graphs for dispersed and interval cliques of $\V_3(\e)$ (for these groups) are presented on the Figures 4,5 and 6.
\begin{center}
\begin{tikzpicture}[scale=0.3] 
 \clip (-12,-12) rectangle (12,12);
\def\R{2.2} 
\def\r{2}

\def\ca{0.866}
\def\cb{0}
\def\cc{-0.866}
\def\cd{-0.866}
\def\ce{0}
\def\cf{0.866}

\def\sa{0.5}
\def\sb{1}
\def\sc{0.5}
\def\sd{-0.5}
\def\se{-1}
\def\sf{-0.5}


\def\cam{0.5}
\def\cbm{-0.5}
\def\ccm{-1}
\def\cdm{-0.5}
\def\cem{0.5}
\def\cfm{1}

\def\sam{0.866}
\def\sbm{0.866}
\def\scm{0}
\def\sdm{-0.866}
\def\sem{-0.866}
\def\sfm{0}

\def\a{4}

\def\Aa{(\a*\R*\ca,\a*\R*\sa)}
\def\Ab{(\a*\R*\cb,\a*\R*\sb)}
\def\Ac{(\a*\R*\cc,\a*\R*\sc)}
\def\Ad{(\a*\R*\cd,\a*\R*\sd)}
\def\Ae{(\a*\R*\ce,\a*\R*\se)}
\def\Af{(\a*\R*\cf,\a*\R*\sf)}

\def\b{2.5}


\def\Ba{(\b*\R*\ca,\b*\R*\sa)}
\def\Bb{(\b*\R*\cb,\b*\R*\sb)}
\def\Bc{(\b*\R*\cc,\b*\R*\sc)}
\def\Bd{(\b*\R*\cd,\b*\R*\sd)}
\def\Be{(\b*\R*\ce,\b*\R*\se)}
\def\Bf{(\b*\R*\cf,\b*\R*\sf)}

\def\v{5pt}

\def\c{1.2}


\def\Ca{(\c*\R*\cam,\c*\R*\sam)}
\def\Cb{(\c*\R*\cbm,\c*\R*\sbm)}
\def\Cc{(\c*\R*\ccm,\c*\R*\scm)}
\def\Cd{(\c*\R*\cdm,\c*\R*\sdm)}
\def\Ce{(\c*\R*\cem,\c*\R*\sem)}
\def\Cf{(\c*\R*\cfm,\c*\R*\sfm)}


\draw[very thick] [color=red]  
\Aa -- \Ba -- \Ca -- \Aa
\Ab -- \Bb -- \Cb -- \Ab
\Ac -- \Bc -- \Cc -- \Ac
\Ad -- \Bd -- \Cd -- \Ad
\Ae -- \Be -- \Ce -- \Ae
\Af -- \Bf -- \Cf -- \Af;


\filldraw[black, thick] 
\Aa circle(\v)
\Ab circle(\v)
\Ac circle(\v)
\Ad circle(\v)
\Ae circle(\v)
\Af circle(\v);


\node[right] at  (\a*\R*\ca + 0.15,\a*\R*\sa)   {\footnotesize ${\bf{x}_{[2,4)}}$};
\node[above] at  (\a*\R*\cb,\a*\R*\sb + 0.1)    {\footnotesize ${\bf{x}^{-1}_{[1,2)}}$};
\node[left]  at  (\a*\R*\cc - 0.15,\a*\R*\sc)   {\footnotesize ${\bf{x}_{[2,3)}}$};
\node[left]  at  (\a*\R*\cd - 0.15,\a*\R*\sd)   {\footnotesize ${\bf{x}_{[1,3)}}$};
\node[below] at  (\a*\R*\ce,\a*\R*\se - 0.15)   {\footnotesize ${\bf{x}_{[3,4)}}$};
\node[right] at  (\a*\R*\cf + 0.15,\a*\R*\sf)   {\footnotesize ${\bf{x}_{[1,4)}}$};


\filldraw[black, thick] 
\Ba circle(\v)
\Bb circle(\v)
\Bc circle(\v)
\Bd circle(\v)
\Be circle(\v)
\Bf circle(\v);

\draw[thick] \Aa -- \Ab -- \Ac -- \Ad -- \Be -- \Bf --\Ba -- \Bb -- \Bc -- \Bd -- \Ae -- 
\Af -- \Aa;
\draw[thin]   \Aa -- \Ac -- \Be -- \Ba -- \Bc -- \Ae -- \Aa 
                    \Ab -- \Bd -- \Bf -- \Ab 
                    \Ad -- \Af -- \Bb -- \Ad;

\filldraw[blue, thick] 
\Ca circle(\v)
\Cb circle(\v)
\Cc circle(\v)
\Cd circle(\v)
\Ce circle(\v)
\Cf circle(\v);

\draw[thick, blue]
\Ca -- \Cb -- \Cc -- \Cd -- \Ce -- \Cf -- \Ca;
\draw[thick, thin, blue]
\Ca -- \Cc -- \Ce -- \Ca
\Cb -- \Cd -- \Cf -- \Cb;

\node at (0,-11.5) {\footnotesize {\rm  The graph}\  $\mathscr{V}_3(\mathbf{e})$\ for $G=C_4$};

\end{tikzpicture} 
\hskip0.5cm
\begin{tikzpicture}[scale=0.3] 
 \clip (-12,-12) rectangle (12,12);
\def\R{2.2} 
\def\r{2}

\def\ca{0.866}
\def\cb{0}
\def\cc{-0.866}
\def\cd{-0.866}
\def\ce{0}
\def\cf{0.866}

\def\sa{0.5}
\def\sb{1}
\def\sc{0.5}
\def\sd{-0.5}
\def\se{-1}
\def\sf{-0.5}


\def\cam{0.5}
\def\cbm{-0.5}
\def\ccm{-1}
\def\cdm{-0.5}
\def\cem{0.5}
\def\cfm{1}

\def\sam{0.866}
\def\sbm{0.866}
\def\scm{0}
\def\sdm{-0.866}
\def\sem{-0.866}
\def\sfm{0}

\def\a{4}

\def\Aa{(\a*\R*\ca,\a*\R*\sa)}
\def\Ab{(\a*\R*\cb,\a*\R*\sb)}
\def\Ac{(\a*\R*\cc,\a*\R*\sc)}
\def\Ad{(\a*\R*\cd,\a*\R*\sd)}
\def\Ae{(\a*\R*\ce,\a*\R*\se)}
\def\Af{(\a*\R*\cf,\a*\R*\sf)}

\def\b{2.5}


\def\Ba{(\b*\R*\ca,\b*\R*\sa)}
\def\Bb{(\b*\R*\cb,\b*\R*\sb)}
\def\Bc{(\b*\R*\cc,\b*\R*\sc)}
\def\Bd{(\b*\R*\cd,\b*\R*\sd)}
\def\Be{(\b*\R*\ce,\b*\R*\se)}
\def\Bf{(\b*\R*\cf,\b*\R*\sf)}

\def\v{5pt}

\def\c{1.2}


\def\Ca{(\c*\R*\cam,\c*\R*\sam)}
\def\Cb{(\c*\R*\cbm,\c*\R*\sbm)}
\def\Cc{(\c*\R*\ccm,\c*\R*\scm)}
\def\Cd{(\c*\R*\cdm,\c*\R*\sdm)}
\def\Ce{(\c*\R*\cem,\c*\R*\sem)}
\def\Cf{(\c*\R*\cfm,\c*\R*\sfm)}


\draw[very thick] [color=red]  
\Aa -- \Ba -- \Ca -- \Aa
\Ab -- \Bb -- \Cb -- \Ab
\Ac -- \Bc -- \Cc -- \Ac
\Ad -- \Bd -- \Cd -- \Ad
\Ae -- \Be -- \Ce -- \Ae
\Af -- \Bf -- \Cf -- \Af;


\filldraw[black, thick] 
\Aa circle(\v)
\Ab circle(\v)
\Ac circle(\v)
\Ad circle(\v)
\Ae circle(\v)
\Af circle(\v);


\node[right] at  (\a*\R*\ca + 0.15,\a*\R*\sa)   {\footnotesize ${\bf{x}_{[2,4)}}$};
\node[above] at  (\a*\R*\cb,\a*\R*\sb + 0.1)    {\footnotesize ${\bf{x}_{[1,2)}}$};
\node[left]  at  (\a*\R*\cc - 0.15,\a*\R*\sc)   {\footnotesize ${\bf{x}_{[2,3)}}$};
\node[left]  at  (\a*\R*\cd - 0.15,\a*\R*\sd)   {\footnotesize ${\bf{x}_{[1,3)}}$};
\node[below] at  (\a*\R*\ce,\a*\R*\se - 0.15)   {\footnotesize ${\bf{x}_{[3,4)}}$};
\node[right] at  (\a*\R*\cf + 0.15,\a*\R*\sf)   {\footnotesize ${\bf{x}_{[1,4)}}$};


\filldraw[black, thick] 
\Ba circle(\v)
\Bb circle(\v)
\Bc circle(\v)
\Bd circle(\v)
\Be circle(\v)
\Bf circle(\v);

\draw[thick] 
\Aa -- \Ab -- \Ac -- \Ad -- \Ae -- \Af -- \Aa;
\draw[thin]   
\Aa -- \Ac -- \Ae -- \Aa
\Ab -- \Ad -- \Af -- \Ab;

\draw[thick]
\Ba -- \Bb --\Bc -- \Bd -- \Be -- \Bf -- \Ba; 
\draw[thin]
\Ba  --\Bc -- \Be -- \Ba 
\Bb  --\Bd -- \Bf -- \Bb;

\filldraw[thick] 
\Ca circle(\v)
\Cb circle(\v)
\Cc circle(\v)
\Cd circle(\v)
\Ce circle(\v)
\Cf circle(\v);

\draw[thick]
\Ca -- \Cb -- \Cc -- \Cd -- \Ce -- \Cf -- \Ca;
\draw[thin]
\Ca -- \Cc -- \Ce -- \Ca
\Cb -- \Cd -- \Cf -- \Cb;

\node at (0,-11.5) {\footnotesize {\rm  The graph}\  $\mathscr{V}_3(\mathbf{e})$\ for $G=C_2\times C_2$};

\end{tikzpicture}

\par
\

{\bf Fig.\,3.}\footnote{The red edges connect elements from the same interval}

\end{center}

\begin{center}
\begin{tikzpicture}[scale=0.3] 
 \clip (-12.7,-14) rectangle (13,8);
\def\R{2.2} 
\def\r{2}

\def\ca{0.866}
\def\cb{0}
\def\cc{-0.866}
\def\cd{-0.866}
\def\ce{0}
\def\cf{0.866}

\def\sa{0.5}
\def\sb{1}
\def\sc{0.5}
\def\sd{-0.5}
\def\se{-1}
\def\sf{-0.5}


\def\cam{0.5}
\def\cbm{-0.5}
\def\ccm{-1}
\def\cdm{-0.5}
\def\cem{0.5}
\def\cfm{1}

\def\sam{0.866}
\def\sbm{0.866}
\def\scm{0}
\def\sdm{-0.866}
\def\sem{-0.866}
\def\sfm{0}

\def\a{4}

\def\Aa{(\a*\R*\ca,\a*\R*\sa)}
\def\Ab{(\a*\R*\cb,\a*\R*\sb)}
\def\Ac{(\a*\R*\cc,\a*\R*\sc)}
\def\Ad{(\a*\R*\cd,\a*\R*\sd)}
\def\Ae{(\a*\R*\ce,\a*\R*\se)}
\def\Af{(\a*\R*\cf,\a*\R*\sf)}

\def\b{2.5}


\def\Ba{(\b*\R*\ca,\b*\R*\sa)}
\def\Bb{(\b*\R*\cb,\b*\R*\sb)}
\def\Bc{(\b*\R*\cc,\b*\R*\sc)}
\def\Bd{(\b*\R*\cd,\b*\R*\sd)}
\def\Be{(\b*\R*\ce,\b*\R*\se)}
\def\Bf{(\b*\R*\cf,\b*\R*\sf)}

\def\v{5pt}

\def\c{1.2}


\def\Ca{(\c*\R*\cam,\c*\R*\sam)}
\def\Cb{(\c*\R*\cbm,\c*\R*\sbm)}
\def\Cc{(\c*\R*\ccm,\c*\R*\scm)}
\def\Cd{(\c*\R*\cdm,\c*\R*\sdm)}
\def\Ce{(\c*\R*\cem,\c*\R*\sem)}
\def\Cf{(\c*\R*\cfm,\c*\R*\sfm)}




\filldraw[black, thick] 
\Ad circle(\v)
\Ae circle(\v)
\Af circle(\v);


\node[left]  at  (\a*\R*\cd - 0.15,\a*\R*\sd)   {\footnotesize ${\bf{x}_{[1,3)}}$};
\node[below] at  (\a*\R*\ce,\a*\R*\se - 0.15)   {\footnotesize ${\bf{x}_{[3,4)}}$};
\node[right] at  (\a*\R*\cf + 0.15,\a*\R*\sf)   {\footnotesize ${\bf{x}_{[1,4)}}$};

\node[right] at  (\b*\R*\ca + 0.15,\b*\R*\sa)   {\footnotesize ${\bf{x}^{-1}_{[2,4)}}$};
\node[above] at  (\b*\R*\cb,\b*\R*\sb + 0.1)    {\footnotesize ${\bf{x}_{[1,2)}}$};
\node[left]  at  (\b*\R*\cc - 0.15,\b*\R*\sc)   {\footnotesize ${\bf{x}^{-1}_{[2,3)}}$};
\node[right] at  (\b*\R*\ce,\b*\R*\se-.5)   {\footnotesize ${\bf{x}^{-1}_{[3,4)}}$};
\node[below] at  (\c*\R*\cam + 0.5,\c*\R*\sam)   {\footnotesize ${\bf{x}^2_{[2,4)}}$};
\node[right]  at  (\c*\R*\ccm,\c*\R*\scm-.5)   {\footnotesize ${\bf{x}^2_{[2,3)}}$};


\filldraw[black, thick] 
\Ba circle(\v)
\Bb circle(\v)
\Bc circle(\v)
\Bd circle(\v)
\Be circle(\v)
\Bf circle(\v);


\draw[thick] 
\Ba -- \Bb
\Ba -- \Bf
\Ba -- \Bc
\Ba -- \Be
\Bb -- \Ba
\Bb -- \Bc
\Bb -- \Ad
\Bb -- \Af
\Bc -- \Bd
\Bc -- \Ae
\Ad -- \Be
\Ad -- \Af
\Bd -- \Ae
\Bd -- \Bf
\Ae -- \Af
\Be -- \Bf;
\draw[blue, thick] 
\Ca -- \Cb -- \Cc -- \Ca;


\draw[thick,red] 
\Ba -- \Ca
\Bb -- \Cb
\Bc -- \Cc
\Ae -- \Be
\Ad -- \Bd
\Af -- \Bf;

\filldraw[blue, thick] 
\Ca circle(\v)
\Cb circle(\v)
\Cc circle(\v)
;

\node at (0,-11.5) {\footnotesize {\rm  The graph $\N(\Q^*)$ for a dispersed clique} };
\node at (0,-13) {\footnotesize ($Q^*=\{\mathbf{x}^{-1}_{[1,2)},\mathbf{x}_{[2,3)},\mathbf{x}_{[2,4)} \}$,\  $G=C_4$)};
\end{tikzpicture} 
\hskip.5cm
\begin{tikzpicture}[scale=0.3] 
 \clip (-12.7,-14) rectangle (13,8);
\def\R{2.2} 
\def\r{2}

\def\ca{0.866}
\def\cb{0}
\def\cc{-0.866}
\def\cd{-0.866}
\def\ce{0}
\def\cf{0.866}

\def\sa{0.5}
\def\sb{1}
\def\sc{0.5}
\def\sd{-0.5}
\def\se{-1}
\def\sf{-0.5}


\def\cam{0.5}
\def\cbm{-0.5}
\def\ccm{-1}
\def\cdm{-0.5}
\def\cem{0.5}
\def\cfm{1}

\def\sam{0.866}
\def\sbm{0.866}
\def\scm{0}
\def\sdm{-0.866}
\def\sem{-0.866}
\def\sfm{0}

\def\a{4}

\def\Aa{(\a*\R*\ca,\a*\R*\sa)}
\def\Ab{(\a*\R*\cb,\a*\R*\sb)}
\def\Ac{(\a*\R*\cc,\a*\R*\sc)}
\def\Ad{(\a*\R*\cd,\a*\R*\sd)}
\def\Ae{(\a*\R*\ce,\a*\R*\se)}
\def\Af{(\a*\R*\cf,\a*\R*\sf)}

\def\b{2.5}


\def\Ba{(\b*\R*\ca,\b*\R*\sa)}
\def\Bb{(\b*\R*\cb,\b*\R*\sb)}
\def\Bc{(\b*\R*\cc,\b*\R*\sc)}
\def\Bd{(\b*\R*\cd,\b*\R*\sd)}
\def\Be{(\b*\R*\ce,\b*\R*\se)}
\def\Bf{(\b*\R*\cf,\b*\R*\sf)}

\def\v{5pt}

\def\c{1.2}


\def\Ca{(\c*\R*\cam,\c*\R*\sam)}
\def\Cb{(\c*\R*\cbm,\c*\R*\sbm)}
\def\Cc{(\c*\R*\ccm,\c*\R*\scm)}
\def\Cd{(\c*\R*\cdm,\c*\R*\sdm)}
\def\Ce{(\c*\R*\cem,\c*\R*\sem)}
\def\Cf{(\c*\R*\cfm,\c*\R*\sfm)}




\filldraw[black, thick] 
\Ad circle(\v)
\Ae circle(\v)
\Af circle(\v);




\node[left]  at  (\a*\R*\cd - 0.15,\a*\R*\sd)   {\footnotesize ${\bf{x}_{[1,3)}}$};
\node[below] at  (\a*\R*\ce,\a*\R*\se - 0.15)   {\footnotesize ${\bf{x}_{[3,4)}}$};
\node[right] at  (\a*\R*\cf + 0.15,\a*\R*\sf)   {\footnotesize ${\bf{x}_{[1,4)}}$};


\node[right] at  (\b*\R*\ca + 0.15,\b*\R*\sa)   {\footnotesize ${\bf{y}_{[2,4)}}$};
\node[above] at  (\b*\R*\cb,\b*\R*\sb + 0.1)    {\footnotesize ${\bf{y}_{[1,2)}}$};
\node[left]  at  (\b*\R*\cc - 0.15,\b*\R*\sc)   {\footnotesize ${\bf{y}_{[2,3)}}$};


\node[right, below] at  (\c*\R*\cam + 1.2,\c*\R*\sam +.1)   {\footnotesize ${\bf{z}_{[2,4)}}$};
\node[left] at  (\c*\R*\cbm+.2,\c*\R*\sbm-.15)    {\footnotesize ${\bf{z}_{[1,2)}}$};
\node[below]  at  (\c*\R*\ccm,\c*\R*\scm)   {\footnotesize ${\bf{z}_{[2,3)}}$};


\filldraw[black, thick] 
\Ba circle(\v)
\Bb circle(\v)
\Bc circle(\v);


\draw[thick]
\Ad -- \Ae -- \Af -- \Ad
\Ba -- \Bb --\Bc -- \Ba;
\draw[thick, red]
\Ba -- \Ca
\Bb -- \Cb
\Bc -- \Cc;

\filldraw[thick] 
\Ca circle(\v)
\Cb circle(\v)
\Cc circle(\v);

\draw[thick]
\Ca -- \Cb -- \Cc -- \Ca;

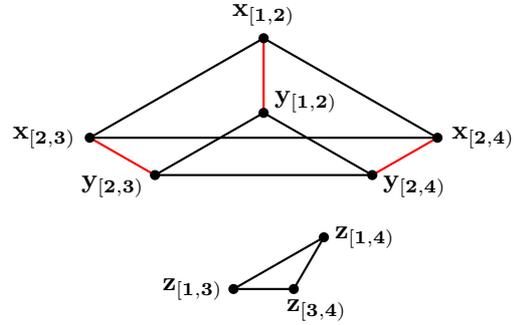
\node at (0,-11.5) {\footnotesize {\rm  The graph $\N(\Q^*)$ for a dispersed clique} };
\node at (0,-13) {\footnotesize ($\Q^*=\{\mathbf{x}_{[1,2)},\mathbf{x}_{[2,3)},\mathbf{x}_{[2,4)} \}$,\  $G=C_2\times C_2$)};

\end{tikzpicture} 

\par
\smallskip

{\bf Fig.\,4.}

\end{center}


Observe that the set of neighbours of any vertex $\x_{[k,s)}$ in   $I_m(x)$ can be decomposed as a union of two disjoint
$(m-1)$-element cliques. More precisely
\begin{equation}\label{neighbours_vertex}
V(\x_{[k,s)})\cap I_m(x)=\left( C(x,k)\setminus\{\x_{[k,s)}\}\right)\cup \left(C(x^{-1},s)\setminus\{\x_{[k,s)}\}\right).
\end{equation}

1. Let $m>3$ and $\Q$ be a maximum dispersed clique in $\G_m(G)$. By Lemma \ref{dispersed_clique} there exists $x\in G^\times$, such that $\Q^{*}$ is a subgraph of $\I_m(x)$, and there exists $j$, $1\leqslant j\leqslant m+1$ such that $Q^*=C(x,j)$. 

First suppose that $o(x)>2$. We will show  that $|N(\Q^{*})| = 2m(m-1)$. In fact, each vertex from $\Q^{*}$ has exactly $|G|+2m-4=3(m-1)$ neighbours in $\V_m(\e)$ by Proposition \ref{neighbours} (1) and different vertices from $\Q^{*}$ do not have common neighbours. Hence 
$$|N(\Q^{*})|=(3(m-1)-(m-1))\cdot m=2m(m-1).$$
It is easily seen that $N(\Q^{*})\cap I_m(x)=I_m(x)\setminus Q^{*}$. Actually, each vertex from $\I_m(x)$, in particular each vertex from $\Q^{*}$ has degree $2m-1$ in $\I_m(x)$, so  $$|N(\Q^{*})\cap I_m(x)| = |Q^{*}|\cdot ((2m-1)-(m-1))=m^2=m(m+1)-m=|I_m(x)|-|Q^{*}|.$$

Now we decompose $N(\Q^*)$ as a union of three disjoint subsets $A$, $B$, $C$ such that within each subset vertices have the same degree. Let 
$
A=C(x^{-1},j)
$ 
be the set of all inverses of elements of $Q^{*}$ and let $B=I_m(x)\setminus (Q^{*}\cup A)$. It is clear that $A\subset I_m(x)$ is an $m$-element dispersed clique disjoint with $Q^*$.
Finally, let 
$$C=\bigcup_{y\notin\{e,x,x^{-1}\}}(\{\y_{[i,j)}^{-1}\, |\, 1\leqslant i < j\}\cup\{\y_{[j,k)}\, |\, j<k\leqslant m+1\})=\bigcup_{y\notin\{e,x,x^{-1}\}}C(y,j)$$
that is $C$ is the set of neighbours of $Q^{*}$ outside $I_m(x)$. It is clear that $N(\Q^*)\subseteq A\cup B\cup C$. Moreover, as it follows from the definition, $|A|=m$, $|B| = m(m+1)-2m = m(m-1)$ and $|C| = m(|G|-3) = m(m-2)$; therefore according to earlier calculations $|A|+|B|+|C| = m+m(m-1)+m(m-2) = 2m(m-1).$ Consequently, $N(\Q^*)=A\cup B\cup C$.
\begin{center}
\begin{tikzpicture}[scale=0.3] 
 \clip (-12.7,-10) rectangle (13,11);
\def\R{2.2} 
\def\r{2}

\def\ca{0.866}
\def\cb{0}
\def\cc{-0.866}
\def\cd{-0.866}
\def\ce{0}
\def\cf{0.866}

\def\sa{0.5}
\def\sb{1}
\def\sc{0.5}
\def\sd{-0.5}
\def\se{-1}
\def\sf{-0.5}


\def\cam{0.5}
\def\cbm{-0.5}
\def\ccm{-1}
\def\cdm{-0.5}
\def\cem{0.5}
\def\cfm{1}

\def\sam{0.866}
\def\sbm{0.866}
\def\scm{0}
\def\sdm{-0.866}
\def\sem{-0.866}
\def\sfm{0}

\def\a{4}

\def\Aa{(\a*\R*\ca,\a*\R*\sa)}
\def\Ab{(\a*\R*\cb,\a*\R*\sb)}
\def\Ac{(\a*\R*\cc,\a*\R*\sc)}
\def\Ad{(\a*\R*\cd,\a*\R*\sd)}
\def\Ae{(\a*\R*\ce,\a*\R*\se)}
\def\Af{(\a*\R*\cf,\a*\R*\sf)}

\def\b{2.5}


\def\Ba{(\b*\R*\ca,\b*\R*\sa)}
\def\Bb{(\b*\R*\cb,\b*\R*\sb)}
\def\Bc{(\b*\R*\cc,\b*\R*\sc)}
\def\Bd{(\b*\R*\cd,\b*\R*\sd)}
\def\Be{(\b*\R*\ce,\b*\R*\se)}
\def\Bf{(\b*\R*\cf,\b*\R*\sf)}

\def\v{5pt}

\def\c{1.2}


\def\Ca{(\c*\R*\cam,\c*\R*\sam)}
\def\Cb{(\c*\R*\cbm,\c*\R*\sbm)}
\def\Cc{(\c*\R*\ccm,\c*\R*\scm)}
\def\Cd{(\c*\R*\cdm,\c*\R*\sdm)}
\def\Ce{(\c*\R*\cem,\c*\R*\sem)}
\def\Cf{(\c*\R*\cfm,\c*\R*\sfm)}


\node[right] at  (\a*\R*\ca + 0.15,\a*\R*\sa)   {\footnotesize ${\bf{x}_{[2,4)}}$};
\node[above] at  (\a*\R*\cb,\a*\R*\sb + 0.1)    {\footnotesize ${\bf{x}^{-1}_{[1,2)}}$};
\node[left]  at  (\a*\R*\cc - 0.15,\a*\R*\sc)   {\footnotesize ${\bf{x}_{[2,3)}}$};


\node[right] at  (\b*\R*\ca,\b*\R*\sa - 0.5)   {\footnotesize ${\bf{x}^{-1}_{[2,4)}}$};
\node[right] at  (\b*\R*\cb,\b*\R*\sb + 0.5)    {\footnotesize ${\bf{x}_{[1,2)}}$};
\node[left]  at  (\b*\R*\cc,\b*\R*\sc -0.5)   {\footnotesize ${\bf{x}^{-1}_{[2,3)}}$};

\node[left] at  (\c*\R*\cdm,\c*\R*\sdm)   {\footnotesize ${\bf{x}^2_{[1,3)}}$};
\node[below] at  (\c*\R*\cem+1,\c*\R*\sem + 0.1)    {\footnotesize ${\bf{x}^2_{[3,4)}}$};
\node[right]  at  (\c*\R*\cfm,\c*\R*\sfm)   {\footnotesize ${\bf{x}^2_{[1,4)}}$};


\draw[thick] [color=red]  
\Aa -- \Ba
\Ab -- \Bb
\Ac -- \Bc;

\draw[thick] 
\Aa -- \Ab -- \Ac -- \Aa;

\draw[thick] 
\Ba -- \Bb -- \Bc -- \Ba;

\draw[blue, thick] 
\Cd -- \Ce -- \Cf -- \Cd;

\filldraw[thick] 
\Aa circle(\v)
\Ab circle(\v)
\Ac circle(\v);

\filldraw[black, thick] 
\Ba circle(\v)
\Bb circle(\v)
\Bc circle(\v);

\filldraw[blue, thick] 
\Cd circle(\v)
\Ce circle(\v)
\Cf circle(\v);

\node at (0,-6.5) {\footnotesize {\rm  The graph $\N(\Q^*)$ for a dispersed clique}};
\node at (0,-8.5) {\footnotesize ($Q^*=\{\mathbf{x}^2_{[1,2)},\mathbf{x}^2_{[2,3)},\mathbf{x}^2_{[2,4)} \}$,\  $G=C_4$)};
\end{tikzpicture} 
\hskip0.5cm
\begin{tikzpicture}[scale=0.3] 
 \clip (-12.7,-10) rectangle (13,11);
\def\R{2.2} 
\def\r{2}

\def\ca{0.866}
\def\cb{0}
\def\cc{-0.866}
\def\cd{-0.866}
\def\ce{0}
\def\cf{0.866}

\def\sa{0.5}
\def\sb{1}
\def\sc{0.5}
\def\sd{-0.5}
\def\se{-1}
\def\sf{-0.5}


\def\cam{0.5}
\def\cbm{-0.5}
\def\ccm{-1}
\def\cdm{-0.5}
\def\cem{0.5}
\def\cfm{1}

\def\sam{0.866}
\def\sbm{0.866}
\def\scm{0}
\def\sdm{-0.866}
\def\sem{-0.866}
\def\sfm{0}

\def\a{4}

\def\Aa{(\a*\R*\ca,\a*\R*\sa)}
\def\Ab{(\a*\R*\cb,\a*\R*\sb)}
\def\Ac{(\a*\R*\cc,\a*\R*\sc)}
\def\Ad{(\a*\R*\cd,\a*\R*\sd)}
\def\Ae{(\a*\R*\ce,\a*\R*\se)}
\def\Af{(\a*\R*\cf,\a*\R*\sf)}

\def\b{2.5}


\def\Ba{(\b*\R*\ca,\b*\R*\sa)}
\def\Bb{(\b*\R*\cb,\b*\R*\sb)}
\def\Bc{(\b*\R*\cc,\b*\R*\sc)}
\def\Bd{(\b*\R*\cd,\b*\R*\sd)}
\def\Be{(\b*\R*\ce,\b*\R*\se)}
\def\Bf{(\b*\R*\cf,\b*\R*\sf)}

\def\v{5pt}

\def\c{1.2}


\def\Ca{(\c*\R*\cam,\c*\R*\sam)}
\def\Cb{(\c*\R*\cbm,\c*\R*\sbm)}
\def\Cc{(\c*\R*\ccm,\c*\R*\scm)}
\def\Cd{(\c*\R*\cdm,\c*\R*\sdm)}
\def\Ce{(\c*\R*\cem,\c*\R*\sem)}
\def\Cf{(\c*\R*\cfm,\c*\R*\sfm)}


\node[right] at  (\a*\R*\ca + 0.15,\a*\R*\sa)   {\footnotesize ${\bf{x}_{[2,4)}}$};
\node[above] at  (\a*\R*\cb,\a*\R*\sb + 0.1)    {\footnotesize ${\bf{x}_{[1,2)}}$};
\node[left]  at  (\a*\R*\cc - 0.15,\a*\R*\sc)   {\footnotesize ${\bf{x}_{[2,3)}}$};


\node[right] at  (\b*\R*\ca,\b*\R*\sa - 0.5)   {\footnotesize ${\bf{y}_{[2,4)}}$};
\node[right] at  (\b*\R*\cb,\b*\R*\sb + 0.5)    {\footnotesize ${\bf{y}_{[1,2)}}$};
\node[left]  at  (\b*\R*\cc,\b*\R*\sc - 0.5)   {\footnotesize ${\bf{y}_{[2,3)}}$};

\node[left] at  (\c*\R*\cdm,\c*\R*\sdm)   {\footnotesize ${\bf{z}_{[1,3)}}$};
\node[below] at  (\c*\R*\cem +1,\c*\R*\sem)    {\footnotesize ${\bf{z}_{[3,4)}}$};
\node[right]  at  (\c*\R*\cfm,\c*\R*\sfm)   {\footnotesize ${\bf{z}_{[1,4)}}$};


\draw[thick] [color=red]  
\Aa -- \Ba
\Ab -- \Bb
\Ac -- \Bc;

\draw[thick] 
\Aa -- \Ab -- \Ac -- \Aa;

\draw[thick] 
\Ba -- \Bb -- \Bc -- \Ba;

\draw[thick] 
\Cd -- \Ce -- \Cf -- \Cd;

\filldraw[thick] 
\Aa circle(\v)
\Ab circle(\v)
\Ac circle(\v);

\filldraw[black, thick] 
\Ba circle(\v)
\Bb circle(\v)
\Bc circle(\v);

\filldraw[thick] 
\Cd circle(\v)
\Ce circle(\v)
\Cf circle(\v);

\node at (0,-6.5) {\footnotesize {\rm  The graph $\N(\Q^*)$ for a dispersed clique}};
\node at (0,-8.5) {\footnotesize ($Q^*=\{\mathbf{z}_{[1,2)},\mathbf{z}_{[2,3)},\mathbf{z}_{[2,4)} \}$,\  $G=C_2\times C_2$)};
\end{tikzpicture} 

\par
\smallskip

{\bf Fig.\,5.}

\end{center}

Take a vertex $v\in A$. Then either $v=\x_{[l,j)}$ for some $l<j$ or $v=\x^{-1}_{[j,s)}$ for some $s>j$.
By \ref{neighbours_vertex} the set of neighbours of $v=\x_{[l,j)}$ in $I_m(x)\setminus \{\x_{[l,j)}^{-1}\}$
is equal to
$$
\left( C(x,l)\setminus\{\x_{[l,j)}\}\right)\cup \left(C(x^{-1},j)\setminus\{\x_{[l,j)}\}\right)=
\left( C(x,l)\setminus\{\x_{[l,j)}\}\right)\cup \left(A\setminus\{\x_{[l,j)}\}\right).
$$

Since any vertex from $C(x,l)\setminus\{\x_{[l,j)}\}$ is adjacent to some vertex from $Q^*$, we see that $v$
has $2m-2$ neighbours in $A\cup B$. Moreover, $v$ has $|G|-3=m-2$ neighbours in $C$, so the  degree of $v$ in $\N(\Q^{*})$ is equal to $3m-4$. In the case when $v=\x^{-1}_{[j,s)}$, the argument is identical.

If $v\in B$, then all its  neighbours (in $N(\Q^*)$) lie in $I_m(x)\setminus Q^{*}$. By Lemma \ref{dispersed_clique}(1) $v$ has only one neighbour in $Q^*$, so since the degree of $v$ in $\I_m(G)$ is equal to $2m-1$, the degree of $v$ in $\N(\Q^{*})$ is equal to $2m-2$. Finally, if $v\in C$, then  clearly $v\in C(y,j)$ for some $y$, $v$ has $m-1$ neighbours in $C(y,j)$ and $|G|-3 = m-2$ neighbours outside $C(y,j)$. Thus any vertex from $C$ has degree $2m-3$  in $\N(\Q^{*})$. Consequently, if $o(x)>2$, then the graph $\N(\Q^{*})$ is not regular. 

Suppose that $o(x)=2$. Then $N(\Q^*)=D\cup E$, where 
$$D=I_m(x)\setminus Q^*\  \  \ {\rm and}\  \  \  E=\bigcup_{y\not\in\{e,x\}}C(y,j).
$$
Observe that all neighbours of any vertex  $v\in D$ (in $N(\Q^*)$) are contained in $I_m(x)\setminus Q^{*}$. Moreover,
it is clear that if $v=x_{[s,t)}$, then $s\neq j$ and $t\neq j$. Thus $v$ has exactly two neighbours in $Q^*$: $\x_{[s,j)}$ (or $ \x_{[j,s)}$ if $j<s$)  and $\x_{[j,t)}$ (or  $\x_{[t,j)}$ if $t<j$). Since $v$ has degree $2m-2$ in $\I_m(x)$ (see Lemma \ref{I(x)}), the degree of $v$ in $N(\Q^*)$ is equal to $2m-4$. If $v\in E$, then $v\in C(y,j)$ for some $y\not\in\{e,x\}$, so $v$ has $m-1$ neighbours in $C(y,j)$. Then for any $z\in G\setminus \{e,x,y\}$ the vertex $v$ has exactly one neighbour in $C(z,j)$, so the degree of $v$ in $\N(\Q^{*})$ is equal to $m-1+(|G|-3)=2m-3.$ Therefore,
in this case the graph $\N(\Q^*)$ is also not regular.
\medskip

2. {\bf Fig.\,6} presents the neighbours graphs for $m=3$ and groups of order $4$, when $Q^*=\pr{1}{2}$.

\begin{center}
\begin{tikzpicture}[scale=0.3] 
 \clip (-12.7,-10) rectangle (13,8);
\def\R{2.2} 
\def\r{2}

\def\ca{0.866}
\def\cb{0}
\def\cc{-0.866}
\def\cd{-0.866}
\def\ce{0}
\def\cf{0.866}

\def\sa{0.5}
\def\sb{1}
\def\sc{0.5}
\def\sd{-0.5}
\def\se{-1}
\def\sf{-0.5}


\def\cam{0.5}
\def\cbm{-0.5}
\def\ccm{-1}
\def\cdm{-0.5}
\def\cem{0.5}
\def\cfm{1}

\def\sam{0.866}
\def\sbm{0.866}
\def\scm{0}
\def\sdm{-0.866}
\def\sem{-0.866}
\def\sfm{0}

\def\a{4}

\def\Aa{(\a*\R*\ca,\a*\R*\sa)}
\def\Ab{(\a*\R*\cb,\a*\R*\sb)}
\def\Ac{(\a*\R*\cc,\a*\R*\sc)}
\def\Ad{(\a*\R*\cd,\a*\R*\sd)}
\def\Ae{(\a*\R*\ce,\a*\R*\se)}
\def\Af{(\a*\R*\cf,\a*\R*\sf)}

\def\b{2.5}


\def\Ba{(\b*\R*\ca,\b*\R*\sa)}
\def\Bb{(\b*\R*\cb,\b*\R*\sb)}
\def\Bc{(\b*\R*\cc,\b*\R*\sc)}
\def\Bd{(\b*\R*\cd,\b*\R*\sd)}
\def\Be{(\b*\R*\ce,\b*\R*\se)}
\def\Bf{(\b*\R*\cf,\b*\R*\sf)}

\def\v{5pt}

\def\c{1.2}


\def\Ca{(\c*\R*\cam,\c*\R*\sam)}
\def\Cb{(\c*\R*\cbm,\c*\R*\sbm)}
\def\Cc{(\c*\R*\ccm,\c*\R*\scm)}
\def\Cd{(\c*\R*\cdm,\c*\R*\sdm)}
\def\Ce{(\c*\R*\cem,\c*\R*\sem)}
\def\Cf{(\c*\R*\cfm,\c*\R*\sfm)}



\filldraw[black, thick] 
\Aa circle(\v)
\Ac circle(\v)
\Ad circle(\v)
\Af circle(\v);


\node[right] at  (\a*\R*\ca + 0.15,\a*\R*\sa)   {\footnotesize ${\bf{x}_{[2,4)}}$};
\node[left]  at  (\a*\R*\cc - 0.15,\a*\R*\sc)   {\footnotesize ${\bf{x}_{[2,3)}}$};
\node[left]  at  (\a*\R*\cd - 0.15,\a*\R*\sd)   {\footnotesize ${\bf{x}_{[1,3)}}$};
\node[right] at  (\a*\R*\cf + 0.15,\a*\R*\sf)   {\footnotesize ${\bf{x}_{[1,4)}}$};

\node[right] at  (\b*\R*\ca - 0.15,\b*\R*\sa-0.15)   {\tiny ${\bf{x}^{-1}_{[2,4)}}$};
\node[above]  at  (\b*\R*\cc + 0.3,\b*\R*\sc - 0.2)   {\tiny ${\bf{x}^{-1}_{[2,3)}}$};
\node[left]  at  (\b*\R*\cd + 0.5,\b*\R*\sd+0.8)   {\tiny ${\bf{x}^{-1}_{[1,3)}}$};
\node[below] at  (\b*\R*\cf,\b*\R*\sf +0.2)   {\tiny ${\bf{x}^{-1}_{[1,4)}}$};

\node[right] at  (\c*\R*\cam,\c*\R*\sam - 0.4)   {\tiny ${\bf{x}^{2}_{[2,4)}}$};
\node[right]  at  (\c*\R*\ccm,\c*\R*\scm - 0.2)   {\tiny ${\bf{x}^{2}_{[2,3)}}$};
\node[right]  at  (\c*\R*\cdm,\c*\R*\sdm)   {\tiny ${\bf{x}^{2}_{[1,3)}}$};
\node[right] at  (\c*\R*\cfm - 0.3,\c*\R*\sfm +0.4)   {\tiny ${\bf{x}^{2}_{[1,4)}}$};


\filldraw[black, thick] 
\Ba circle(\v)
\Bc circle(\v)
\Bd circle(\v)
\Bf circle(\v);

\draw[thick, red] 
\Aa -- \Ba -- \Ca -- \Aa
\Ac -- \Bc -- \Cc -- \Ac
\Ad -- \Bd -- \Cd -- \Ad
\Af -- \Bf -- \Cf -- \Af;

\draw[thick] 
\Aa -- \Af
\Ac -- \Ad
\Ad -- \Af
\Cd -- \Cf
\Ba -- \Bf
\Bc -- \Bd
\Bd -- \Bf
\Ca -- \Cf
\Cc -- \Cd;
\draw[thick]   
\Aa -- \Ac 
\Ba -- \Bc 
\Ca -- \Cc;

\filldraw[blue, thick] 
\Ca circle(\v)
\Cc circle(\v)
\Cd circle(\v)
\Cf circle(\v);

\draw[thick]

\Ca -- \Cc;

\node at (0,-6.5) {\footnotesize {\rm  The graph $\N(\Q^*)$  of an interval clique} };
\node at (0,-8) {\footnotesize ($\mathcal{Q}^*=\{\mathbf{x}_{[1,2)},\mathbf{x}^{-1}_{[1,2)},\mathbf{x}^{2}_{[1,2)} \}$,\  $G=C_4$)};

\end{tikzpicture} 
\hskip0.5cm
\begin{tikzpicture}[scale=0.3] 
 \clip (-12.7,-10) rectangle (13,8);
\def\R{2.2} 
\def\r{2}

\def\ca{0.866}
\def\cb{0}
\def\cc{-0.866}
\def\cd{-0.866}
\def\ce{0}
\def\cf{0.866}

\def\sa{0.5}
\def\sb{1}
\def\sc{0.5}
\def\sd{-0.5}
\def\se{-1}
\def\sf{-0.5}


\def\cam{0.5}
\def\cbm{-0.5}
\def\ccm{-1}
\def\cdm{-0.5}
\def\cem{0.5}
\def\cfm{1}

\def\sam{0.866}
\def\sbm{0.866}
\def\scm{0}
\def\sdm{-0.866}
\def\sem{-0.866}
\def\sfm{0}

\def\a{4}

\def\Aa{(\a*\R*\ca,\a*\R*\sa)}
\def\Ab{(\a*\R*\cb,\a*\R*\sb)}
\def\Ac{(\a*\R*\cc,\a*\R*\sc)}
\def\Ad{(\a*\R*\cd,\a*\R*\sd)}
\def\Ae{(\a*\R*\ce,\a*\R*\se)}
\def\Af{(\a*\R*\cf,\a*\R*\sf)}

\def\b{2.5}


\def\Ba{(\b*\R*\ca,\b*\R*\sa)}
\def\Bb{(\b*\R*\cb,\b*\R*\sb)}
\def\Bc{(\b*\R*\cc,\b*\R*\sc)}
\def\Bd{(\b*\R*\cd,\b*\R*\sd)}
\def\Be{(\b*\R*\ce,\b*\R*\se)}
\def\Bf{(\b*\R*\cf,\b*\R*\sf)}

\def\v{5pt}

\def\c{1.2}


\def\Ca{(\c*\R*\cam,\c*\R*\sam)}
\def\Cb{(\c*\R*\cbm,\c*\R*\sbm)}
\def\Cc{(\c*\R*\ccm,\c*\R*\scm)}
\def\Cd{(\c*\R*\cdm,\c*\R*\sdm)}
\def\Ce{(\c*\R*\cem,\c*\R*\sem)}
\def\Cf{(\c*\R*\cfm,\c*\R*\sfm)}


\draw[thick, red]  
\Aa -- \Ba -- \Ca -- \Aa
\Ac -- \Bc -- \Cc -- \Ac
\Ad -- \Bd -- \Cd -- \Ad
\Af -- \Bf -- \Cf -- \Af;


\filldraw[black, thick] 
\Aa circle(\v)
\Ac circle(\v)
\Ad circle(\v)
\Af circle(\v);


\node[right] at  (\a*\R*\ca + 0.15,\a*\R*\sa)   {\footnotesize ${\bf{x}_{[2,4)}}$};
\node[left]  at  (\a*\R*\cc - 0.15,\a*\R*\sc)   {\footnotesize ${\bf{x}_{[2,3)}}$};
\node[left]  at  (\a*\R*\cd - 0.15,\a*\R*\sd)   {\footnotesize ${\bf{x}_{[1,3)}}$};
\node[right] at  (\a*\R*\cf + 0.15,\a*\R*\sf)   {\footnotesize ${\bf{x}_{[1,4)}}$};

\node[right] at  (\b*\R*\ca - 0.15,\b*\R*\sa-0.15)  {\tiny ${\bf{y}_{[2,4)}}$};
\node[above] at  (\b*\R*\cc + 0.3,\b*\R*\sc - 0.2)  {\tiny ${\bf{y}_{[2,3)}}$};
\node[left]  at  (\b*\R*\cd + 0.5,\b*\R*\sd+0.2)   	{\tiny ${\bf{y}_{[1,3)}}$};
\node[below] at  (\b*\R*\cf,\b*\R*\sf +0.2)   		{\tiny ${\bf{y}_{[1,4)}}$};

\node[right] at  (\c*\R*\cam,\c*\R*\sam - 0.4)   {\tiny ${\bf{z}_{[2,4)}}$};
\node[right]  at  (\c*\R*\ccm,\c*\R*\scm - 0.2)   {\tiny ${\bf{z}_{[2,3)}}$};
\node[right]  at  (\c*\R*\cdm,\c*\R*\sdm)   {\tiny ${\bf{z}_{[1,3)}}$};
\node[right] at  (\c*\R*\cfm - 0.3,\c*\R*\sfm)   {\tiny ${\bf{z}_{[1,4)}}$};


\filldraw[black, thick] 
\Ba circle(\v)
\Bc circle(\v)
\Bd circle(\v)
\Bf circle(\v);

\draw[thick] 
\Aa -- \Ac -- \Ad -- \Af -- \Aa;

\draw[thick]
\Ba -- \Bc -- \Bd -- \Bf -- \Ba;

\filldraw[thick] 
\Ca circle(\v)
\Cc circle(\v)
\Cd circle(\v)
\Cf circle(\v);

\draw[thick]
\Ca -- \Cc -- \Cd -- \Cf -- \Ca;

\node at (0,-6.5) {\footnotesize {\rm  The graph $\N(\Q^*)$  of an interval clique} };
\node at (0,-8) {\footnotesize ($Q^*=\{\mathbf{x}_{[1,2)},\mathbf{y}_{[1,2)},\mathbf{z}_{[1,2)} \}$,\  $G=C_2\times C_2$)};

\end{tikzpicture} 
\par
\smallskip

{\bf Fig.\,6.}
\end{center}

Suppose that $m>3$ and $\Q$ is an interval clique, i.e. $Q^*=\pr{k}{l}$, where $1\leqslant k<l\leqslant m+1$. Take $x\in G^\times$ and $\x_{[k,l)}\in \pr{k}{l}$. Suppose that $o(x)>2$. According to \ref{neighbours_vertex}  the set of neighbours of $\x_{[k,l)}$ in $I_m(x)$, 
is the sum of two disjoint $(m-1)$-element cliques, that is:
$$
V(\x_{[k,l)})\cap I_m(x)=\left( C(x,k)\setminus\{\x_{[k,l)}\}\right)\cup \left(C(x^{-1},l)\setminus\{\x_{[k,l)}\}\right).
$$
By the same reason we can write
$$
V(\x^{-1}_{[k,l)})\cap I_m(x)=\left( C(x^{-1},k)\setminus\{\x^{-1}_{[k,l)}\}\right)\cup \left(C(x,l)\setminus\{\x^{-1}_{[k,l)}\}\right).
$$
We will show that any neighbour $v$  of $\Q^*$ lying in $I_m(x)$ has $m$ neighbours in the set $N(\Q^*)\cap I_m(x)$.
Notice that $v$ has one of the following four forms: $\x^{-1}_{[i,k)}$, $\x_{[k,j)}$, $\x_{[i,l)}$ or $\x^{-1}_{[l,j)}$.
Suppose that $v=\x^{-1}_{[i,k)}$. Then $v\in C(x,k)$, so $v$ has $m-2$ neighbours in $C(x,k)\setminus\{\x_{[k,l)}\}$ and has no neighbours in $C(x^{-1},l)\setminus\{\x_{[k,l)}\}$.
Notice $v$ has two more neighbours in two remaining cliques. Namely, we have $v^{-1}\sim v$, $v^{-1}=\x_{[i,k)}\in C(x^{-1},k)$, $\x^{-1}_{[i,l)} \sim \x^{-1}_{[i,k)}=v$ and 
$\x^{-1}_{[i,l)}\in C(x,l)\setminus\{\x^{-1}_{[k,l)}\}$.  Consequently, $v$ has $m$ neighbours in the set $N(\Q^*)\cap I_m(x)$. Analogous calculations for vertices $v$ of three other forms give the same result.

If $o(x)=2$, then
$$
V(\x_{[k,l)})\cap I_m(x)=\left( C(x,k)\setminus\{\x_{[k,l)}\}\right)\cup \left(C(x,l)\setminus\{\x_{[k,l)}\}\right).
$$
Notice that each element of $ C(x,k)\setminus\{\x_{[k,l)}\}$ has exactly one neighbour in $C(x,l)\setminus\{\x_{[k,l)}\}$ (in particular if $s<k$, then $x_{[s,k)}\sim x_{[s,l)}$). Thus if $v$ is a neighbour of $\Q^*$ lying in $I_m(x)$, then $v$ belongs to exactly one of the  above two cliques and has exactly one neighbour in the second clique. This shows that $v$ has $m-1$ neighbours in the set $N(\Q^*)\cap I_m(x)$. 

Observe that if $v=\x^{\pm 1}_{[s,t)} \in \pr{s}{t}\cap N(\Q^*)$, then the neighbours of $v$ lying outside $I_m(x)$ are of the form $\y_{[s,t)}$, where $y\in G\setminus \{e,x,x^{-1}\}$. Thus if $o(x)>2$, then $v$ has $m+(|G|-3)=2m-2$ neighbours in $\N(\Q^*)$ and if $o(x)=2$, then
the degree of $v$ in $\N(\Q^*)$ is equal to $m-1+(|G|-2)=2m-2$. Consequently, the graph $\N(\Q^*)$ is reqular of degree $2m-2$.
\end{proof}

\begin{cor}\label{aut_int_clique}
Let $F\colon \G_m(G)\to \G_m(H)$ be an isomorphism of graphs such that $F(\e_G)=\e_H$, where $G$ and $H$ are finite groups 	 and let $\Q$ be a clique in $\G_m(G)$. Then $F(\Q)$ is a maximum interval (resp. maximum dispersed) clique in $\G_m(H)$ if and only if $\Q$ is a maximum interval (resp. maximum dispersed) clique in $\G_m(G)$, with the exception of the case where $|G|=3$ and $m=2$. In particular, if $(m,|G|)\neq (2,3)$, then any automorphism of the graph $\G_m(G)$ fixing $\e$ preserves the type of a maximum clique.
\end{cor}


\begin{proof}  Suppose that $(m,|G|)\neq (2,3)$. It is clear that if the graphs $\G_m(G)$ and $\G_m(H)$ are isomorphic then groups $G$ and $H$ have the same orders.
If $|G|>m+1$, then the assertion follows from the fact that maximum interval cliques are maximum cliques of $\G_m(G)$ and there are not other maximum cliques by Proposition \ref{maximum_clique}. The case $|G|=m+1>3$ follows immediately by Proposition \ref{inv_type}. 

Finally, if we assume that $|G|=|H|<m+1$, then obviously each isomorphism of $F\colon\G_m(G)\to \G_m(H)$ such that $F(\e_G)=\e_H$ induces a bijection between maximum cliques, and in particular $F$ induces  a bijection  between
sets sets $\{I_m(g)\mid g\in G^\times\}$ and $\{I_m(h)\mid h\in H^\times\}$. Therefore, none maximum interval clique can be send to a clique contained in some set $I_m(h)$, that is, $F$ must also preserve the type of maximum interval cliques. 
\end{proof}
\bigskip

We finish this section with the following example.
\begin{center}
\begin{tikzpicture}[scale=1.0]
\node at (60:3/2.1) {\footnotesize $\bullet$};
\node at (120:3/2.1) {\footnotesize $\bullet$};
\node at (180:3/2.1) {\footnotesize $\bullet$};
\node at (240:3/2.1) {\footnotesize $\bullet$};
\node at (300:3/2.1) {\footnotesize $\bullet$};
\node at (360:3/2.1) {\footnotesize $\bullet$};

\node at (30:3.6/2.1) {\footnotesize $\bullet$};
\node at (150:3.6/2.1) {\footnotesize $\bullet$};
\node at (270:3.6/2.1) {\footnotesize $\bullet$};

\node[above right] at (60:3/2.1) {\tiny $(x,e)$};
\node[above left] at (120:3/2.1) {\tiny $(x,x)$};
\node[left] at (180:3/2.1) {\tiny $(e,x)$};
\node[below left] at (240:3/2.1) {\tiny $(x^{-1},e)$};
\node[below right] at (300:3/2.1) {\tiny $(x^{-1},x^{-1})$};
\node[right] at (360:3/2.1) {\tiny $(e,x^{-1})$};

\node[above right] at (30:3.6/2.1) {\tiny $(x,x^{-1})$};
\node[above left] at (150:3.6/2.1) {\tiny $(x^{-1},x)$};
\node[below] at (270:3.6/2.1) {\tiny $(e,e)$};

\draw[thin, dashed] (0,0) circle(1.55); 

\draw[] (60:3/2.1)--(120:3/2.1)--(180:3/2.1)--(240:3/2.1)--(300:3/2.1)--(360:3/2.1)--(60:3/2.1);
\draw[] (60:3/2.1)--(240:3/2.1) 
		(180:3/2.1)--(360:3/2.1)
		(120:3/2.1)--(300:3/2.1);

\draw[very thin] (30:3.6/2.1)--(60:3/2.1)
			(30:3.6/2.1)--(120:3/2.1)
			(30:3.6/2.1)--(180:3/2.1)
			(30:3.6/2.1)--(240:3/2.1)
			(30:3.6/2.1)--(300:3/2.1)
			(30:3.6/2.1)--(360:3/2.1);

\draw[very thin] (150:3.6/2.1)--(60:3/2.1)
			(150:3.6/2.1)--(120:3/2.1)
			(150:3.6/2.1)--(180:3/2.1)
			(150:3.6/2.1)--(240:3/2.1)
			(150:3.6/2.1)--(300:3/2.1)
			(150:3.6/2.1)--(360:3/2.1);

\draw[very thin] (270:3.6/2.1)--(60:3/2.1)
			(270:3.6/2.1)--(120:3/2.1)
			(270:3.6/2.1)--(180:3/2.1)
			(270:3.6/2.1)--(240:3/2.1)
			(270:3.6/2.1)--(300:3/2.1)
			(270:3.6/2.1)--(360:3/2.1);

\node at (270:6/2.1) {\footnotesize{\rm  The graph}\  $\G_2(C_3)$};
\end{tikzpicture}\hskip1cm
\begin{tikzpicture}[scale=1.0]
\node at (60:3/2.1) {\footnotesize $\bullet$};
\node at (120:3/2.1) {\footnotesize $\bullet$};
\node at (180:3/2.1) {\footnotesize $\bullet$};
\node at (240:3/2.1) {\footnotesize $\bullet$};
\node at (300:3/2.1) {\footnotesize $\bullet$};
\node at (360:3/2.1) {\footnotesize $\bullet$};

\node at (30:3.6/2.1) {\footnotesize $\bullet$};
\node at (150:3.6/2.1) {\footnotesize $\bullet$};
\node at (270:3.6/2.1) {\footnotesize $\bullet$};

\node[above right] at (60:3/2.1) {\tiny $(x,e)$};
\node[above left] at (120:3/2.1) {\tiny $(x,x)$};
\node[left] at (180:3/2.1) {\tiny $(e,x)$};
\node[below left] at (240:3/2.1) {\tiny $(x^{-1},e)$};
\node[below right] at (300:3/2.1) {\tiny $(x^{-1},x^{-1})$};
\node[right] at (360:3/2.1) {\tiny $(e,x^{-1})$};

\node[above right] at (30:3.6/2.1) {\tiny $(x,x^{-1})$};
\node[above left] at (150:3.6/2.1) {\tiny $(x^{-1},x)$};
\node[below] at (270:3.6/2.1) {\tiny $(e,e)$};

\draw[thin, dashed] (0,0) circle(1.55); 

\draw[] (60:3/2.1)--(180:3/2.1)--(300:3/2.1)--(60:3/2.1);
\draw[] (360:3/2.1)--(120:3/2.1)--(240:3/2.1)--(360:3/2.1);

\draw[very thin] (30:3.6/2.1)--(150:3.6/2.1)--(270:3.6/2.1)--(30:3.6/2.1);

\node at (270:6/2.1) {\footnotesize {\rm  The graph}\  $\overline{\G_2(C_3)}$};
\end{tikzpicture}\hskip1cm

{\bf Fig.\,7.}
\end{center}

\begin{example}\label{m=2}{\rm 
Let $G=C_3=\{e,x,x^2\}$ and $m=2$. The graph $\G_2(C_3)$ is presented on the left of {\bf Fig.\,7} with the subgraph $\I_2(x)$ inside the dashed circle. 

\noindent Looking at the complement of $\G_2(C_3)$ on the right of {\bf Fig.\,7} one can easily see that the transposition interchanging vertices $(x,e)$ and $(x^{-1},x^{-1})$
(with all other vertices fixed) is an automorphism of $\G_2(C_3)$ sending the interval clique
$\{(e,e),(x,x),(x^{-1},x^{-1})\}$ onto a dispersed clique $\{(e,e),(x,x),(x,e)\}$.}
\end{example}

\medskip


\section{Homogeneous homomorphisms}

Recall that if $\X$ and $\Y$ are graphs with the sets of vertices $X$ and $Y$ respectively, then the map $F\colon
X\to Y$ is {\em a homomorphism}
if $F(x)$ and $F(y)$ are adjacent in $\Y$ whenever $x$ and $y$ are adjacent in $\X$.
When $\X$ and $\Y$ have no loops, which is our usual case, this definition
implies that if $x\sim y$, then $F(x) \neq F(y)$. If $F$ is a homomorphism between $\X$ and $\Y$ we will write $F\colon \X\to \Y$, even though it is really a function between $X$ and $Y$.

\medskip

Let $G$ and $H$ be groups with the identity elements $\e_G$ and $\e_H$ respectively. We will describe homomorphisms between graphs $\G_m(G)$ and $\G_m(H)$ which preserve intervals.
A  homomorphism of graphs $F\colon \G_m(G)\to \G_m(H)$ is said to be {\em homogeneous} if 
$$F(\e_G)=\e_H\  \  {\rm and}\  \   
F(\pr{k}{l})\subseteq \prh{k}{l}\  \  {\rm for\  all}\   1\leqslant k<l\leqslant m+1.
$$
For simplicity we denote the identities of $G$ and $H$ using the same symbol $e$.

Let $_iG^k$ denote the subproduct of $G^m$ of the form
$$
\underbrace{\{e\}\times\dots\times\{e\}}_{i \   \   {\rm  times }}\times \underbrace{G\times\dots\times G}_{k \   \   {\rm  times }}\times \underbrace{\{e\}\times\dots\times \{e\}}_{m-i-k \   \   {\rm  times } }.
$$ 
The elements of $_iG^k$ we will denote in the following shortened form:
$$
{_i(g_1,\dots,g_k)}=(\underbrace{e,e,\dots,e}_{i\   \   {\rm  times}},g_1,g_2,\dots,g_k,\underbrace{e,e,\dots,e}_{m-i-k\   \   {\rm  times}}). 
$$
In what follows we will use also the following condition characterizing adjacency of vertices:
\begin{equation}\label{x-y}
\x\sim\y \ \ {\rm if\ \ and\ \ only \ \ if}\ \ \x\y^{-1}\in \Ss \ \ {\rm if\ \ and\  \ only\ \ if}\ \ \y\x^{-1}\in \Ss.
\end{equation}


\begin{lem}\label{hom}
	Let $G$ be a group and $m\geqslant 4$.
	\begin{enumerate} 
	\item Let $\x={_i(g_1,g_2,\dots,g_k,e)}$, $\y={_{i}(e,g_2,\dots\dots,g_k,g_{k+1})}\in G^m$ be such that $g_1\neq e$, and $g_{k+1}\neq e$, where $k\geqslant 2$. Suppose that the element  $\ga=(a_1,a_2,\dots,a_m)\in G^m$ is such that $a_j\neq e$ for some $j\in \{1,2,\dots,i\}\cup\{i+k+2,\dots,m\}$. Then either $\x\not\sim\ga$ or $\y\not\sim\ga$ in $\G_m(G)$.
	\item Let $\x={_i(x_1,x_2,\dots,x_k)}\in {_iG^k}$ with $x_1, x_k\in G^\times$ satisfies in $\G_m(G)$:
	$$
	\x\sim{_i(a_1,a_2,\dots,a_{k-1},e)} \   \   {\rm  and }\  \x\sim {_i(e,a_2,\dots,a_{k-1},a_k)}
	$$ 
	for some $a_1, a_k\in G^\times$ and $a_2,\dots,a_{k-1}\in G$, where $k>2$. 
	
	Then either $\x={_i(a_1,a_2,\dots,a_k)}$ or $a_1=a_k$ and $\x={_i(a_1,a_1a_2,\dots,a_1a_{k-1},a_1)}$.
	\end{enumerate}
\end{lem}


\begin{proof} (1) Assume that $a_j\neq e$ for some $j\leqslant i$ and take a minimal such index. Suppose that $\x\sim\ga$ and $\y\sim\ga$ in $\G_m(G)$. Then by (\ref{x-y})
$$
\begin{array}{rcl}
\x\ga^{-1} & =& {_{j-1}(a_j^{-1},\ldots, a_i^{-1},g_1a_{i+1}^{-1},g_2a_{i+2}^{-1},\ldots,g_ka_{i+k}^{-1}, a_{i+k+1}^{-1},\ldots,a_m^{-1})}\in \Ss\\[6pt]
\y\ga^{-1} & =& {_{j-1}(a_j^{-1},\ldots, a_i^{-1},a_{i+1}^{-1},g_2a_{i+2}^{-1},\ldots,g_ka_{i+k}^{-1}, g_{k+1}a_{i+k+1}^{-1},\ldots,a_m^{-1})}\in \Ss\\

\end{array}
$$
Since the $j$-th entries of both vertices $\x\ga^{-1}$ and $\y\ga^{-1}$  are the same and not equal $e$, all nonidentity entries are equal. Now comparing the $i+1$-th entries we see that they differ because $g_1\neq e$. So either $g_1a_{i+1}^{-1}=e$ or $a_i^{-1}=e$. In any case all entries on the right from this place are equal $e$, since they are equal in both vertices. But the $i+k+1$-entries differs from each other so one of them is not equal $e$, a contradiction.   
The proof for the case $i+j+2\leqslant j$ is analogous. 

\medskip

(2) By part (1) without loss of generality we may assume that $i=0$ and $k=m$. As before, by (\ref{x-y}) we have  
$$
\begin{array}{rcl}
\x\cdot(a_1,a_2,\dots,a_{m-1},e)^{-1} & =& (x_1a_1^{-1},x_2a_2^{-1},\ldots,x_{m-1}a_{m-1}^{-1},x_m)\in \Ss,\\[6pt]
\x\cdot (e,a_2,\dots,a_{m-1},a_m)^{-1} & =& (x_1,x_2a_2^{-1},\ldots,x_{m-1}a_{m-1}^{-1},x_ma_{m}^{-1})\in \Ss.
\end{array}
$$
If $x_1a_1^{-1}\neq e$, then all entries of the first vertex are equal because $x_m\neq e$. In particular $x_1a_1^{-1}=x_2a_2^{-1}\neq e$. This implies that the first two entries of the second element are also equal because $x_1\neq e$. So $x_1=x_1a^{-1}$ which gives $g_1=e$, a contradiction. Analogously we get a contradiction if we assume that $x_ma_m^{-1}\neq e$. Therefore  we may assume that $a_1=x_1$ and $a_m=x_m$.

Now suppose that $x_2a_2^{-1}\neq e$. Then similarly as in the previous case, all entries but the first one 
of the first vertex are equal $x_m$. We have also $x_1=x_2a_2^{-1}$ in the second vertex. Therefore $x_1=x_m=a_1,\ x_2=a_1a_2,\ \ldots ,x_{m-1}=a_1a_{m-1}$ that is $\x=(a_1,a_1a_2,\ldots,a_1a_{m-1},a_1)$.

Finally let us assume that $x_2a_2^{-1}=e$. Then analyzing the second vertex we see that all its entries but the first one are equal $e$. This means $x_2=a_2,\ \ldots,\  x_m=a_m$, i.e. $\x=(a_1,\dots,a_m)$. This ends the proof.
\end{proof}


\begin{thm}\label{morphism}
	Let $G$, $H$ be groups and $m>1$. Then every homogeneous graph homomorphism (isomorphism) $F\colon \G_m(G)\to \G_m(H)$ is induced by a group monomorphism (isomorphism), that is
 	$$
 	F(g_1,g_2,\dots,g_m)=(f(g_1),f(g_2),\dots,f(g_m))
 	$$
 	for some monomorphism (isomorphism) of groups $f\colon G\to H$.
\end{thm}


\begin{proof}
First we consider the case $m=2$. We will make use of the following properties of $\G_2(G)$:

\noindent For $a,x,y\in G$, and $x\neq y$
\begin{itemize}
\item[(a)]  if $(x,y)\sim (a,a)$, then either $a=x$ or $a=y$
\item[(b)] if  $(x,y)\sim (e,a)$, then either $a=y$ or $a=x^{-1}y$.
\end{itemize}

Suppose $F\colon \G_2(G)\to \G_2(H)$ is a homogeneous homomorphism. Since $F(\pr{1}{2})\subseteq \prh{1}{2}$, there is a map $f\colon G\to H$ such that $F(g,e)=(f(g),e)$ for all $g\in G$. Moreover, if $a\neq b$ are elements in $G^\times$, then $(a,e)\sim (b,e)$ and hence $(f(a),e)\sim (f(b),e)$. In particular, $f(a)\neq f(b)$ as $\G_2(H)$ does not contain loops. Thus the map $f\colon G\to H$ is injective. In case when $F$ is an isomorphism of graphs the map $f$ is bijective.
We claim that $F(g,g)=(f(g),f(g))$ and $F(e,g)=(e,f(g))$ for all $g\in G$. Indeed if $g\neq e$, then $F(g,g)= (x,x)$ for some $x\in H$. Since $(g,e)\sim (g,g)\sim (e,e)$, we have $(x,x)=F(g,g)\sim F(g,e)=(f(g),e)$.
By (a) it follows that either $x=f(g)$ or $x=e$. But $F(e,e)=(e,e)$ and $(g,g)\sim (e,e)$, so $x\neq e$ and hence $x=f(g)$.
Similarly $F(g,e)=(e,y)$, for some $y\in H^\times$. Since $(g,g)\sim (e,g)$, we have $(f(g),f(g))=F(g,g)\sim F(e,g)=(e,y)$. By (b) it follows that $y=f(g)$, so $F(e,g)=(e,f(g))$.

For $g\in G^\times$ the elements  $(e,e),(g,e),(e,g^{-1})$ form a clique in $\G_2(G)$, so 
$$
\{F(e,e),F(g,e), F(e,g^{-1})\}=\{(e,e), (f(g),e),(e,f(g^{-1})\}
$$  
is a clique in $\G_2(H)$. 
{\small
\begin{center}
\begin{tikzpicture}[scale=0.7]
\node at (1,0) {$\bullet$};
\node at (-1,0) {$\bullet$};
\node at (3,0) {$\bullet$};
\node at (2,1.73) {$\bullet$};
\node[below] at (1,0) {$(e,e)$};
\node[below] at (-1,0) {$(e,e)$};
\node[right] at (3,0) {$(e,f(g^{-1}))$};
\node[right] at (2,1.73) {$(f(g),e)$};
\draw (3,0)--(1,0)--(2,1.73)--(3,0);

\node at (-3,0) {$\bullet$};
\node at (-2,1.73) {$\bullet$};
\node[left] at (-3,0) {$(e,g^{-1})$};
\node[left] at (-2,1.73) {$(g,e)$};
\draw (-3,0)--(-1,0)--(-2,1.73)--(-3,0);
\draw[->, dashed, thick] (-1.9,1.8) to [bend left] node [above] {$F$} (1.9,1.8);
\draw[->, dashed, thick] (-0.9,0.1) to [bend left] node [above] {$F$} (0.9,0.1);
\draw[->, dashed, thick] (-2.9,-0.1) to [bend right] node [below] {$F$} (2.9,-0.1);
\end{tikzpicture}

{\bf Fig.\,8.}
\end{center}}
\medskip

\noindent By (b) it follows that $f(g^{-1})={f(g)}^{-1}.$

Now let $a,b\in G^\times$ be such that $ab\neq e$. Take $x,y\in H$ such that $F(ab,a)=(x,y)$.
Clearly $(ab,ab)\sim (ab,a)$, so $F(ab,ab)\sim F(ab,a)$, that is $(f(ab),f(ab))\sim (x,y)$. By (a) either $x=f(ab)$, or $y=f(ab)$.  

{\small
\begin{center}
\begin{tikzpicture}[scale=0.65]
\node at (0,0) {$\bullet$};
\node at (1,2) {$\bullet$};
\node at (1,-2) {$\bullet$};

\node at (-1,2) {$\bullet$};
\node at (-1,-2) {$\bullet$};
\node[left] at (0,0) {$(ab,a)$};
\node[right] at (1,2) {$(b,e)$};
\node[right] at (1,-2) {$(e,b^{-1})$};
\node[left] at (-1,2) {$(a,a)$};
\node[left] at (-1,-2) {$(ab,e)$};
\draw (1,2)--(0,0)--(-1,2);
\draw (1,-2)--(0,0)--(-1,-2);

\node at (8,0) {$\bullet$};
\node at (9,2) {$\bullet$};
\node at (9,-2) {$\bullet$};
\node at (7,2) {$\bullet$};
\node at (7,-2) {$\bullet$};
\node[right] at (8,0) {$(x,y)=F(ab,a)$};
\node[right] at (9,2) {$(f(b),e)$};
\node[right] at (9,-2) {$(e,f(b)^{-1})$};
\node[left] at (7,2) {$(f(a),f(a))$};
\node[left] at (7,-2) {$(f(ab),e)$};
\draw (9,2)--(8,0)--(7,2);
\draw (9,-2)--(8,0)--(7,-2);

\draw[->, dashed, thick] (0.1,0.1) to [bend left] node [below] {$F$} (7.9,0.1);
\draw[->, dashed, thick] (-0.9,2.1) to [bend left] node [above] {$$} (6.9,2.1);
\draw[->, dashed, thick] (-0.9,-2.1) to [bend right] node [above] {$$} (6.9,-2.1);
\draw[->, dashed, thick] (1.1,2.1) to [bend left] node [above] {$$} (8.9,2.1);
\draw[->, dashed, thick] (1.1,-2.1) to [bend right] node [above] {$$} (8.9,-2.1);

\end{tikzpicture}

{\bf Fig.\,9}
\end{center}}
\medskip

\noindent Suppose first that $y=f(ab)$. Then $(ab,a)\sim (a,a)$ implies $F(ab,a)=(x,f(ab))\sim (f(a),f(a))=F(a,a)$. Since $x\neq f(a,b)\neq f(a)$ we have $x=f(a)$ by (a). Now $(ab,a)\sim (ab,e)$ thus $F(ab,a)=(f(a),f(ab))\sim (f(ab),e)$ which implies $f(ab)^{-1}f(a)=f(ab)$ and then $f(a)=f(ab)^2$ by a condition symmetric to (b). Further $(ab,a)\sim (b,e)$ therefore $(f(a),f(ab))\sim (f(b),e)$ implies $f(ab)^{-1}f(a)=f(b)$ and so $f(a) =f(ab)f(b)$ by (b). Earlier we obtained $f(a)=f(ab)^2$, therefore $f(b)=f(ab)$, a contradiction. 

So we may assume that $x=f(ab)$. Since $(ab,a)\sim (a,a)$, we see that $F(ab,a)=(f(ab),y)\sim (f(a),f(a))$. Therefore by (a) $y=f(a)$, because $f(ab)\neq f(a)$. So we obtained that $F(ab,a)=(f(ab),f(a))$ and if $a$ and $b$ run independently the set $G^{\times}$ with additional condition $ab\neq e$ the vertex $(ab,a)$ run all elements of the set $G^2\setminus(\pr{1}{2}\cup\pr{2}{3}\cup\pr{1}{3})$. Therefore for arbitrary $g,h\in G$ we have $F(g,h)=(f(g),f(h)).$

Now we show that $f$ is a homomorphism of groups. Since $(ab,a)\sim (e,b^{-1})$, we see that $F(ab,a)=(f(ab),f(a))\sim (e,f(b)^{-1})$. Now, by the assumption $ab\neq e$, so $f(a)\neq f(b)^{-1}$ and then 
by (b) $f(ab)^{-1}f(a)=f(b)^{-1}$. Therefore
$$
f(ab)=f(a)f(b) \   \   {\rm  for\  all}\  a,b\in G.
$$

We will now consider the general case. 
We will show by induction that there exists  a group monomorphism  $f\colon G\to H$  such that for any $k\geqslant 2$ and $i\leqslant m-k$
\begin{equation}\label{ind}
F(_iG^k)\subseteq {_iH^k} \   \   {\rm  and}\  \   F( {_i(g_1,g_2,\dots,g_k))}= {_i(f(g_1),f(g_2),\dots,f(g_k))},
\end{equation}
where $_i(g_1,g_2,\dots,g_k)\in {_iG^k}$. 

The case $k=2$ is almost done by the case $m=2$ considered above. We only need to know that $F(_iG^2)\subseteq {_iH^2}$. Take ${_i(g_1,g_2)}\in {_iG^2}$. If $g_1=g_2$, then $F({_i(g_1,g_2)})\in {_iH^2}$, because $F$ is homogeneous. Suppose that $g_1\neq g_2$. Then we have a three-element clique in $\G_m(G)$:
$\{{_i(g_1,g_2)}, {_i(g_1,g_1)}, {_i(g_2,g_2)}\}$. Thus we have a three-element clique in $\G_m(H)$:
$$
\{F({_i(g_1,g_2)}), {_i(h_1,h_1)}, {_i(h_2,h_2)}\},
$$ 
where $F({_i(g_1,g_1)})={_i(h_1,h_1)}$, $F({_i(g_2,g_2)})={_i(h_2,h_2)}$.
Now it can be easily proved that 
$$
F({_i(g_1,g_2)})\in {_iH^2}.
$$
We demonstrate it for the case $m=3.$

Let  $F(e,g_1,g_2)=(x,y,z)$, where $x\neq e$. Then $(x,y,z)\sim (e,f(g_1),f(g_1)$ and $(x,y,z)\sim (e,f(g_2),f(g_2)),$ i.e. $(x,yf(g_1)^{-1},zf(g_1)^{-1})\in \Ss$ and $(x,yf(g_2)^{-1},zf(g_2)^{-1})\in \Ss$. If $zf(g_2)^{-1}$ $\neq e$, then $zf(g_2)^{-1}=yf(g_2)^{-1}=x$. But $(e,g_1,g_2)\sim (e,e,g_2)$, thus $(x,y,z)\sim (e,e,f(g_2))$ that is $(x,y,zf(g_2)^{-1})\in \Ss$. Hence $zf(g_2)^{-1}=y=x$ and then $f(g_2)=e$, a contradiction. Therefore $z=f(g_2)$. This means that $zf(g_1)^{-1}\neq e$, and similarly as before, $(x,yf(g_1)^{-1},zf(g_1)^{-1})\in \Ss$, which implies $y=xf(g_1)$ and $z=xf(g_1)$. Consequently $y=z$. We have also $(e,g_1,g_2)\sim (e,g_1,e)$, so $(x,y,z)\sim(e,f(g_1),e$ and then $(x,yf(g_1)^{-1},z)\in \Ss$. Thus $y=zf(g_1)$, which finally gives $f(g_1)=e$, a contradiction. In these considerations we used the assumption $x\neq e$. So $x=e$ and $F(\{e\}\times G\times G)\subseteq \{e\}\times H\times H$.

By the considered case $m=2$ there is a group monomorhism $f_i\colon G\to H$ such that 
$$F({_i(g_1,g_2)})=
{_i(f_i(g_1),f_i(g_2))}.
$$
Since $_iG^2\cap {_{i+1}G^2}={_{i+1}G^1}$, we see that $f_i=f_{i+1}$ for all $i$. This finishes the proof for $k=2$.

Suppose that $k\geqslant 2$ and (\ref{ind}) holds for all $i\leqslant m-k$. 
Take $g\in G^\times$ and ${_i(x_1,x_2,\dots,x_k)}\in {_{i}G^k}$ with $x_1\neq e$ and $gx_1\neq e$. Notice that
$${_i(gx_1,gx_2,\dots,gx_k,g)}\sim {_i(gx_1,gx_2,\dots,gx_k,e)}$$
and  
$$
{_i(gx_1,gx_2,\dots,gx_k,g)}\sim {_i(e,gx_2,\dots,gx_k,g)}.
$$
Applying (\ref{ind}) one obtains
$$
F({_i(gx_1,gx_2,\dots,gx_k,g)})\sim {_i(f(g)f(x_1),f(g)f(x_2),\dots,f(g)f(x_k),e)}
$$
and
$$
F({_i(gx_1,gx_2,\dots,gx_k,g)})\sim {_i(e,f(g)f(x_2),\dots,f(g)f(x_k),f(g))}
$$
Part 1. of Lemma \ref{hom} gives us that $F({_i(gx_1,gx_2,\dots,gx_k,g)})\in {_iG^{k+1}}$ and part 2.
implies 
$$
F({_i(gx_1,gx_2,\dots,gx_k,g)})={_i(f(g)f(x_1),f(g)f(x_2),\dots,f(g)f(x_k),f(g))}.
$$
Substituting above $g=g_{k+1}$, $x_1=g_{k+1}^{-1}g_1$, $x_2=g_{k+1}^{-1}g_1$,$\dots$, $x_k=g_{k+1}^{-1}g_k$ we obtain
that
$$
F( {_i(g_1,g_2,\dots,g_k,g_{k+1}))}= {_i(f(g_1),f(g_2),\dots,f(g_k),f(g_{k+1}))}.
$$
This finishes the proof.
\end{proof}


\begin{cor}
Let $G$, $H$ be groups and $m\geqslant 2$. Then the graphs  $\G_m(G)$ and $\G_m(H)$ are isomorphic under a homogeneous isomorphism if and only if the groups $G$ and $H$ are isomorphic.
\end{cor}


\begin{proof}
If $\varphi\colon G\to H$ is an isomorphism of groups, then certainly the map $\Phi\colon G^m\to H^m$, 
$$
(g_1,g_2,\dots,g_m)\mapsto ({g_1}^\varphi,{g_2}^\varphi,\dots,{g_m}^\varphi)
$$ 
induces a homogeneous isomorphism between graphs $\G_m(G)$ and $\G_m(H)$.

If $F\colon \G_m(G)\to \G_m(H)$ is a homogeneous graph isomorphism, then by Theorem \ref{morphism} there exists an isomorphism of groups $f\colon G\to H$, such that 
$$
F(g_1,g_2,\dots,g_m)=(f(g_1),f(g_2),\dots,f(g_m)).
$$ 
Thus the groups $G$ and $H$ are isomorphic.
\end{proof}

 
\section{Automorphisms and isomorphisms} 

It is well known that in any Cayley graph the right transfers form a group of vertex-transitive group of automorphisms of the graph. We denote this group by $\mathbf{T}_m(G)$: 
$$\mathbf{T}_m(G)=\{T_{\g}:\ \g\in G^m\}, \  \ T_{\g}:G^m\rightarrow G^m, \ \ \x^{T_{\g}}= \x\g,\ \ {\rm for}\ \ \x\in G^m.$$ 

In the previous section we defined homogeneous homomorphism. As a consequence of this definition by a homogeneous automorphisms we mean automorphisms fixing all intervals. Since they are determined by automorphism of the group $G$
(by Theorem \ref{morphism}), we denote them by $\mathbf{Aut}_m(G)$ and we use the same letters for denoting automorphisms of $G$ and homogeneous automorphisms of $\G_m(G)$: If $f\in {\rm Aut}(G)$ and $\x=(x_1,\ldots,x_m)\in G^m$, then
$$\x^f=(f(x_1),\ldots,f(x_m)).$$
It is clear that for $\g\in G^m$ and $f\in \mathbf{Aut}_m(G)$
\begin{equation}\label{fTf}
f^{-1}T_{\g}f=T_{\g^f}
\end{equation}
that is $\mathbf{Aut}_{m}(G)$ normalizes $\mathbf{T}_m(G)$. 

In this section we give a description of the group $\mathbf{Aut}(\G_m(G))$ of all automorphisms of $\G_m(G)$. The cases of abelian and non-abelian groups appear to be essentially different. We begin with a simpler one.


\begin{lem}\label{abelian}
Let $G$ be an abelian group and $m\geqslant 2$. For $i=1,2,\dots,m$\ let\ $\gamma_i\colon G^m\to G^m$ be the mappings given by
$$
(g_1,g_2,\dots,g_m)^{\gamma_i}=(g_1,\dots,g_{i-1},g_{i-1}g_i^{-1}g_{i+1},g_{i+1},\dots,g_m),
$$ 
(we assume $g_0=g_{m+1}=e$). Then
\begin{enumerate}
\item all $\gamma_i$ are automorphisms of the group $G^m$ of order $2$, satisfying the condition 
$\mathcal{S}^{\gamma_i}=\mathcal{S}$ and then all they are automorphisms of the graph $\G_m(G)$.

\item for $|i-j|>1$, $(1\leqslant i,j\leqslant m)$, we have $\gamma_i\gamma_j=\gamma_j\gamma_i$. 

\item for all $i, j$, $(1\leqslant i,\,j\leqslant m,\ i+j\leqslant m)$,\ $\gamma_i\gamma_{i+1}\dots\gamma_{i+j}$ is an automorphism of $\G_m(G)$ of order $j+2$, in particular the automorphisms $\gamma_i\gamma_{i+1}$ have order $3$ and $\gamma_1\dots\gamma_m=\omega$ has order $m+1$.

\item the subgroup  $\mathbf{\Gamma}_m=\langle\gamma_1,\gamma_2,\dots,\gamma_m\rangle $ of $\mathbf{Aut}(\G_m(G))$  is isomorphic to the symmetric group $S_{m+1}$ of degree $m+1$.
\end{enumerate}

\end{lem}
\begin{proof}
\noindent (1) By the definition of $\gamma_i$ we easily see that $\gamma_i^2=1$. Furthermore, again by the definition of $\gamma_i$ we have $\x_{[i,i+1)}^{\gamma_i}=\x_{[i,i+1)}^{-1}$ and \vspace{6pt}\par
\begin{equation}\label{gami}
{\rm
\begin{tabular}{ll}
	for $k<i$, & $\x_{[k,i)}^{\gamma_i}=\x_{[k,i+1)}$ and $\x_{[k,i+1)}^{\gamma_i}=\x_{[k,i)}$, \\[6pt]
	for $i<s$, & $\x_{[i,s+1)}^{\gamma_i}=\x_{[i+1,s+1)}$ and $\x_{[i+1,s+1)}^{\gamma_i}=\x_{[i,s+1)}$. \\[6pt]
\end{tabular}}
\end{equation}\par
\noindent For all other $1\leqslant k<l\leqslant m+1$ the elements $\x_{[k,l)}$ are fixed points of $\gamma_i$. Hence $\mathcal{S}^{\gamma_i}=\mathcal{S}$.
Now, by abelianity of $G$, it is easily seen that $\gamma_i$ is an automorphism of the group $G^{m}$:
$$\begin{array}{rcl}
(\g\h)^{\gamma_i} & = & (g_1h_1,\dots,g_{i-1}h_{i-1}, g_ih_i,g_{i+1}h_{i+i}\dots,g_mh_m)^{\gamma_i}\\
& = & (g_1h_1,\dots,g_{i-1}h_{i-1},(g_{i-1}h_{i-1}) (g_i^{-1}h_{i}^{-1})(g_{i+1}h_{i+1}),g_{i+1}h_{i+1},\dots,g_mh_m)\\
& = & (g_1,\dots,g_{i-1},g_{i-1} g_i^{-1}g_{i+1},g_{i+1},\dots,g_m)(h_1,\dots,h_{i-1},h_{i-1}h_{i}^{-1}h_{i+1},h_{i+1},\dots,h_m)\\
& = & \g^{\gamma_i}\h^{\gamma_{i}}
\end{array}
$$
Therefore, if $\g,\h\in G^m$ are such that $\g\sim\h$ and $\x=\x_{[k,s+1)}$ is such that $\h=\x\g$, then $\h^{\gamma_i}=(\x\g)^{\gamma_i}=\x^{\gamma_i}\g^{\gamma_i}$. Hence $\g^{\gamma_i}\sim\h^{\gamma_i}$.

\medskip

\noindent (2) The proof is obvious.

\medskip

\noindent (3) Suppose first that $j=1$. For $1\leqslant i\leqslant m$ the automorphisms $\eta=\gamma_i\gamma_{i+1}$ have order $3$. In fact:
$$
\begin{array}{rcl}
\g^{\gamma_i\gamma_{i+1}} & = & (g_1,\dots,g_{i-1},g_{i-1}g_i^{-1}g_{i+1}, g_{i-1}g_{i}^{-1}g_{i+2},g_{i+2},\dots,g_m),\\ [8pt]
\end{array}
$$
and
$$
\begin{array}{rcl}
\g^{(\gamma_i\gamma_{i+1})^2} & = & (g_1,\dots,g_{i-1},g_{i-1}g_{i+1}^{-1}g_{i+2}, g_{i}g_{i+1}^{-1}g_{i+2},g_{i+2},\dots,g_m).\\ [8pt]
\end{array}
$$
We have also 
$$
\begin{array}{rcl}
\g^{\gamma_{i+1}\gamma_{i}} & = & (g_1,\dots,g_{i-1},g_{i-1}g_{i+1}^{-1}g_{i+2}, g_{i}g_{i+1}^{-1}g_{i+2},g_{i+2},\dots,g_m),\\ [8pt]
\end{array}
$$
Then $(\gamma_i\gamma_{i+1})^2=\gamma_{i+1}\gamma_i$, which means that $(\gamma_i\gamma_{i+1})^3=1$.
\medskip

For the general case we obtain
$$
\begin{array}{rcl}
\g^{\gamma_i\dots\gamma_{i+j}} & = & (g_1,\dots,g_{i-1},g_{i-1}g_i^{-1}g_{i+1}, g_{i-1}g_{i}^{-1}g_{i+2},\dots,g_{i-1}g_i^{-1}g_{i+j+1},g_{i+j+1},\dots,g_m),\\ [8pt]

\g^{(\gamma_i\dots\gamma_{i+j})^2}  & = &  
(g_1,\dots,g_{i-1},g_{i-1}g_{i+1}^{-1}g_{i+2}, g_{i-1}g_{i+1}^{-1}g_{i+3},\dots,\\ & & \dots,g_{i-1}g_{i+1}^{-1}g_{i+j+1},g_{i}g_{i+1}^{-1}g_{i+j+1},g_{i+j+1},\dots,g_m),\\ [8pt]

\g^{(\gamma_i\dots\gamma_{i+j})^3}  &= & 
(g_1,\dots,g_{i-1},g_{i-1}g_{i+2}^{-1}g_{i+3}, g_{i-1}g_{i+2}^{-1}g_{i+4},\dots,\\ 
& & \dots,g_{i}g_{i+2}^{-1}g_{i+j+1},g_{i+1}g_{i+2}^{-1}g_{i+j+1},g_{i+j+1},\dots,g_m),\\ [8pt]

& & \dots\dots \dots \dots\dots \dots \dots\dots \dots  \\ [8pt]

\g^{(\gamma_i\dots\gamma_{i+j})^{j+1}} & = & 
(g_1,\dots,g_{i-1},g_{i-1}g_{i+j}^{-1}g_{i+j+1}, g_{i}g_{i+j}^{-1}g_{i+j+1},\dots,\\ & & \dots,g_{i+j-2}g_{i+j}^{-1}g_{i+j+1},g_{i+j-1}g_{i+j}^{-1}g_{i+j+1},g_{i+j+1},\dots,g_m).\\ [8pt]
\end{array}
$$
Now, one can easily see that $\g^{(\gamma_i\dots\gamma_{i+j})^{j+2}}=1.$ 
\medskip

\noindent (4) It is well known that the symmetric group $S_{m+1}$ can be described as the group isomorphic to 
\begin{equation}\label{gen_rel}
\langle\sigma_1,\dots,\sigma_m|\ \sigma_i^2=1, \ (\sigma_i\sigma_{i+1})^3=1, \ \sigma_i\sigma_j=\sigma_j\sigma_i \ \ \ {\rm for}\ \ |i-j|>1,\ i,j=1,\dots,m \rangle.
\end{equation}
Since $\mathbf{\Gamma}_m$ is generated by elements $\gamma_i$, $i=1,\dots,m$, satisfyjng the same relations as in (\ref{gen_rel}), $\mathbf{\Gamma}_m$ is a homomorphic image of $S_{m+1}$. More precisely, for $i=1,\dots,m$ let $\sigma_i=(i,i+1)$ be a transposition of neighbour elements in the set $\{1,\dots,m+1\}$. Then $\sigma_i$-s satisfy all the relations (\ref{gen_rel}) and the correspondence $\sigma_i\mapsto\gamma_i$ can be extended to an epimorphism from $S_{m+1}$ onto $\mathbf{\Gamma}_m$. 
The possible homomorphic images of $S_{m+1}$ have order $1$ or $2$ or are isomorphic to $S_{m+1}$. The group $\mathbf{\Gamma}_m$ has more than $2$ elements, so the above map extends to an isomorphism of $S_{m+1}$ onto $\mathbf{\Gamma}_m$.
\end{proof}


\begin{lem} For $i=1,\dots,p$ we define automorphisms $\tau_i$ in the following way depending on whether $m=2p-1$ or $m=2p$

$$\begin{array}{rclcrcl}
m&=&2p-1&&m&=&2p \\ [6pt]\hline 
\tau_p & = & \gamma_p,&& \tau_p & = & \gamma_{p+1}\gamma_p\gamma_{p+1},\\
\tau_{p-1} & = & \gamma_{p-1}\gamma_{p+1}\tau_p\gamma_{p+1}\gamma_{p-1},&& \tau_{p-1} & = & \gamma_{p-1}\gamma_{p+2}\tau_p\gamma_{p+2}\gamma_{p-1},\\
\tau_{p-2} & = & \gamma_{p-2}\gamma_{p+2}\tau_{p-1}\gamma_{p+2}\gamma_{p-2},&& \tau_{p-2} & = & \gamma_{p-2}\gamma_{p+3}\tau_{p-1}\gamma_{p+3}\gamma_{p-2},\\
\dots & & \dots && \dots & & \dots \\
\tau_1 & = & \gamma_1\gamma_{2p}\tau_2\gamma_{2p}\gamma_1&&\tau_1 & = & \gamma_1\gamma_{2p+1}\tau_2\gamma_{2p+1}\gamma_1.\\
\end{array}$$
Then 
\begin{enumerate}
\item for $1\leqslant i,j\leqslant p$, $\tau_i\tau_j=\tau_j\tau_i$.
\item $\tau_1\dots\tau_p=\varepsilon\tau$, where $\varepsilon\colon G^m\to G^m$ is defined by $\g^{\varepsilon}=\g^{-1}$ and $\tau\colon G^m\to G^m$ given by $$(g_1,g_2,\dots,g_m)^\tau=(g_m,g_{m-1},\dots,g_2,g_1)$$ are commuting automorphisms  of order two of  $\G_m(G)$. 
\item $\langle \mathbf{\Gamma}_m, \tau \rangle=\mathbf{\Gamma}_m\times\langle\varepsilon\rangle$, provided $G$ is not an elementary abelian $2$-group.
\end{enumerate}
\end{lem}
\begin{proof}
\noindent (1) As we noticed in the proof of the previous lemma, the correspondence $(i,i+1)\mapsto \gamma_i$ extends to an isomorphism of $S_{m+1}$ onto $\mathbf{\Gamma}_m$. It follows from standard calculation that in this isomorphism the transpositions $(1,m+1), \ (2,m),\ \dots,\ (p,m-p+1)$ are mapped onto $\tau_1,\ \tau_2,\ \dots,\ \tau_p$ respectively. These transpositions have disjoint supports, so they commute and because of that their immages also commute.   

\medskip 

\noindent (2) Suppose first $m=2p-1$. For an arbitrary $\g\in G^m$ we have
$$\begin{array}{rcl}
\g^{\tau_p}& = &  (g_1,\dots,g_{p-1},g_{p-1}g_p^{-1}g_{p+1},g_{p+1},\dots,g_{2p-1}) \\[6pt] 
\g^{\tau_p\tau_{p-1}} & = & (g_1,\dots,g_{p-2},g_{p-2}g_{p+1}^{-1}g_{p+2},g_{p-2}g_p^{-1}g_{p+2},g_{p-2}g_{p-1}^{-1}g_{p+2},g_{p+2}\dots,g_{2p-1}) \\[6pt]
& & \dots\dots\dots \dots\dots\dots\dots\dots\dots \\ [6pt]
\g^{\tau_p\tau_{p-1}\cdots\tau_{p-i+1}} & =  & 
(g_1,\dots,g_{p-i},g_{p-i}g_{p+i-1}^{-1}g_{p+i},\dots,g_{p-i}g_p^{-1}g_{p+i},\dots, \\
& & \dots,g_{p-i}g_{p-i+1}g_{p+i},g_{p+i},\dots,g_{2p-1}) \\[0pt]
\end{array}$$
and so on. It follows from this schema that
$$\begin{array}{rcl}
\g^{\tau_p\cdots\tau_1}& = & (g_0g_{2p-1}^{-1}g_{2p},\dots,g_0g_{p+1}^{-1}g_{2p},g_{0}g_p^{-1}g_{2p},g_0g_{p-1}^{-1}g_{2p},\dots,g_0g_{1}^{-1}g_{2p}) \\[4pt]
& = & (g_{2p-1}^{-1},\dots,g_{p+1}^{-1},g_p^{-1},g_{p-1}^{-1},\dots,g_{1}^{-1}) \\[4pt]
& = & \g^{\tau\varepsilon}
\end{array}$$
because by our convention $g_0=g_{m+1}=e.$ Hence we get 
$\tau_{p}\tau_{p-1}\cdots\tau_2\tau_1=\tau\varepsilon$. 
\smallskip

Now let $m=2p$. Then
$$\begin{array}{rcl}
\g^{\tau_p} & = & (g_1,\dots,g_{p-1},g_{p-1}g_{p+1}^{-1}g_{p+2},g_{p-1}g_p^{-1}g_{p+2},g_{p+2},\dots,g_{2p}) \\[6pt] 
\g^{\tau_p\tau_{p-1}}& = & (g_1,\dots,g_{p-2},g_{p-2}g_{p+2}^{-1}g_{p+3},g_{p-2}g_{p+1}^{-1}g_{p+3},g_{p-2}g_p^{-1}g_{p+3},\\
 &  &  g_{p-2}g_{p-1}^{-1}g_{p+3},g_{p+3}\dots,g_{2p}) \\ [8pt]
 &  &  \dots\dots\dots \dots\dots\dots\dots\dots\dots\\
\g^{\tau_p\tau_{p-1}\dots\tau_{p-i+1}} & =  &
(g_1,\dots,g_{p-i},g_{p-i}g_{p+i}^{-1}g_{p+i+1},\dots,g_{p-i}g_p^{-1}g_{p+i+1},g_{p-i}g_{p-1}^{-1}g_{p+i+1},\\
& & \dots ,g_{p-i}g_{p-i+1}g_{p+i+1},g_{p+i+1},\dots,g_{2p}) \\[0pt]
\end{array}$$
Again, similarly as in the previous case we have
$$\begin{array}{rcl}
\g^{\tau_p\cdots\tau_1}& = & (g_0g_{2p}^{-1}g_{2p+1},\dots,g_0g_{p+1}^{-1}g_{2p+1},g_{0}g_p^{-1}g_{2p+1},g_0g_{p-1}^{-1}g_{2p+1},\dots,g_0g_{1}^{-1}g_{2p+1}) \\[4pt]
& = & (g_{2p}^{-1},\dots,g_{p+2}^{-1},g_{p+1}^{-1},g_p^{-1},g_{p-1}^{-1},\dots,g_{1}^{-1}) \\[4pt]
& = & \g^{\tau\varepsilon}
\end{array}$$
\medskip

\noindent (3) Assume that $G$ is not an elementary abelian $2$-group. It is clear from (2) that $\varepsilon\in\langle \mathbf{\Gamma}_m,\tau\rangle$ and $\tau\in\langle \mathbf{\Gamma}_m,\varepsilon\rangle$, therefore $\langle \mathbf{\Gamma}_m,\tau\rangle=\langle \mathbf{\Gamma}_m,\varepsilon\rangle $. Since $\varepsilon$ commutes with all elements of $\mathbf{\Gamma}_m$ it does not belong to $\mathbf{\Gamma}_m$ as it is isomorphic to $S_{m+1}$. Hence $$\langle\mathbf{\Gamma}_m,\varepsilon\rangle = \mathbf{\Gamma}_m\times\langle\varepsilon\rangle$$
which ends the proof.
\end{proof}


We are  now ready to determine the group of automorphisms $\mathbf{Aut}(\G_m(G))$, in the case when the group $G$ is abelian.

\begin{thm}\label{stabilizer_of_e}
Let $G$ be an abelian group. If either \begin{enumerate}
\item[{\rm (a)}] $m>3$, or
\item[{\rm (b)}] $m=3$ and $G$ is of exponent bigger than $2$, or
\item[{\rm (c)}] $m=2$ and $|G|\neq 3$,
\end{enumerate}
then 
\begin{enumerate}
\item the stabilizer of $\e\in G^m$ in the automorphism group $\mathbf{Aut}(\G_m(G))$ is equal to
$$\mathbf{Aut}_m(G)\times \mathbf{\Gamma}_m \simeq {\mathbf{Aut}}(G)\times S_{m+1};$$
\item the group of all automorphisms of the graph $\G_m(G)$ is equal to
$$\mathbf{Aut}(\G_m(G))=\mathbf{T}_m\rtimes \big(\mathbf{Aut}_m(G)\times \mathbf{\Gamma}_m\big) \simeq G^m\rtimes \big(\mathbf{Aut}(G)\times S_{m+1}\big).$$
\end{enumerate}
\end{thm}
\begin{proof}
\noindent (1) It follows immediately from the definitions that for all $f\in\mathbf{Aut}_m(G)$ and all $\gamma_i$, $1\leqslant i\leqslant m$, 
$f\gamma_i=\gamma_i f$. Therefore 
$$\langle \mathbf{Aut}_m(G), \mathbf{\Gamma}_m\rangle =\mathbf{Aut}_m(G)\times \mathbf{\Gamma}_m.$$
Let $\alpha$ be an automorphism of $\G_m(G)$ such that $\e^{\alpha}=\e$. Hence $V(\e)^{\alpha}=V(\e)$ and then $\alpha$ induces an automorphism $\overline{\alpha}$ of the graph $\B_m(G)$ (by Corollary \ref{aut_int_clique}). It is enough to show, that \par\medskip

($*$)\ \ \ {\it $\overline{\alpha}$ is induced by some $\beta\in\mathbf{\Gamma}_m$.} \par
\medskip

\noindent This gives that $f=\alpha\beta^{-1}$ fixes all intervals, that is $f\in\mathbf{Aut}_m(G)$. Hence we obtain that $\alpha=f\beta\in \mathbf{Aut}_m(G)\times \mathbf{\Gamma}_m$. 

\medskip

If $m=2$, then the graph $\B_2(G)$ is a cycle with three vertices 
$$
\pr{1}{2}, \pr{2}{3}, \pr{1}{3}
$$
and its automorphism group is isomorphic to the symmetric group $S_3$. The group $\mathbf{\Gamma}_2$ is isomorphic to $S_3$ and it is easily seen that this group acts on the graph  $\B_2(G)$ faithfully.

\medskip

If $m>3$, the complement $\overline{\B}_m(G)$ of $\B_m(G)$ is isomorphic to the Kneser graph $KG_{m+1,2}$. It is well known that its group of automorphisms is isomorphic to the symmetric group $S_{m+1}$. The action of various automorphisms $\gamma_i$ on the set $S$ induces various non-trivial automorphisms of $\B_m(G)$ by (\ref{gami}), so the natural homomorphism of $\mathbf{\Gamma}_m$ into the group of automorphisms of $\B_m(G)$ is in fact an epimorphism. Hence the condition ($*$) is fulfilled.

\begin{center}
\begin{tikzpicture}[scale=1.0]
\node at (30:3/2.1) {\footnotesize $\bullet$};
\node at (90:3/2.1) {\footnotesize $\bullet$};
\node at (150:3/2.1) {\footnotesize $\bullet$};
\node at (210:3/2.1) {\footnotesize $\bullet$};
\node at (270:3/2.1) {\footnotesize $\bullet$};
\node at (330:3/2.1) {\footnotesize $\bullet$};

\node[right] at (30:3/2.1) {\tiny $\pr{2}{4}$};
\node[above] at (90:3/2.1) {\tiny $\pr{1}{2}$};
\node[left] at (150:3/2.1) {\tiny $\pr{2}{3}$};
\node[left] at (210:3/2.1) {\tiny $\pr{1}{3}$};
\node[below] at (270:3/2.1) {\tiny $\pr{3}{4}$};
\node[right] at (330:3/2.1) {\tiny $\pr{1}{4}$};

\draw (30:3/2.1)--(90:3/2.1)--(150:3/2.1)--(210:3/2.1)--(270:3/2.1)--(330:3/2.1)--(30:3/2.1);
\draw (90:3/2.1)--(210:3/2.1)--(330:3/2.1)--(90:3/2.1); \draw (30:3/2.1)--(150:3/2.1)--(270:3/2.1)--(30:3/2.1);
\node at (270:5/2.1) {\footnotesize $m=3$,\  \  {\rm the graph}\  $\mathscr{B}_3$};
\end{tikzpicture}\hskip1cm
\begin{tikzpicture}[scale=1.0]
\node at (-0.9,0.75) {\footnotesize $\bullet$};
\node at (0,0.75) {\footnotesize $\bullet$};
\node at (0.9,0.75) {\footnotesize $\bullet$};
\node at (0.9,-0.75) {\footnotesize $\bullet$};
\node at (0,-0.75) {\footnotesize $\bullet$};
\node at (-0.9,-0.75) {\footnotesize $\bullet$};

\node[above] at (-0.9,0.75) {\tiny $\pr{1}{2}$}; \node[below] at (-0.9,-0.75) {\tiny $\pr{3}{4}$};
\node[above] at (0,0.75) {\tiny $\pr{2}{3}$}; \node[below] at (0,-0.75) {\tiny $\pr{1}{4}$};
\node[above] at (0.9,0.75) {\tiny $\pr{2}{4}$}; \node[below] at (0.9,-0.75) {\tiny $\pr{1}{3}$};

\draw (-0.9,0.75)--(-0.9,-0.75);\draw (0,0.75)--(0,-0.75);\draw (0.9,0.75)--(0.9,-0.75);

\node at (0,-5/2.1) {\footnotesize $m=3$,\  \  {\rm the graph}\   $\overline{\mathscr{B}}_3$};
\end{tikzpicture}

{\bf Fig.\,10.}
\end{center}

The situation is a little different for $m=3$ (see {\bf Fig.\,10}); the group of automorphisms of $\B_3(G)$ is isomorphic to 
the standard wreath product of the symmetric group $S_3$ and  the cyclic group $C_2$, which is a semidirect product 
$$
\big(C_2\times C_2\times C_2\big)\rtimes S_3
$$ 
with action of $S_3$ by permutations of coordinates. Observe that the last group is isomorphic to the direct product
$S_4\times C_2$. Hence the homomorphism of $\mathbf{\Gamma_3}$ into the group of all automorphisms of $\B_m(G)$ mentioned in the previous paragraph is not an epimorphism. More precisely, $\mathbf{\Gamma}_3$ induces a subgroup of index $2$ in the group of automorphisms of $\B_3(G)$. However, not all automorphisms of $\B_3(G)$ are induced by automorphisms of $\G_3(G)$. For instance a function 
$\varphi: \B_3(G)\rightarrow \B_3(G) $ such that  $\pr{1}{3}^{\varphi}=\pr{2}{4}$ and $\pr{2}{4}^{\varphi}=\pr{1}{3}$, and fixing all other intervals is an automorphism of $\B_3(G)$  (see {\bf Fig.\,10})  which cannot be induced by an automorphism of $\G_3(G)$. This follows from the fact that for an element $x\in G$ such that $x\neq x^{-1}$ the set $\{\x_{[1,2)},\x_{[1,3)},\x_{[1,4)}\}$ is a clique in $\G_3(G)$. If $\alpha$ is an automorphism of $\G_3(G)$ such that $\overline{\alpha}=\varphi$, then $\{\x_{[1,2)}^{\alpha},\x_{[1,3)}^{\alpha},\x_{[1,4)}^{\alpha}\}$ is also a clique in $\G_3(G)$ but 
$\x_{[1,2)}^{\alpha}\in \pr{1}{2},\ \x_{[1,3)}^{\alpha}\in\pr{2}{4},\ \x_{[1,4)}^{\alpha}\in \pr{1}{4}$
and there are not $3$-cliques with vertices from these intervals. Therefore all automorphisms of $\G_3(G)$ induce on $\B_3(G)$ automorphisms forming a subgroup of index $2$ and then they come from $\mathbf{\Gamma}_3(G)$. This means that also in the case $m=3$ the condition ($*$) is fulfilled, provided the group $G$ has exponent bigger than $2$.

\medskip

\noindent (2) We have already noticed that $\mathbf{Aut}_m(G)$ normalizes $\mathbf{T}_m(G)$. It is also easily seen that for all $i$, $1\leqslant i\leqslant m$, and all $\g\in G^m$
$$\gamma_iT_{\g}\gamma_i=T_{\g^{\gamma_i}}.$$
Hence the group $\mathbf{Aut}_m(G)\times \mathbf{\Gamma}_m$ normalizes $\mathbf{T}_m$ and then 
\begin{equation}\label{aut_ab}
\mathbf{T}_m\rtimes \big(\mathbf{Aut}_m(G)\times \mathbf{\Gamma}_m\big)
\end{equation}
is a subgroup of $\mathbf{Aut}(\G_m(G))$.

Now let $\alpha\in\mathbf{Aut}(\G_m(G))$ and suppose that $\g\in G^m$ is such that
$\g^{\alpha}=e$. Then $\e^{T_\g\alpha}=\e$ and by Lemma \ref{stabilizer_of_e} $T_\g\alpha=f\gamma$ for some $f\in\mathbf{Aut}_m(G)$, $\gamma\in\mathbf{\Gamma}_m$. Thus
$$
\alpha=T_{\g^{-1}}f\gamma\in \mathbf{T}_m\rtimes \big(\mathbf{Aut}_m(G)\times \mathbf{\Gamma}_m\big).$$
\end{proof}


\begin{rem}
{\em (a)  Let $G=C_2=\langle x\rangle$ be a cyclic group of order $2$ and $m=3$. 
\begin{center}
\begin{tikzpicture}[scale=1.2]
\node at (0:3/2.1) {\footnotesize $\bullet$};
\node at (45:3/2.1) {\footnotesize $\bullet$};
\node at (90:3/2.1) {\footnotesize $\bullet$};
\node at (135:3/2.1) {\footnotesize $\bullet$};
\node at (180:3/2.1) {\footnotesize $\bullet$};
\node at (225:3/2.1) {\footnotesize $\bullet$};
\node at (270:3/2.1) {\footnotesize $\bullet$};
\node at (315:3/2.1) {\footnotesize $\bullet$};

\node[right] at (0:3/2.1) {\tiny $(x,x,x)$};
\node[right] at (45:3/2.1) {\tiny $(x,e,e)$};
\node[above] at (90:3/2.1) {\tiny $(x,e,x)$};
\node[left] at (135:3/2.1) {\tiny $(x,x,e)$};
\node[left] at (180:3/2.1) {\tiny $(e,x,e)$};
\node[left] at (225:3/2.1) {\tiny $(e,e,x)$};
\node[below] at (270:3/2.1) {\tiny $(e,e,e)$};
\node[right] at (315:3/2.1) {\tiny $(e,x,x)$};

\draw (0:3/2.1)--(45:3/2.1)--(90:3/2.1)--(135:3/2.1)--(180:3/2.1)--(225:3/2.1)--(270:3/2.1)--(315:3/2.1)--(0:3/2.1);
\draw (0:3/2.1)--(90:3/2.1)--(180:3/2.1)--(270:3/2.1)--(0:3/2.1);
\draw (225:3/2.1)--(315:3/2.1)--(45:3/2.1)--(135:3/2.1)--(225:3/2.1); 
\draw (0:3/2.1)--(135:3/2.1)--(270:3/2.1)--(45:3/2.1)--(180:3/2.1)--(315:3/2.1)--(90:3/2.1)--(225:3/2.1)--(0:3/2.1);

\node at (270:5/2.1) {\footnotesize The graph $\G_3(C_2)$};
\end{tikzpicture}\hskip1cm
\begin{tikzpicture}[scale=1.2]
\node at (-1.35,0.75) {\footnotesize $\bullet$};
\node at (-0.45,0.75) {\footnotesize $\bullet$};
\node at (0.45,0.75) {\footnotesize $\bullet$};
\node at (1.35,0.75) {\footnotesize $\bullet$};
\node at (1.35,-0.75) {\footnotesize $\bullet$};
\node at (0.45,-0.75) {\footnotesize $\bullet$};
\node at (-0.45,-0.75) {\footnotesize $\bullet$};
\node at (-1.35,-0.75) {\footnotesize $\bullet$};

\node[above] at (-1.35,0.75) {\tiny $(e,e,e)$}; 
\node[below] at (-1.35,-0.75) {\tiny $(x,e,x)$};
\node[above] at (-0.45,0.75) {\tiny $(e,e,x)$}; 
\node[below] at (-0.45,-0.75) {\tiny $(x,e,e)$};
\node[above] at (0.45,.75) {\tiny $(e,x,e)$}; 
\node[below] at (0.45,-0.75) {\tiny $(x,x,x)$};
\node[above] at (1.35,0.75) {\tiny $(e,x,x)$}; 
\node[below] at (1.35,-0.75) {\tiny $(x,x,e)$};

\draw (-1.35,0.75)--(-1.35,-0.75);
\draw (-0.45,0.75)--(-0.45,-0.75);
\draw (0.45,0.75)--(0.45,-0.75);
\draw (1.35,0.75)--(1.35,-0.75);

\node at (0,-5/2.1) {\footnotesize The graph $\overline{\G_3(C_2)}$};
\end{tikzpicture}

{\bf Fig.\,11.}
\end{center}

Then $\mathbf{Aut}(\G_3(C_2))$ is isomorphic to the standard wreath product of the symmetric group $S_4$ and $C_2$, which is a semidirect product 
	$$\big(C_2\times C_2\times C_2\times C_2\big)\rtimes S_4$$
	with action of $S_4$ by permutations of coordinates (see {\bf Fig.\,11}).

 Since $\mathbf{T}_3$ has order $2^3$, the group of automorphisms fixing $\e$ (which fixes also $(x,e,x)$, because this is the unique vertex not adjacent to $\e$) is ismorphic to the direct product $S_4\times C_2$, so it is such as the group of automorphisms of $\B_3(C_2)$. Therefore all automorphisms of $\B_3(C_2)$ are induced by automorphisms of $\G_3(C_2)$.

\medskip

\noindent (b) Let $G=C_3=\{e,x,x^2\}$ be a cyclic group of order $3$ and $m=2$. By Example \ref{m=2} the graph
$\overline{\G_2(C_3)}$ consists of three disjoint triangles, so the automorphism group $\mathbf{Aut}(\G_2(C_3))$ is isomorphic to the standard wreath product of two copies of the symmetric group $S_3$, which is a semidirect product
$$
\big(S_3\times S_3\times S_3\big)\rtimes S_3
$$
with action of $S_3$ by permutations of coordinates. 
}\end{rem}


By Proposition \ref{neighbours}(3) we see the crucial difference between graphs $\G_m(G)$
for abelian and non-abelian groups. It follows from this that $G$ is non-abelian if and only if there exist vertices with exactly one path of length $2$ connecting them. In other words, $G$ is non-abelian if and only if the square $A^2$ of the adjacency matrix $A$ of $\G_m(G)$ (where $m\geqslant 3$) has entries equal to $1$.


\begin{lem}\label{diff} Let $G$ be a group, $Z(G)$ its center and $m\geqslant 3$. Let $k,l$ be such that $1\leqslant k<l\leqslant m+1$, $l-k\geqslant 2$. Then 
\begin{enumerate}
\item if $x\notin Z(G)$, then for any $y\in G$ such that $xy\neq yx$ the element $\x_{[k,l)}$ is the unique element of the set $V(\e)\cap V(\g)$, where 
$$\g=\y_{[i,k)}(\y\x)_{[k,j)}\x_{[j,l)},\ (1\leqslant i<k)\ \ {\rm or}\ \ \g=\x_{[k,i)}(\y\x)_{[i,l)}\y_{[l,j)},\ (l<j\leqslant m+1).$$
\item if for some $\ga\in G^m\setminus\{\e\}$ either  $\x_{[k,k+1)}\in V(\e)\cap V(\ga)$ or $\x_{[1,m+1)}\in V(\e)\cap V(\ga)$, then  $|V(\e)\cap V(\ga)|>1.$ \vspace{2pt}
\end{enumerate}
\end{lem}


\begin{proof}
The first part follows immediately from the proof of Proposition \ref{neighbours}. For the proof of the second part note that if $\g=\y_{[i,j)}\x_{[k,k+1)}$ and $j<k$ or $k<i$, then
$\y_{[i,j)} $ and $\x_{[k,k+1)}$ commute and then $\y_{[i,j)}$ is also a neighbour of $\g$. If $i\leqslant k$ and $l\leqslant j$, then $\g= (\y\x^{-1}\y^{-1})_{[k,k+1)}\y_{[i,j)}$ which again shows that $\y_{[i,j)}$ is a neighbour of $\g$.
\end{proof}


It is easily seen that the automorphisms $\gamma_i$ ($i=1,\dots,m$) defined in Lemma \ref{abelian} restricted to the subgraphs $\I_m(x)$ are their automorphisms. So is the group $\mathbf{Aut}_m(G)\times \mathbf{\Gamma}_m$ restricted to the subgraph $\V_m(\e)$. But most of these functions are not automorphism of the graph $\G_m(G)$. 


\begin{cor}	
If $G$ is non-abelian and $m\geqslant 2$, then no $\gamma_i$, $i=1,\dots,m$, is an automorphism of the graph $\G_m(G)$. 	
\end{cor}	


\begin{proof}
Let $m>3$ and $x\notin Z(G)$. Then by (\ref{gami})  $\x_{[i,i+1)}^{\gamma_{i-1}} = \x_{[i-1,i+1)}$ and $\x_{[i,i+1)}^{\gamma_{i+1}} = \x_{[i,i+2)}$ which is not possible by Lemma \ref{diff}. If $m=2$, then and $x,\ y$ are non-commuting elements of $G$, then $(x,y)\sim (x,e)$ but $(x,y)^{\gamma_1}=(x^{-1}y,y)\not\sim (x,e)=(x,e)^{\gamma_1}$. Similarly $(x,y)\sim (e,x^{-1}y)$ and $(x,y)^{\gamma_2}=(x,xy^{-1})\not\sim (e,y^{-1}x)=(e,x^{-1}y)^{\gamma_2}$.
\end{proof}


\begin{lem}\label{dihedral}
 Let  $G$ be a group and  $m\geqslant 2$. Then
\begin{enumerate}
\item the map $\tau\colon G^m\to G^m$ given by $$(g_1,g_2,\dots,g_m)^\tau=(g_m,g_{m-1},\dots,g_2,g_1)$$ is an automorphism  of order two of  $\G_m(G)$,
\item The map $\omega\colon G^m\to G^m$ given by
$$(g_1,g_2,\dots,g_m)^\omega=(g_1^{-1}g_2,g_1^{-1}g_3,\dots,g_1^{-1}g_m,g_1^{-1})$$ is an automorphism of order 
$m+1$ of  $\G_m(G)$,
\item  the subgroup  $\mathbf{\Delta}_m=\langle\tau,\omega\rangle$ of $\mathbf{Aut}(\G_m(G))$  is isomorphic to the dihedral group $D_{m+1}$.
\end{enumerate}
\end{lem}


\begin{proof}

\noindent (1) Let $\g,\h\in G^m$ be such that $\g\sim\h$. Take $x\in G^\times$ and integers $k<s$ such that 
$\h=\x_{[k,s)}\cdot\g$. It is clear that
the element ${\x_{[k,s)}}^\tau= \x_{[m+2-s,m+2-k)}$ satisfies $\h^\tau={\x_{[k,s)}}^\tau\g^\tau$. Therefore $\g^\tau\sim\h^\tau.$

\medskip

\noindent (2) Using induction we will prove that for $\g=(g_1,g_2,\dots,g_m)$ and $k=1,2,...,m$ 
\begin{equation}\label{it}
\g^{\omega^k}=(g_k^{-1}g_{k+1},g_k^{-1}g_{k+2},\dots,g_k^{-1}g_{m},g_k^{-1},g_k^{-1}g_{1},\dots,g_k^{-1}g_{k-1}).
\end{equation}
The case $k=1$  is clear by the definition of $\omega$ (we assume that $g_0=e$). Let $k\geqslant 1$. Assuming 
(\ref{it}) we obtain
$$
\begin{array}{rl}
\g^{\omega^{k+1}} = & (g_k^{-1}g_{k+1},g_k^{-1}g_{k+2},\dots,g_k^{-1}g_{m},g_k^{-1},g_k^{-1}g_{1},\dots,g_k^{-1}g_{k-1})^\omega \\ 
=  & (g_{k+1}^{-1}g_kg_k^{-1}g_{k+2},g_{k+1}^{-1}g_kg_k^{-1}g_{k+3},\dots,g_{k+1}^{-1}g_kg_k^{-1}g_{m},g_{k+1}^{-1}g_kg_k^{-1},\dots,g_{k+1}^{-1}g_k)\\ 
= &
(g_{k+1}^{-1}g_{k+2},g_{k+1}^{-1}g_{k+3},\dots,g_{k+1}^{-1}g_{m},g_{k+1}^{-1},\dots,g_{k+1}^{-1}g_{k})
\end{array}
$$
concluding the proof of (\ref{it}). As a result  $\g^{\omega^m}=(g_m^{-1},g_m^{-1}g_1,g_m^{-1}g_2,\dots, g_m^{-1}g_{m-1})$ and
$\g^{\omega^{m+1}}=\g$. Therefore $\omega^{m+1}=\mathbf{id}$.

\medskip

We now show that $\omega$ is a graph automorphism. Let $\g,\h\in G^m$ be such that  $\g\sim\h$. Take $x\in G^\times$  and integers $k<s$ such that $\h=\x_{[k,s)}\cdot\g$.

\medskip

\noindent{\bf Case 1:} $k=1$, $s=m+1$.  Then $h_1^{-1}=g_1^{-1}x^{-1}$ and
$$
\begin{array}{rl}
\h^\omega & = (h_1^{-1}h_2,h_1^{-1}h_3,\dots,h_1^{-1}h_m,h_1^{-1})
 =(g_1^{-1}x^{-1}xg_2,g_1^{-1}x^{-1}xg_3,\dots,g_1^{-1}x^{-1}xg_m, g_1^{-1}x^{-1})\\ 
& =(g_1^{-1}g_2,g_1^{-1}g_3,\dots,g_1^{-1}g_m, g_1^{-1}x^{-1})=(g_1^{-1}g_2,g_1^{-1}g_3,\dots,g_1^{-1}g_m, x^{g_1} g_1^{-1}).
\end{array}
$$
Therefore  $\h^\omega=(\x^g)_{[m,m+1)}\cdot\g^\omega$. Consequently $\g^\omega\sim \h^\omega$.

\bigskip

\noindent{\bf Case 2:} $k=1$ and  $s<m+1$. In this case also $h_1^{-1}=g_1^{-1}x^{-1}$ and
$$
\begin{array}{rl}
\h^\omega = & (h_1^{-1}h_2,h_1^{-1}h_3,\dots,h_1^{-1}h_{s-1},h_1^{-1}h_{s}, \dots,h_1^{-1}h_m,h_1^{-1})\\ 
= &
(g_1^{-1}x^{-1}xg_2,g_1^{-1}x^{-1}xg_3,\dots,g_1^{-1}x^{-1}xg_{s-1},g_1^{-1}x^{-1}g_{s}, \dots,g_1^{-1}x^{-1}g_m,g_1^{-1}x^{-1})\\  
= &
(g_1^{-1}g_2,g_1^{-1}g_3,\dots,g_1^{-1}g_{s-1},g_1^{-1}x^{-1}g_{s}, \dots,g_1^{-1}x^{-1}g_m,g_1^{-1}x^{-1})\\
= &
(g_1^{-1}g_2,g_1^{-1}g_3,\dots,g_1^{-1}g_{s-1},(x^{-1})^{g_1}g_1^{-1}g_{s}, \dots,(x^{-1})^{g_1}g_1^{-1}g_m,(x^{-1})^{g_1}g_1^{-1}).
\end{array}
$$
It means that $\h^\omega={(\x^{-1})^{g_1}}_{[s,m+1)}\cdot\g^\omega$. Therefore $\g^\omega\sim \h^\omega$.

\bigskip

\noindent{\bf Case 3:}  $1<k< s\leqslant m+1$. In this case $h_1^{-1}=g_1^{-1}$ and
$$
\begin{array}{rl}
\h^\omega= & (h_1^{-1}h_2,h_1^{-1}h_3,\dots,h_1^{-1}h_{k-1},h_1^{-1}h_k, \dots,h_1^{-1}h_{s-1},h_1^{-1}h_{s}, \dots,h_1^{-1}h_m,h_1^{-1})\\ 
= &
(g_1^{-1}g_2,g_1^{-1}g_3,\dots,g_1^{-1}g_{k-1},g_1^{-1}xg_k, \dots,g_1^{-1}xg_{s-1},g_1^{-1}g_{s}, \dots,g_1^{-1}g_m,g_1^{-1})\\
= &
(g_1^{-1}g_2,g_1^{-1}g_3,\dots,g_1^{-1}g_{k-1},x^{g_1}g_1^{-1} g_k, \dots,x^{g_1}g_1^{-1}g_{s-1},g_1^{-1}g_{s}, \dots,g_1^{-1}g_m,g_1^{-1}).
\end{array}
$$
Therefore  $\h^\omega=({\x}^{g_1})_{[k,s)}\cdot\g^\omega$ and  hence $\g^\omega\sim \h^\omega$. 
As a result $\omega$ is a graph automorphism of $\mathscr{G}_m(G)$.

\bigskip
\noindent 3. By (\ref{it}) for $k=m$ it follows that for any $\g\in G^m$
$$
\begin{array}{rl}
\g^{\omega^m\tau}= & (g_m^{-1},g_m^{-1}g_1,g_m^{-1}g_2,\dots, g_m^{-1}g_{m-1})^\tau
=(g_m^{-1}g_{m-1},g_m^{-1}g_{m-2},\dots,g_m^{-1}g_1,g_m^{-1})\\ 
= &
(g_m,g_{m-1},\dots,g_2,g_1)^\omega=\g^{\tau\omega}.
\end{array}
$$
Thus $\omega^{-1}\tau=\omega^m\tau=\tau\omega$ and  therefore automorphisms $\tau$ and $\omega$  generate the dihedral group $D_{m+1}$. 
\end{proof}


\begin{lem}\label{omega_n_normalize}
	If $G$ is a non-abelian group, then
\begin{enumerate}
 	\item the automorphism $\omega$ of the graph $\G_m(G)$ is not an automorphism of the group $G^m$;
 	\item for an arbitrary $\g=(g_1,\dots,g_m)\in G^m$
 	$$\omega^{-1}T_{\g}\omega=T_{\g^{\omega}}f_{g_1},$$
 	where $f_{g_1}$ is the inner automorphism of $G$ induced by conjugation by $g_1$:
 	$$(x_1,\ldots,x_m)^{f_{g_1}}=(g_1^{-1}x_1g_1^{-1},\dots,g_1^{-1}x_mg_1).$$
In particular, $\omega$ does not normalize $\mathbf{T}_m(G)$.
 \end{enumerate} 
\end{lem}


\begin{proof}
	(1) For $\g=(g_1,g_2,\dots,g_{m-1},g_m)$ and $\h=(h_1,h_2,\dots,h_{m-1},h_m)$ we have
	$$
	\begin{array}{rcl}
	(\g\h)^{\omega} & = & (g_1h_1,g_2h_2,\dots,g_{m-1}h_{m-1},g_mh_m)^{\omega} = \\
	&& (h_1^{-1}g_1^{-1}g_2h_2,h_1^{-1}g_1^{-1}g_3h_3,\dots,h_1^{-1}g_1^{-1}g_{m-1}h_{m-1},h_1^{-1}g_1^{-1})\\[4pt]

	\g^{\omega}\h^{\omega} & = & (g_1^{-1}g_2,g_1^{-1}g_3,\dots,g_1^{-1}g_{m-1},g_1^{-1})\times \\
	&& (h_1^{-1}h_2,h_1^{-1}h_3,\dots,h_1^{-1}h_{m-1},h_1^{-1}) = \\
	&& (g_1^{-1}g_2h_1^{-1}h_2,g_1^{-1}g_3h_1^{-1}h_3,\dots,g_1^{-1}g_{m-1}h_1^{-1}h_{m-1},g_1^{-1}h_1^{-1})\\
		
	\end{array}
	$$
So, if $g_1$ and $h_1$ do not commute, $(\g\h)^{\omega}\neq\g^{\omega}\h^{\omega}$.

\medskip

(2) By (\ref{it})
$$
\begin{array}{rcl}
(x_1,x_2,\ldots,x_m)^{\omega^{-1}T_{\g}\omega} & = & (x_m^{-1},x_m^{-1}x_1,x_m^{-1}x_2,\ldots,x_m^{-1}x_{m-2},x_m^{-1}x_{m-1})^{T_{\g}\omega}\\[4pt]
& = & (x_m^{-1}g_1,x_m^{-1}x_1g_2,x_m^{-1}x_2g_3,\ldots,x_m^{-1}x_{m-2}g_{m-1},x_m^{-1}x_{m-1}g_m)^{\omega}\\[4pt]
& = & (g_1^{-1}x_1g_2,g_1^{-1}x_2g_3,g_1^{-1}x_3g_4,\ldots,g_1^{-1}x_{m-1}g_{m},g_1^{-1}x_m)\\[4pt]
& = & (x_1^{g_1}(g_1^{-1}g_2),x_2^{g_1}(g_1^{-1}g_3),x_3^{g_1}(g_1^{-1}g_4),\ldots,x_{m-1}^{g_1}(g_1^{-1}g_{m}),x_m^{g_1}g_1^{-1})\\[4pt]
& = & (x_1,x_2,x_3,\ldots,x_{m-1},x_m)^{f_{g_1}T_{\g^{\omega}}},\\
\end{array}
$$

\end{proof}
\bigskip


\begin{lem}\label{omega_fixed_points}
Let $G$ be a non-abelian group. Then

\begin{enumerate}
	\item every fixed point of $\omega$ has the form $(g,g^2,\dots,g^m)$, where $g\in G$ is an element of order dividing $m+1$. In particular, if ${\rm gcd}(m+1,|G|)=1$, then $\e$ is the unique fixed point of $\omega$ and it is the unique fixed point of $\mathbf{\Delta}_{m}=\langle\omega,\ \tau\rangle$;
	\item  for every $x\in G^{\times}$ the subgraph $\I_m(x)$ is invariant under action of $\mathbf{\Delta}_m$. 

 	\end{enumerate} 	
\end{lem}


\begin{proof}
	(1) If $(g_1,g_2,\dots,g_{m-1},g_m)^{\omega}=(g_1,g_2,\dots,g_{m-1},g_m)$, then by the definition of $\omega$, 
$$g_1=g_1^{-1}g_2,\ g_2=g_1^{-1}g_3,\ \dots,\ g_{m-1}=g_1^{-1}g_m,\ g_m=g_1^{-1}. $$
Hence for $k=1,2,\dots,m$ we have $g_k=g_1^k$. In particular $g_1^{-1}=g_m=g_1^m$, hence $g_1^{m+1}=e.$	
\medskip

(2) Let $x\in G^{\times}$ and $1\leqslant k<l\leqslant m+1$. Then
$$ \begin{array}{l}
\{\x_{[k,l)},\, \x_{[k-1,l-1)},\, \dots,\, \x_{[1,l-k+1)},\, \x_{[l-k,m+1)}^{-1},\, \x_{[l-k-1,m)}^{-1},\,\dots,\\[6pt] 
\x_{[1,m-l+k+2)}^{-1},\x_{[m+k-l+1,m+1)},\ \dots, \x_{[k+1,l+1)}\}
\end{array}
$$
is an orbit of $\x_{[k,l)}$ under action of $\langle \omega\rangle$. Moreover, $\x_{[i,j)}^{\tau}=\x_{[m-j+1,m-i+1)}$, so if $\x_{[i,j)}$ belongs to this orbit, then $\x_{[i,j)}^{\tau}$ as well. By the definition of the graph $\I_m(x)$ the orbit is contained in the set of verices of $\I_m(x)$.
\end{proof}


\begin{thm}\label{stabilizer_of_e_nonabel}
Let $G$ be a non-abelian group. Then 
\begin{enumerate}
\item the stabilizer of $\e\in G^m$ in the automorphism group $\mathbf{Aut}(\G_m(G))$ is equal to
$$\mathbf{Aut}_m(G)\times \mathbf{\Delta}_m\simeq \mathbf{Aut}(G)\times D_{m+1};$$
\item the group of all automorphisms of the graph $\G_m(G)$ is equal
$$\mathbf{Aut}(\G_m(G))=\big(\mathbf{T}_m\rtimes \mathbf{Aut}_m(G)\big)\rtimes \mathbf{\Delta}_m\simeq \big(G^m \rtimes \mathbf{Aut}(G)\big)\rtimes D_{m+1}.$$
\end{enumerate}
	
\end{thm}


\begin{proof}
\noindent (1) Similarly as in the abelian case, for all  $f\in\mathbf{Aut}_m(G)$, $f\omega=\omega f$ and $f\tau=\tau f$, hence
$$\langle \mathbf{Aut}_m(G), \mathbf{\Delta}_m\rangle =\mathbf{Aut}_m(G)\times \mathbf{\Delta}_m.$$
Also, as for abelian groups, it is enough to show, that if $\alpha\in \mathbf{Aut}(\G_m(G))$ fixes $\e$ and $\overline{\alpha}$ is the automorphism of $\B_m(G)$ induced by $\alpha$, then \par\medskip

($*$)\ \ \ {\it $\overline{\alpha}$ is induced by some $\beta\in\mathbf{\Delta}_m$.} \par

\medskip

If $m=2$, then the graph $\B_2(G)$ is a cycle with three vertices 
$$
\pr{1}{2}, \pr{2}{3}, \pr{1}{3}
$$
and its automorphism group is isomorphic to the symmetric group $S_3$. The group $\mathbf{\Delta}_2$ is isomorphic to $S_3$ and it is easily seen that it acts on the graph  $\B_2(G)$ faithfully.

\medskip

Now let $m>3$ and let
$$
\mathbf{B_1}(G)=\{ \pr{1}{2}, \pr{2}{3},\dots,\pr{m}{m+1},\pr{1}{m+1}\}
$$
and
$$
\mathbf{B}_2(G)=\{\pr{k}{l}\mid\, 1\leqslant k<l\leqslant m+1,\ 2\leqslant l-k < m \}.
$$
We claim that $\mathbf{B}_i(G)^\alpha=\mathbf{B}_i(G)$ for $i=1,2$. To this end, suppose that 
there exists an element $\pr{s}{s+1}\in\mathbf{B}_1(G)$ such that $\pr{s}{s+1}^{\alpha}=\pr{k}{l}$, where $2\leqslant l-k<m$.
By Lemma \ref{diff} one can take an element $\ga$ such that $\e$
and $\ga $ have only one common neighbour of the form $\y_{[k,l)}$, where $y\in G\setminus Z(G)$. Then 
$\y_{[k,l)}^{\alpha^{-1}}\in\pr{s}{s+1}$ the unique common element of $\e^{\alpha^{-1}}=\e$ and $\ga^{\alpha^{-1}}$ This contradicts the second part of Lemma \ref{diff}, and consequently proves the claim. Note that the subgraph of $\mathscr{B}_m(G)$  induced by $\mathbf{B_1}(G)$ is an $m+1$ element cycle, so its automorphism group is the dihedral group $D_{m+1}$. On the other hand it is easy to check that the automorphisms $\tau$ and $\omega$ of $\G_m(G)$ from Lemma \ref{dihedral} act on  the cycle $\mathbf{B_1}(G)$ according to:
$$
\pr{s+1}{s+2}^\omega =\pr{s}{s+1},\  \  \pr{1}{2}^\omega=\pr{1}{m+1}\  \  \pr{1}{m+1}^\omega=\pr{m}{m+1}
$$
and
$$
\pr{s}{s+1}^\tau=\pr{m+1-s}{m+2-s} \  \  \  \pr{1}{m+1}^\tau=\pr{1}{m+1}.
$$
Thus it is clear that the  automorphisms  determined  by $\tau $ and $\omega$ generate the full group of automorphisms of the cycle $\mathbf{B_1}(G)$. This proves that one can  find $\delta\in \mathbf{Aut}(\G_m(G))$ such that
$\alpha\delta^{-1}$ fixes all intervals from $\mathbf{B}_1(G)$.

It remains to show that if an automorphism $\varphi$ of $\G_m(G)$ fixes each interval from $\mathbf{B}_1(G)$, then it fixes also each interval from $\mathbf{B}_2(G)$. If $m\geqslant 4$,  $1<k<l<m+1$, then the intervals $\pr{k}{l}$ from $\mathbf{B}_2(G)$ have exactly four neighbours in $\mathbf{B}_1(G)$. There are: $\pr{k-1}{k}$, $\pr{k}{k+1}$, $\pr{l-1}{l}$, $\pr{l}{l+1}$. Thus if the last four intervals are fixed by $\varphi$, then $\pr{k}{l}$ also must be fixed by $\varphi$ as the unique
common neighbour of these four intervals. By the same reason $\pr{1}{k}$ (resp.$\pr{l}{m+1}$) must be fixed by $\varphi$ as the unique common neighbour of three fixed  intervals: $\pr{1}{2}$,  $\pr{k-1}{k}$, $\pr{k}{k+1}$ (resp. $\pr{l-1}{l}$,  $\pr{l}{l+1}$, $\pr{m}{m+1}$). 
\medskip

This argument does not work when $m=3$ because there exists the unique automorphism $\varphi$ of $\B_3(G)$ which is trivial on $\mathbf{B}_1(G)=\{\pr{1}{2},\ \pr{2}{3},\ \pr{3}{4},\ \pr{1}{4}\}$ and non-trivial on $\mathbf{B}_2(G)$. This is the automorphism as defined above for the abelian case when $m=3$. But similar arguments as there show that this automorphism is not induced by an automorphism of $\G_3(G)$.

\medskip

\noindent (2) It is clear that $\mathbf{Aut}(\G_m(G))$ is a product of the groups $\mathbf{T}_m$  and $\mathbf{Aut}_m(G)\times \mathbf{\Delta}_m $. The subgroup $\mathbf{T}_m$ is normalized by $\mathbf{Aut}_m(G)$, so we have a semidirect product $\mathbf{T}_m\rtimes \mathbf{Aut}_m(G)$ which in turn is normalized by $\mathbf{\Delta}_m$. Therefore $\mathbf{Aut}(\G_m(G))$ is a two step semidirect product
$$\big(\mathbf{T}_m\rtimes \mathbf{Aut}_m(G)\big)\rtimes \mathbf{\Delta}_m.$$
\end{proof}

\begin{rem}
	{\em Notice that the subgraph of $\B_m(G)$  with the vertex set $\mathbf{B}_2(G)$ is the complement of the stable Kneser graph $SG_{m+1,2}$, whose group of automorphisms is described in \cite{B}.}
\end{rem}

We can now prove the main result of this paper.
\begin{thm}\label{izo}
Let $G$ and $H$ be groups and $m>1$. Then the graphs $\G_m(G)$ and $\G_m(H)$ are isomorphic if and only if the groups $G$ and $H$ are isomorphic.
\end{thm}


\begin{proof} It is obvious that isomorphic groups have isomorphic Cayley graphs. Suppose that $F\colon\G_m(G)\to \G_m(H)$ is a graph isomorphism. Observe that if $G$ is an elementary abelian $2$-group, then so is $H$ by Proposition \ref{V(g)=V(h)}. So we assume that neither $G$ nor $H$ is such a group.
In light of Theorem \ref{morphism} it is enough to show that there exists an automorphism $\varphi \in \mathbf{Aut}(\G_m(H))$ such that $\widehat{F}=\varphi\circ F$ is a homogeneous isomorphism. Notice that if $F(\e_G)= \ga\neq \e_H$, then the composition $\mathbf{T}_{\ga^{-1}}\circ F$ is an isomorphism sending the element $\e_G$ onto $\e_H$. Thus we may assume that $F(\e_G)=\e_H$ and then also $F(V(\e_G))=V(\e_H)$. 

Consider the graphs of intervals $\mathscr{B}_m(G)$ and $\mathscr{B}_m(H).$ Since the sets $\{\e_G\}\cup \pr{k}{l}$ and $\{\e_G\}\cup\prh{k}{l}$ form maximum $|G|$-element cliques around $\e_G$ and $\e_H$ (it is clear that $|G|=|H|$), the isomorphism $F$ induces an isomorphism of graphs $\mathscr{B}_m(G)$ and $\mathscr{B}_m(H)$. Indeed, if $m\neq |G|-1$, then  the maximum $|G|$-element cliques  containing $\e_G$ and $\e_H$ are of interval type. Thus $F$ induces a bijection $\widehat{F}$ between vertices of $\mathscr{B}_m(G)$ and $\mathscr{B}_m(H)$. If $m=|G|-1$ our graphs have also dispersed $|G|$-element cliques. It is seen, by   Corollary \ref{aut_int_clique},  that each isomorphism preserves the type of a maximum clique. So in this case,  $F$ induces also a bijection $\widehat{F}$ between vertices of $\mathscr{B}_m(G)$ and $\mathscr{B}_m(H)$. It is obvious that  $\widehat{F}$ is a graph isomorphism.

It follows from Lemma \ref{diff} that $G$ is abelian if and only if $H$ is abelian. Consider the map
$\widehat{\psi}\colon \mathscr{B}_m(H)\to \mathscr{B}_m(H)$ given by $\widehat{\psi}(\prh{k}{l})=\widehat{F}(\pr{k}{l})$. Since $\widehat{F}$ is a graph isomorphism, $\widehat{\psi}$ is a graph automorphism. Furthermore, using the same arguments as in the proofs of Theorem \ref{stabilizer_of_e} and Theorem \ref{stabilizer_of_e_nonabel}, in all special cases, one can show that $\widehat{\psi}$ is induced by an automorphism $\psi$ of $\G_m(H)$ fixing $\e_H$. 
Hence the map $\psi^{-1}\circ F$ is a homogeneous isomorphism of $\G_m(G)$ onto $\G_m(H)$, that is by Theorem \ref{morphism}
$$(\psi^{-1}\circ F)(g_1,g_2,\ldots,g_m)=(f(g_1),f(g_2),\ldots,f(g_m))$$
for some isomorphism $f:G\rightarrow H$. Consequently, the groups $G$ and $H$ are isomorphic.
\end{proof}


\begin{thebibliography}{PGX}

\bibitem{BI} G.M. Bergman, I.M. Isaacs, \textit{Rings with fixed-point-free group actions}, Proc. London Math. Soc. \textbf{27} (1) (1973), 69--87. 


\bibitem{B} B. Brown, \textit{Symmetries of the stable Kneser graphs}, Advances in Applied Mathematics, \textbf{45} (2010), 12--14.

\bibitem{GR} C. Godsil, G. Royle, \textit{Algebraic Graph Theory}, Grad. Texts in Math., vol. \textbf{207}, Springer-Verlag, New York, 2001.

\bibitem{I} I. M. Isaacs, {\it Character Theory of Finite Groups}, Academic Press, New York, 1976.

\bibitem{P} D. S. Passman, \textit{Fixed rings and integrality}, J. Algebra \textbf{68} (1981), 510--519.

\bibitem{P1}  D. S. Passman, \textit{Infinite crossed products}, Pure and Applied Mathematics, \textbf{135} Academic Press, Inc., Boston, MA, 1989.

\bibitem{EP} E.R. Puczy{\l}owski, \textit{On fixed rings of automorphisms}, Proc. Amer. Math. Soc., \textbf{90}, (1984), 517-518.

\bibitem{Q}  D. Quinn, \textit{Integrality over fixed rings}, J. London Math. Soc. (2) \textbf{40}, (1989) 206--214.

\bibitem{Z} P. H. Zieschang, {\it Cayley graphs of finite groups}, J. Algebra, \textbf{118}, (1988), 447--454. 

\end{thebibliography}
\end{document}